\documentclass[reqno]{amsart}
\usepackage{amssymb}
\usepackage{graphicx}

\usepackage[usenames, dvipsnames]{color}
\usepackage{verbatim}
\usepackage{mathrsfs}
\usepackage{bm}
\usepackage{cite}

\numberwithin{equation}{section}

\newtheorem{theorem}{Theorem}[section]
\newtheorem{corollary}[theorem]{Corollary}
\newtheorem{lemma}[theorem]{Lemma}
\newtheorem{prop}[theorem]{Proposition}

\theoremstyle{definition}
\newtheorem{remark}[theorem]{Remark}

\theoremstyle{definition}

\theoremstyle{definition}

\makeatletter
\def\dashint{\operatorname%
{\,\,\text{\bf-}\kern-.98em\DOTSI\intop\ilimits@\!\!}}
\makeatother

\def\\det{\text{\det}}

\def\.5{\frac{1}{2}}

\newcommand{\RN}[1]{%
  \textup{\uppercase\expandafter{\romannumeral#1}}%
}

\renewcommand{\epsilon}{\varepsilon}

\newcounter{marnote}

\begin{document}

\title[Stress concentration factors]{Stress concentration factors for the Stokes flow with two nearly touching rigid particles}


\author[Z.W. Zhao]{Zhiwen Zhao}

\address[Z.W. Zhao]{Beijing Computational Science Research Center, Beijing 100193, China.}
\email{zwzhao365@163.com}


\date{\today} 


\maketitle
\begin{abstract}
In this paper, a mathematical model of two adjacent rigid particles immersed into a viscous incompressible fluid is considered. The main feature of the flow is that the Cauchy stress tensor consisting of the strain tensor and the pressure will appear blow-up as the distance between these two particles tends to zero. For the purpose of making clear this high concentration, a family of unified stress concentration factors are precisely captured in all dimensions, which determine whether the Cauchy stress tensor will blow up or not. As a direct application, we establish optimal gradient estimates and asymptotics of the Cauchy stress tensor for Stokes flow, which indicate that its maximal singularity comes from the pressure.

\end{abstract}


\section{Introduction}

The study on suspensions of rigid particles in a viscous incompressible fluid has aroused widespread attention and interest due to its extensive applications in biotechnology, medical science, chemical engineering, food processing, environmental geophysics and composites manufacturing, see e.g. \cite{G1994,S2011,PT2009}. In particular, the two-dimensional model of suspensions can be used to describe biological thin films, which are now intensively investigated in virtue of their application in pharmaceutical industry. Quantitative analysis on the Cauchy stress tensor of such suspensions becomes a key issue for both theory and practical applications.

In this paper, we consider a mathematical model of two close-to-touching rigid particles immersed into the incompressible Stokes flow with low Reynolds number, where the inertial forces can be neglected. The key feature of the flow is that the interparticle closeness will cause high concentration of the Cauchy stress tensor. In order to clearly describe the problem, we first formulate our domain. Consider a bounded domain $D$ with $C^{2,\gamma}$ boundary in $\mathbb{R}^{n}$, whose interior is filled with an incompressible viscous fluid of viscosity $\mu$ and contains two rigid particles $D_{1}^{\ast}$ and $D_{2}^{\ast}$ with $C^{2,\gamma}$ boundaries, where $n\geq2$, $0<\gamma<1$ and $\mu>0$. Assume that these two particles keep far away from the external boundary $\partial D$ and touch only at one point. By applying an appropriate coordinate system, we have
\begin{align*}
\partial D^{\ast}_{1}\cap\partial D^{\ast}_{2}=\{0\},\quad D_{i}^{\ast}\subset\{(x',x_{n})\in\mathbb{R}^{n}|\,(-1)^{i-1}x_{n}>0\},\quad i=1,2.
\end{align*}
Here and below, all ($n-1$)-dimensional variables and domains are denoted by adding superscript prime, for example, $x'$ and $B'$. Let $D_{1}^{\ast}$ and $D_{2}^{\ast}$ move up and down by a small positive constant $\varepsilon$ along $x_{n}$-axis, respectively. To be precise,
\begin{align*}
D_{i}^{\varepsilon}:=D_{i}^{\ast}+(0',(-1)^{i-1}\varepsilon),\quad i=1,2.
\end{align*}
For simplicity, we drop superscripts and set
\begin{align*}
D_{i}:=D_{i}^{\varepsilon},\;i=1,2,\quad\Omega:=D\setminus\overline{D_{1}\cup D_{2}},\quad\mathrm{and}\quad\Omega^{\ast}:=D\setminus\overline{D_{1}^{\ast}\cup D^{\ast}_{2}}.
\end{align*}

Denote by
\begin{align}\label{LAK01}
\boldsymbol{\Psi}:=\{\boldsymbol{\psi}\in C^1(\mathbb{R}^{n}; \mathbb{R}^{n})\ |\ \nabla\boldsymbol{\psi}+(\nabla\boldsymbol{\psi})^T=0\}
\end{align}
the linear space of rigid displacement, whose base is written as
\begin{align}\label{OPP}
\{\boldsymbol{\psi}_{\alpha}\}_{\alpha=1}^{\frac{n(n+1)}{2}}:=\left\{\,\mathbf{e}_{i},\,x_{k}\mathbf{e}_{j}-x_{j}\mathbf{e}_{k}\,\big|\,1\leq i\leq n,\,1\leq j<k\leq n\,\right\},
\end{align}
where $\{\mathbf{e}_{1},...,\mathbf{e}_{n}\}$ represents the standard basis of $\mathbb{R}^{n}$. For the convenience of later use, the order of these elements is prescribed as follows: $\boldsymbol{\psi}_{\alpha}=\mathbf{e}_{\alpha}$ in the case when $\alpha=1,2,...,n$; $\boldsymbol{\psi}_{\alpha}=x_{n}\mathbf{e}_{\alpha-n}-x_{\alpha-n}\mathbf{e}_{n}$ in the case when $\alpha=n+1,...,2n-1$; in the case of $\alpha=2n,...,\frac{n(n+1)}{2}\,(n\geq3)$, there exist two integers $1\leq i_{\alpha}<j_{\alpha}<n$ such that
$\boldsymbol{\psi}_{\alpha}=x_{j_{\alpha}}\mathbf{e}_{x_{i_{\alpha}}}-x_{i_{\alpha}}\mathbf{e}_{x_{j_{\alpha}}}$.

Let $\mathbf{u}=(u^{1},u^{2},...,u^{n})^{T}:D\rightarrow\mathbb{R}^{n}$ and $p:D\rightarrow\mathbb{R}$ represent the fluid velocity and pressure, respectively. The Stokes flow considered in this paper can be described as follows:
\begin{align}\label{La.002}
\begin{cases}
\nabla\cdot\sigma[\mathbf{u},p]=0,\;\nabla\cdot \mathbf{u}=0,&\hbox{in}\ \Omega,\\
\mathbf{u}|_{+}=\mathbf{u}|_{-},&\hbox{on}\ \partial{D}_{i},\,i=1,2,\\
e(\mathbf{u})=0,&\hbox{in}~D_{i},\,i=1,2,\\
\int_{\partial{D}_{i}}\boldsymbol{\psi}_{\alpha}\cdot\sigma[\mathbf{u},p]\nu=0,&i=1,2,\,\alpha=1,2,...,\frac{n(n+1)}{2},\\
\mathbf{u}=\boldsymbol{\varphi},&\hbox{on}\ \partial{D},
\end{cases}
\end{align}
where $\boldsymbol{\varphi}=(\varphi^{1},\varphi^{2},...,\varphi^{n})^{T}\in C^{2}(\partial D;\mathbb{R}^{n})$ is the given velocity field satisfying the compatibility condition \eqref{COM001} below, $e(\mathbf{u})=\frac{1}{2}(\nabla \mathbf{u}+(\nabla \mathbf{u})^{T})$ denotes the strain tensor, $\sigma[\mathbf{u},p]:=2\mu e(\mathbf{u})-p\mathbb{I}$ represents the Cauchy stress tensor with $\mathbb{I}$ being identity matrix, and $\nu$ is the unit outer normal to the domain. In view of the incompressible condition $\nabla\cdot\mathbf{u}=0$, it follows from a straightforward computation that $\nabla\cdot\sigma[\mathbf{u},p]=\mu\Delta\mathbf{u}-\nabla p$. Combining the incompressible condition $\nabla\cdot\mathbf{u}=0$ and Gauss theorem, we know that $\varphi$ must verify the following compatibility condition:
\begin{align}\label{COM001}
\int_{\partial D}\boldsymbol{\varphi}\cdot\nu=0.
\end{align}
The existence and uniqueness of a weak solution to \eqref{La.002} satisfying condition \eqref{COM001} can be demonstrated by using an integral variational formulation and Riesz representation theorem, see the detailed proof in \cite{L1959} with minor modification. With regard to the corresponding regularity of solution to Stokes flow, it can be deduced directly from the general regularity theory established in \cite{ADN1964} and \cite{S1966}, because the Stokes flow is elliptic in the sense of Douglis-Nirenberg, see \cite{T1984}.

The mathematical and physical problem of interest is to study the blow-up feature of the Cauchy stress tensor in the narrow region between inclusions. For this purpose, it needs to make clear the singular behavior of $e(\mathbf{u})$ and $p$ in terms of the interparticle distance $\varepsilon$, respectively. The first step is to carry out a linear decomposition for the solution $(\mathbf{u},p)$ of problem \eqref{La.002} by using the second and third lines of \eqref{La.002} as follows:
\begin{equation}\label{Decom}
\mathbf{u}=\sum^{2}_{i=1}\sum_{\alpha=1}^{\frac{n(n+1)}{2}}C_i^{\alpha}\mathbf{u}_{i}^{\alpha}+\mathbf{u}_{0},\quad\mathrm{and}\;p=\sum^{2}_{i=1}\sum_{\alpha=1}^{\frac{n(n+1)}{2}}C_i^{\alpha}p_{i}^{\alpha}+p_{0},\quad\mathrm{in}\;\Omega,
\end{equation}
where $C_{i}^{\alpha}$, $i=1,2,\,\alpha=1,2,...,\frac{n(n+1)}{2}$ are the free constants, to be determined by boundary integral condition in the fourth line of \eqref{La.002}, $\mathbf{u}_{0},\mathbf{u}_{i}^{\alpha}\in{C}^{2,\gamma}(\Omega;\mathbb{R}^n)$, $p_{0},p_{i}^{\alpha}\in C^{1,\gamma}(\Omega)$, $i=1,2$, $\alpha=1,2,...,\frac{n(n+1)}{2}$ are, respectively, the solutions of
\begin{equation}\label{qaz001}
\begin{cases}
\nabla\cdot\sigma[\mathbf{u}_{0},p_{0}]=0,\;\nabla\cdot\mathbf{u}_{0}=0&\mathrm{in}~\Omega,\\
\mathbf{u}_{0}=0,&\mathrm{on}~\partial{D}_{1}\cup\partial{D_{2}},\\
\mathbf{u}_{0}=\boldsymbol{\varphi},&\mathrm{on}~\partial{D},
\end{cases}
\end{equation}
and
\begin{equation}\label{qaz003az}
\begin{cases}
\nabla\cdot\sigma[\mathbf{u}^{\alpha}_{i},p_{i}^{\alpha}]=0,\;\nabla\cdot\mathbf{u}_{i}^{\alpha}=0,&\mathrm{in}~\Omega,\\
\mathbf{u}_{i}^{\alpha}=\boldsymbol{\psi}_{\alpha},&\mathrm{on}~\partial{D}_{i},~i=1,2,\\
\mathbf{u}_{i}^{\alpha}=0,&\mathrm{on}~\partial{D_{j}}\cup\partial{D},~j\neq i.
\end{cases}
\end{equation}
It is worth emphasizing that the objective of \eqref{Decom} is to extract these free constants $C_{i}^{\alpha}$, $i=1,2,\,\alpha=1,2,...,\frac{n(n+1)}{2}$ from the original problem \eqref{La.002} and make them be calculated by its fourth line. Define
\begin{align}\label{CTL001}
\mathbf{u}_{b}:=\sum_{\alpha=1}^{\frac{n(n+1)}{2}}C_{2}^\alpha(\mathbf{u}_{1}^\alpha+\mathbf{u}_{2}^\alpha)+\mathbf{u}_{0},\quad p_{b}:=\sum_{\alpha=1}^{\frac{n(n+1)}{2}}C_{2}^\alpha(p_{1}^\alpha+p_{2}^\alpha)+p_{0}.
\end{align}
In light of linearity, we see that $(\mathbf{u}_{b},p_{b})$ solves
\begin{align}\label{GLQ}
\begin{cases}
\nabla\cdot\sigma[\mathbf{u}_{b},p_{b}]=0,\;\nabla\cdot\mathbf{u}_{b}=0,&\mathrm{in}\;\Omega,\\
\mathbf{u}_{b}=\sum\limits^{\frac{n(n+1)}{2}}_{\alpha=1}C^{\alpha}_{2}\boldsymbol{\psi}_{\alpha},&\mathrm{on}\;\partial D_{1}\cup\partial D_{2},\\
\mathbf{u}_{b}=\boldsymbol{\varphi},&\mathrm{on}\;\partial D.
\end{cases}
\end{align}
From \eqref{Decom} and \eqref{CTL001}, we have
\begin{align}\label{Decom002}
\nabla\mathbf{u}=&\sum_{\alpha=1}^{\frac{n(n+1)}{2}}(C_{1}^\alpha-C_{2}^\alpha)\nabla\mathbf{u}_{1}^\alpha+\nabla \mathbf{u}_{b},\quad p=\sum_{\alpha=1}^{\frac{n(n+1)}{2}}(C_{1}^\alpha-C_{2}^\alpha)p_{1}^\alpha+p_{b},
\end{align}
which, in combination with the fourth line of \eqref{La.002}, shows that
\begin{align}\label{JGRO001}
\sum\limits_{\alpha=1}^{\frac{n(n+1)}{2}}(C_{1}^\alpha-C_{2}^{\alpha}) a_{11}^{\alpha\beta}=\mathcal{B}_{\beta}[\boldsymbol{\varphi}],\quad\beta=1,2,...,\frac{n(n+1)}{2},
\end{align}
where, for $\alpha,\beta=1,2,...,\frac{n(n+1)}{2}$,
\begin{align}\label{LGBC}
a_{11}^{\alpha\beta}:=-\int_{\partial{D}_{1}}\boldsymbol{\psi}_{\beta}\cdot\sigma[\mathbf{u}_{1}^{\alpha},p_{1}^{\alpha}]\nu,\quad \mathcal{B}_{\beta}[\boldsymbol{\varphi}]:=\int_{\partial D_{1}}\boldsymbol{\psi}_{\beta}\cdot\sigma[\mathbf{u}_{b},p_{b}]\nu.
\end{align}
Observe from \eqref{GLQ} that there is no potential difference on upper and lower boundaries of the narrow region between $D_{1}$ and $D_{2}$, which will lead to no blow-up of $\nabla\mathbf{u}_{b}$ and $p_{b}$. Then in view of \eqref{Decom002}, the problem is reduced to the establishments of the following asymptotics:
\begin{itemize}
{\it
\item[({\em a})] asymptotics of $\nabla \mathbf{u}_{1}^{\alpha},\,p_{1}^{\alpha}$, $\alpha=1,2,...,\frac{n(n+1)}{2}$;
\item[({\em b})] asymptotics of $C_{1}^{\alpha}-C_{2}^{\alpha}$, $\alpha=1,2,...,\frac{n(n+1)}{2}$.}
\end{itemize}
With regard to problem ({\em a}), we construct the explicit singular functions and then use the iterate technique developed in \cite{LX2022} to prove that these auxiliary functions are actually the leading terms corresponding to $\nabla \mathbf{u}_{1}^{\alpha}$ and $p_{1}^{\alpha}$, $\alpha=1,2,...,\frac{n(n+1)}{2}$. In order to solve problem ({\em b}) in all dimensions, it needs to accurately calculate each element of the coefficient matrix $(a_{11}^{\alpha\beta})_{\frac{d(d+1)}{2}\times\frac{d(d+1)}{2}}$ of \eqref{JGRO001}. For that purpose, a lot of technique, such as the maximum modulus principle, the rescale argument, the interpolation inequality and the standard interior and boundary estimates for the Stokes flow, will be used. Although we will present these results in a concise way, it actually contains quite complex computations, especially calculations for the off-diagonal elements in Lemma \ref{lemmabc} below. Until now, it remains to analyze the asymptotic behavior for $\mathcal{B}_{\beta}[\boldsymbol{\varphi}]$, $\beta=1,2,...,\frac{n(n+1)}{2}$. In light of the definitions of $\mathbf{u}_{b}$ and $p_{b}$, we see that the value of each $\mathcal{B}_{\beta}[\boldsymbol{\varphi}]$ will vary with the boundary data $\varphi$ and so are  $C_{1}^{\alpha}-C_{2}^{\alpha}$, $\alpha=1,2,...,\frac{n(n+1)}{2}$. Due to the free boundary value feature induced by the boundary data $\boldsymbol{\varphi}$, every $\mathcal{B}_{\beta}[\boldsymbol{\varphi}]$ is also called the $stress\; concentration\;factor\;or\;blow$-$up\;factor$.

As shown in Theorem 1.5 of \cite{LX2022}, the key to construction of the optimal lower bounds on $|\nabla\mathbf{u}|$ lies in finding the limit factor of $\mathcal{B}_{\beta}[\boldsymbol{\varphi}]$ as the distance $\varepsilon\rightarrow0$. For that purpose, Li and Xu \cite{LX2022} considered the symmetric domain and established the following convergence result in the presence of the boundary data of odd function (see Proposition 6.1 of \cite{LX2022}): for $\beta=1,2,...,\frac{n(n+1)}{2}$,
\begin{align}\label{KGA001}
\mathcal{B}_{\beta}[\boldsymbol{\varphi}]=\mathcal{B}_{\beta}^{\ast}[\varphi]+O(\rho_{n}(\varepsilon)),\quad\rho_{n}(\varepsilon)=
\begin{cases}
\sqrt{\varepsilon},&n=2,\\
|\ln\varepsilon|^{-1},&n=3,
\end{cases}
\end{align}
where $\mathcal{B}_{\beta}^{\ast}[\varphi]$ is defined by
\begin{align}\label{KGF001}
\quad\mathcal{B}_{\beta}^{\ast}[\varphi]:=\int_{\partial D_{1}^{\ast}}\boldsymbol{\psi}_{\beta}\cdot\sigma[\mathbf{u}_{b}^{\ast},p_{b}^{\ast}]\nu.
\end{align}
Here $(\mathbf{u}_{b}^{\ast},p_{b}^{\ast})$ satisfies
\begin{align}\label{LCRD001}
\begin{cases}
\nabla\cdot\sigma[\mathbf{u}_{b}^{\ast},p_{b}^{\ast}]=0,\;\nabla\cdot\mathbf{u}^{\ast}_{b}=0,&\mathrm{in}\;\Omega^{\ast},\\
\textbf{u}_{b}^{\ast}=\sum\limits^{\frac{n(n+1)}{2}}_{\alpha=1}C^{\alpha}_{\ast}\boldsymbol{\psi}_{\alpha},&\mathrm{on}\;(\partial D_{1}^{\ast}\cup\partial D_{2}^{\ast})\setminus\{0\},\\
\mathbf{u}_{b}^{\ast}=\boldsymbol{\varphi},&\mathrm{on}\;\partial D,
\end{cases}
\end{align}
where $C_{\ast}^{\alpha}$, $\alpha=1,2,...,\frac{n(n+1)}{2}$ are determined by
\begin{align}\label{LDCZ003}
\int_{\partial D_{1}^{\ast}\cup\partial D_{2}^{\ast}}\boldsymbol{\psi}_{\beta}\cdot\sigma[\mathbf{u}_{b}^{\ast},p_{b}^{\ast}]\nu=0,\quad\beta=1,2,...,\frac{n(n+1)}{2}.
\end{align}
It is worth pointing out that the shortcoming of the idea adopted in \cite{LX2022} is that it needs to impose some special symmetric condition on the domain and the parity condition on boundary data in order to capture these stress concentration factors. Moreover, according to their idea, it needs to use the infinitesimal differences of $C_{1}^{\alpha}-C_{2}^{\alpha}$ to establish the convergence of $C_{2}^{\alpha}\rightarrow C_{\ast}^{\alpha}$, as $\varepsilon\rightarrow0$, $\alpha=1,2,...,n$, which is the key to the establishment of \eqref{KGA001}. However, in the case of $n>3$, the differences of $C_{1}^{\alpha}-C_{2}^{\alpha}$,  $\alpha=1,2,...,n$ are of constant order but not the infinitely small quantity with respect to the distance $\varepsilon$, see \eqref{GMARZT001} below. So their idea cannot be used to deal with the case when $n>3$.

Inspired by previous work \cite{MZ2021} on the study of the stress concentration factors for the elasticity problem arising from composites, this paper also aims to get ride of those harsh symmetric conditions on the domain and the parity condition on boundary data and establish the convergence between $\mathcal{B}_{\beta}[\varphi]$ and $\mathcal{B}_{\beta}^{\ast}[\varphi]$ in the context of Stokes flow in any dimension. As a direct consequence, we present an asymptotic formula for the Cauchy stress tensor. This not only proves the optimal blow-up rate but also reveals that the greatest singularity of the Cauchy stress tensor arises from the pressure.

Before listing our major results, we further describe the domain. Suppose that there exists a small positive constant $R_{0}$ independent of $\varepsilon$ such that $\partial D_{1}$ and $\partial D_{2}$ near the origin are, respectively, formulated by two smooth functions $\pm(\varepsilon/2+h(x'))$ with
\begin{align}\label{CONVEX001}
h(x')=\kappa|x'|^{2},\quad|x'|\leq2R_{0},
\end{align}
where $\kappa>0$ is a $\varepsilon$-independent constant.  For $z'\in B'_{R_{0}}$ and $0<t\leq2R_{0}$, denote by
\begin{align*}
\Omega_{t}(z'):=&\left\{x\in \mathbb{R}^{n}~\big|~|x_{n}|<\varepsilon/2+h(x'),~|x'-z'|<t\right\}
\end{align*}
a narrow region between two particles. Remark that we don't require the symmetry of the whole domain $\Omega$, although condition \eqref{CONVEX001} implies that the thin gap $\Omega_{2R_{0}}$ is symmetric with respect to $x_{i}$, $i=1,2,...,n$.
To simplify the notation, we use the abbreviated notation $\Omega_{t}$ to denote $\Omega_{t}(0')$ with its upper and lower boundaries written as
$\Gamma^{\pm}_{t}:=\left\{x\in\mathbb{R}^{n}|\,x_{n}=\pm(\varepsilon/2+h(x')),\;|x'|<t\right\}$, respectively. Similarly, denote $\Omega_{t}^{\ast}:=\Omega_{t}|_{\varepsilon=0}$ in the touching case.

For $i,j=1,2,\,\alpha,\beta=1,2,...,\frac{n(n+1)}{2}$, define
\begin{align}\label{LMZR}
a_{ij}^{\ast\alpha\beta}=\int_{\Omega^{\ast}}(2\mu e(\mathbf{u}_{i}^{\ast\alpha}), e(\mathbf{u}_j^{\ast\beta})),\quad b_i^{\ast\alpha}=-\int_{\partial D}\mathbf{u}_{0}^{\ast}\cdot\sigma[\mathbf{u}_{i}^{\ast\alpha},p_{i}^{\ast\alpha}],
\end{align}
where $\mathbf{u}_{0}^{\ast},\mathbf{u}_{i}^{\ast\alpha}\in{C}^{2,\gamma}(\Omega^{\ast};\mathbb{R}^n)$, $p_{0}^{\ast},p_{i}^{\ast\alpha}\in C^{1,\gamma}(\Omega^{\ast})$, $i=1,2$, $\alpha=1,2,...,\frac{n(n+1)}{2}$, are the solutions of
\begin{equation}\label{ZG001}
\begin{cases}
\nabla\cdot\sigma[\mathbf{u}_{0}^{\ast},p_{0}^{\ast}]=0,\;\nabla\cdot\mathbf{u}_{0}^{\ast}=0&\mathrm{in}~\Omega^{\ast},\\
\mathbf{u}_{0}^{\ast}=0,&\mathrm{on}~(\partial{D}_{1}^{\ast}\cup\partial{D}_{2}^{\ast})\setminus\{0\},\\
\mathbf{u}_{0}^{\ast}=\boldsymbol{\varphi},&\mathrm{on}~\partial{D},
\end{cases}
\end{equation}
and
\begin{equation}\label{qaz001111}
\begin{cases}
\nabla\cdot\sigma[\mathbf{u}^{\ast\alpha}_{i},p_{i}^{\ast\alpha}]=0,\;\nabla\cdot\mathbf{u}_{i}^{\ast\alpha}=0,&\mathrm{in}~\Omega^{\ast},\\
\mathbf{u}_{i}^{\ast\alpha}=\boldsymbol{\psi}_{\alpha},&\mathrm{on}~\partial{D}_{i}^{\ast},~i=1,2,\\
\mathbf{u}_{i}^{\ast\alpha}=0,&\mathrm{on}~\partial{D}_{j}^{\ast}\cup\partial{D},~j\neq i,
\end{cases}
\end{equation}
respectively. Note that the definitions of $a_{ii}^{\ast\alpha\alpha}$, $i=1,2,\,\alpha=1,2,...,n$ are only valid if $n\geq3$, see Lemma \ref{lemmabc} for more details. Set
\begin{align}
&\mathbb{A}^{\ast}=(a_{11}^{\ast\alpha\beta})_{\frac{n(n+1)}{2}\times\frac{n(n+1)}{2}},\quad \mathbb{B}^{\ast}=\bigg(\sum\limits^{2}_{i=1}a_{i1}^{\ast\alpha\beta}\bigg)_{\frac{n(n+1)}{2}\times\frac{n(n+1)}{2}},\label{WZW}\\
&\mathbb{C}^{\ast}=\bigg(\sum\limits^{2}_{j=1}a_{1j}^{\ast\alpha\beta}\bigg)_{\frac{n(n+1)}{2}\times\frac{n(n+1)}{2}},\quad \mathbb{D}^{\ast}=\bigg(\sum\limits^{2}_{i,j=1}a_{ij}^{\ast\alpha\beta}\bigg)_{\frac{n(n+1)}{2}\times\frac{n(n+1)}{2}}.\notag
\end{align}
For $\alpha=1,2,...,\frac{n(n+1)}{2}$, $n\geq2$, substitute $\Big(\sum\limits_{i=1}^{2}b_{i}^{\ast1},...,\sum\limits_{i=1}^{2}b_{i}^{\ast\frac{n(n+1)}{2}}\Big)^{T}$ for the elements of $\alpha$-th column of $\mathbb{D}^{\ast}$ and then obtain new matrix $\mathbb{D}^{\ast\alpha}$ as follows:
\begin{gather*}
\mathbb{D}^{\ast\alpha}=
\begin{pmatrix}
\sum\limits^{2}_{i,j=1}a_{ij}^{\ast11}&\cdots&\sum\limits_{i=1}^{2}b_{i}^{\ast1}&\cdots&\sum\limits^{2}_{i,j=1}a_{ij}^{\ast1\,\frac{n(n+1)}{2}} \\\\ \vdots&\ddots&\vdots&\ddots&\vdots\\\\ \sum\limits^{2}_{i,j=1}a_{ij}^{\ast\frac{n(n+1)}{2}\,1}&\cdots&\sum\limits_{i=1}^{2}b_{i}^{\ast\frac{n(n+1)}{2}}&\cdots&\sum\limits^{2}_{i,j=1}a_{ij}^{\ast\frac{n(n+1)}{2}\,\frac{n(n+1)}{2}}
\end{pmatrix}.
\end{gather*}

In the case of $n=2,3$, denote
\begin{gather}\mathbb{A}_{0}^{\ast}=\begin{pmatrix} a_{11}^{\ast n+1\,n+1}&\cdots&a_{11}^{\ast n+1\frac{n(n+1)}{2}} \\\\ \vdots&\ddots&\vdots\\\\a_{11}^{\ast\frac{n(n+1)}{2}n+1}&\cdots&a_{11}^{\ast\frac{n(n+1)}{2}\frac{n(n+1)}{2}}\end{pmatrix},\label{LAGT001}\\
\mathbb{B}_{0}^{\ast}=\begin{pmatrix} \sum\limits^{2}_{i=1}a_{i1}^{\ast n+1\,1}&\cdots&\sum\limits^{2}_{i=1}a_{i1}^{\ast n+1\,\frac{n(n+1)}{2}} \\\\ \vdots&\ddots&\vdots\\\\ \sum\limits^{2}_{i=1}a_{i1}^{\ast\frac{n(n+1)}{2}1}&\cdots&\sum\limits^{2}_{i=1}a_{i1}^{\ast\frac{n(n+1)}{2}\frac{n(n+1)}{2}}\end{pmatrix},\notag\\
\mathbb{C}_{0}^{\ast}=\begin{pmatrix} \sum\limits^{2}_{j=1}a_{1j}^{\ast1\,n+1}&\cdots&\sum\limits^{2}_{j=1}a_{1j}^{\ast1\frac{n(n+1)}{2}} \\\\ \vdots&\ddots&\vdots\\\\ \sum\limits^{2}_{j=1}a_{1j}^{\ast\frac{n(n+1)}{2}\,n+1}&\cdots&\sum\limits^{2}_{j=1}a_{1j}^{\ast\frac{n(n+1)}{2}\frac{n(n+1)}{2}}\end{pmatrix}.\notag
\end{gather}
For $\alpha=1,2,...,\frac{n(n+1)}{2}$, after substituting column vector $\Big(b_{1}^{\ast n+1},...,b_{1}^{\ast\frac{n(n+1)}{2}}\Big)^{T}$ for the elements of $\alpha$-th column in the matrix $\mathbb{B}_{0}^{\ast}$, we obtain the new matrix denoted by $\mathbb{B}_{0}^{\ast\alpha}$ as follows:
\begin{gather*}
\mathbb{B}_{0}^{\ast\alpha}=
\begin{pmatrix}
\sum\limits^{2}_{i=1}a_{i1}^{\ast n+1\,1}&\cdots&b_{1}^{\ast n+1}&\cdots&\sum\limits^{2}_{i=1}a_{i1}^{\ast n+1\,\frac{n(n+1)}{2}} \\\\ \vdots&\ddots&\vdots&\ddots&\vdots\\\\ \sum\limits^{2}_{i=1}a_{i1}^{\ast\frac{n(n+1)}{2}\,1}&\cdots&b_{1}^{\ast\frac{n(n+1)}{2}}&\cdots&\sum\limits^{2}_{i=1}a_{i1}^{\ast\frac{n(n+1)}{2}\frac{n(n+1)}{2}}
\end{pmatrix}.
\end{gather*}
Write
\begin{align}\label{ZZ002}
\mathbb{F}_{0}^{\ast}=\begin{pmatrix} \mathbb{A}_{0}^{\ast}&\mathbb{B}_{0}^{\ast} \\  \mathbb{C}_{0}^{\ast}&\mathbb{D}^{\ast}
\end{pmatrix},\quad \mathbb{F}_{0}^{\ast\alpha}=\begin{pmatrix} \mathbb{A}_{0}^{\ast}&\mathbb{B}_{0}^{\ast\alpha} \\  \mathbb{C}_{0}^{\ast}&\mathbb{D}^{\ast\alpha}
\end{pmatrix},\quad\alpha=1,2,...,\frac{n(n+1)}{2}.
\end{align}

In the case of $n>3$, for $\alpha=1,2,...,\frac{n(n+1)}{2}$, let the elements of $\alpha$-th column of $\mathbb{B}^{\ast}$ be replaced by column vector $\Big(b_{1}^{\ast1},...,b_{1}^{\ast\frac{n(n+1)}{2}}\Big)^{T}$ and we then acquire the new matrix $\mathbb{B}_{1}^{\ast\alpha}$ as follows:
\begin{gather*}
\mathbb{B}_{1}^{\ast\alpha}=
\begin{pmatrix}
\sum\limits^{2}_{i=1}a_{i1}^{\ast11}&\cdots&b_{1}^{\ast1}&\cdots&\sum\limits^{2}_{i=1}a_{i1}^{\ast1\,\frac{n(n+1)}{2}} \\\\ \vdots&\ddots&\vdots&\ddots&\vdots\\\\ \sum\limits^{2}_{i=1}a_{i1}^{\ast\frac{n(n+1)}{2}\,1}&\cdots&b_{1}^{\ast\frac{n(n+1)}{2}}&\cdots&\sum\limits^{2}_{i=1}a_{i1}^{\ast\frac{n(n+1)}{2}\frac{n(n+1)}{2}}
\end{pmatrix}.
\end{gather*}
Set
\begin{align}\label{ZZ003}
\mathbb{F}_{1}^{\ast}=\begin{pmatrix} \mathbb{A}^{\ast}&\mathbb{B}^{\ast} \\  \mathbb{C}^{\ast}&\mathbb{D}^{\ast}
\end{pmatrix},\quad\mathbb{F}^{\ast\alpha}_{1}=\begin{pmatrix} \mathbb{A}^{\ast}&\mathbb{B}_{1}^{\ast\alpha} \\  \mathbb{C}^{\ast}&\mathbb{D}^{\ast\alpha}
\end{pmatrix},\quad\alpha=1,2,...,\frac{n(n+1)}{2}.
\end{align}

Previously in \cite{LX2022}, the free constants $C_{\ast}^{\alpha}$, $\alpha=1,2,...,\frac{n(n+1)}{2}$ are determined by condition \eqref{LDCZ003} but their values are inexplicit. In this paper, their values can be explicitly given by
\begin{align}\label{ZZWWWW}
C_{\ast}^{\alpha}=&
\begin{cases}
\frac{\det\mathbb{F}_{0}^{\ast\alpha}}{\det \mathbb{F}_{0}^{\ast}},&n=2,3,\vspace{0.5ex} \\
\frac{\det\mathbb{F}_{1}^{\ast\alpha}}{\det \mathbb{F}_{1}^{\ast}},&n>3,
\end{cases}
\end{align}
where $\mathbb{F}_{i}^{\ast\alpha}$ and $\mathbb{F}_{i}^{\ast\alpha}$, $i=0,1,\,\alpha=1,2,...,\frac{n(n+1)}{2}$ are defined by \eqref{ZZ002}--\eqref{ZZ003}. By capturing the exact values of $C_{\ast}^{\alpha}$, $\alpha=1,2,...,\frac{n(n+1)}{2}$, condition \eqref{LDCZ003} can be left out, which is an important improvement.

In the following, $O(1)$ denotes some quantity satisfying that $|O(1)|\leq\,C$ for some $\varepsilon$-independent positive constant $C$, whose value may differ at each occurrence and depend only on $\mu,n,R_{0},\kappa$, and the upper bounds of $\|\partial D\|_{C^{2,\gamma}}$ and $\|\partial D_{i}\|_{C^{2,\gamma}}$, $i=1,2$.

Define
\begin{align}\label{JTD}
r_{\varepsilon}:=&
\begin{cases}
\varepsilon^{\frac{1}{24}},&n=2,\\
|\ln\varepsilon|^{-1},&n=3,\\
\varepsilon^{\min\{\frac{1}{12},\frac{n-3}{24}\}},&n>3.
\end{cases}
\end{align}
The main objective of this paper is to establish the following convergence result.
\begin{theorem}\label{JGR}
Assume that $D_{1},D_{2}\subset D\subset\mathbb{R}^{n}\,(n\geq2)$ are described as above and condition \eqref{CONVEX001} holds. Then for a arbitrarily small $\varepsilon>0$,
\begin{align*}
\mathcal{B}_{\beta}[\boldsymbol{\varphi}]=\mathcal{B}_{\beta}^{\ast}[\boldsymbol{\varphi}]+O(r_{\varepsilon}),\quad\beta=1,2,...,\frac{n(n+1)}{2},
\end{align*}
where $\mathcal{B}_{\beta}[\boldsymbol{\varphi}]$ and $\mathcal{B}_{\beta}^{\ast}[\boldsymbol{\varphi}]$ are defined by \eqref{LGBC} and \eqref{KGF001}, respectively, the infinitely small quantity $r_{\varepsilon}$ is given by \eqref{JTD}.
\end{theorem}

\begin{remark}
Combining \eqref{ZZWWWW}, \eqref{LFN001} and \eqref{KTG001}, we obtain that for $\beta=1,2,...,\frac{n(n+1)}{2}$,
\begin{align*}
&\left|\int_{\partial D_{1}^{\ast}\cap\partial\Omega_{R_{0}}^{\ast}}\boldsymbol{\psi}_{\beta}\cdot\sigma[\mathbf{u}^{\ast}_{b},p^{\ast}_{b}]\right|\notag\\
&=\left|\int_{\partial D_{1}^{\ast}\cap\partial\Omega_{R_{0}}^{\ast}}\sum^{\frac{n(n+1)}{2}}_{\alpha=1}C^{\alpha}_{\ast}\boldsymbol{\psi}_{\beta}\cdot\sigma[\mathbf{u}^{\ast\alpha}_{1}+\mathbf{u}^{\ast\alpha}_{2},p^{\ast\alpha}_{1}+p_{2}^{\ast\alpha}]+\int_{\partial D_{1}^{\ast}\cap\partial\Omega_{R_{0}}^{\ast}}\boldsymbol{\psi}_{\beta}\cdot\sigma[\mathbf{u}^{\ast}_{0},p^{\ast\alpha}_{0}]\right|\notag\\
&\leq C\int^{R_{0}}_{0}t^{-2}e^{-\frac{1}{2Ct}}dt\leq CR_{0}^{-1}e^{-\frac{1}{2CR_{0}}},
\end{align*}
where $\Omega_{R_{0}}^{\ast}:=\Omega_{R_{0}}|_{\varepsilon=0}=\Omega^{\ast}\cap\{|x'|<R_{0}\}$. Then we have
\begin{align}\label{TGY}
\mathcal{B}_{\beta}^{\ast}[\boldsymbol{\varphi}]=\mathcal{B}_{\beta}^{\ast}[\boldsymbol{\varphi}]|_{\partial D_{1}^{\ast}\setminus\partial\Omega^{\ast}_{R_{0}}}+O(1)R_{0}^{-1}e^{-\frac{1}{2CR_{0}}},
\end{align}
where $\mathcal{B}_{\beta}^{\ast}[\boldsymbol{\varphi}]|_{\partial D_{1}^{\ast}\setminus\partial\Omega^{\ast}_{R_{0}}}:=\int_{\partial D_{1}^{\ast}\setminus\partial\Omega_{R_{0}}^{\ast}}\boldsymbol{\psi}_{\beta}\cdot\sigma[\mathbf{u}^{\ast}_{b},p^{\ast}_{b}].$ The result in \eqref{TGY} plays a key role in application to the numerical computations and simulations. In fact, \eqref{TGY} means that it only needs to use regular meshes to computationally calculate $\mathcal{B}_{\beta}^{\ast}[\boldsymbol{\varphi}]|_{\partial D_{1}^{\ast}\setminus\partial\Omega^{\ast}_{R_{0}}}$, which corresponds to the portion outside the narrow region between these two particles. It is worth emphasizing that this computation method was first presented in \cite{KLY2015} for a similar blow-up factor in the perfect conductivity problem. See \cite{MZ2021} for the corresponding result on the elasticity problem.

\end{remark}

The paper is structured as follows. Section \ref{SEC005} is dedicated to the establishment of Theorem \ref{JGR}. A direct application of Theorem \ref{JGR} gives the optimal gradient estimates and asymptotics for the Cauchy strain tensor, see Theorem \ref{MAINZW002} and Corollary \ref{MGA001} in Section \ref{KGRA90}, respectively. Let us conclude this introduction by reviewing
some earlier relevant investigations.

Ammari, Kang, Kim and Yu \cite{AKKY2020} were the first to rigorously establish an asymptotic expansion for the Cauchy stress tensor for Stokes flow in two dimensions by introducing the explicit singular functions and using the method of bipolar coordinates. Recently, Li and Xu \cite{LX2022} overcame the difficulty caused by the pressure to develop an adapted iterate technique with respect to the energy and then obtained the optimal upper and lower bounds on the Cauchy stress tensor. The original idea of this iterate technique was first presented in previous work \cite{LLBY2014} to establish gradient estimates for solutions to a class of elliptic systems arising from composites.

Besides the Cauchy stress tensor, hydrodynamic forces and the effective viscosity are another two important quantities of interest in fluid mechanics. Recently, based on the dual variational principle, Gorb \cite{G2016} developed a clear and concise analytical method to obtain the asymptotic expansions of hydrodynamic forces for two nearly touching spherical particles, as one of the particles moves with prescribed translational and angular velocities to the second particle. Li, Wang and Zhao \cite{LWZ2020} then extended the results in \cite{G2016} to the general $m$-convex inclusions in two and three dimensions with $m\geq2$, which revealed that the singularities induced by rotation hold dominant position in all directions of the forces rather than linear motion. With regard to the effective viscosity, Berlyand, Gorb and Novikov \cite{BGN2009} considered a two-dimensional mathematical model of a highly concentrated suspension and introduced a $fictitious\;fluid\;approach$ for inclusions in a densely packed suspension, which allows one to capture all singular terms in the asymptotic expansions of the viscous dissipation rate under generic boundary condition. For more previous studies on hydrodynamic forces and the effective viscosity, see \cite{BBP2006,FA1967,G1981,NK1984,CB1967,C1974,GCB1967} and the reference therein.

It has been proved by Ammari, Garapon, Kang and Lee \cite{AGKL2008} that when  $\lambda\rightarrow\infty$ and $\mu$ is fixed, the Stokes system becomes the limit of the Lam\'{e} system $\mu\Delta\mathbf{u}+(\lambda+\mu)\nabla\nabla\cdot\mathbf{u}=0$, which models the elasticity problem arising from composite materials. So it is natural that there appears similar high stress concentration for the elasticity problem. Bao, Li and Li \cite{BLL2015,BLL2017} took advantage of the iterate technique built in \cite{LLBY2014} to establish the pointwise upper bounds on the gradient of solution to the Lam\'{e} system with partially infinite coefficients in the presence of two strictly convex inclusions, as the distance between these two inclusions goes to zero. The optimality of the blow-up rate has also been demonstrated in \cite{L2018,MZ2021}. In the case when $m$-convex inclusions approach closely to the external boundary, Li and Zhao \cite{LZ2019} established the optimal gradient estimates for the solution, which revealed a novel blow-up phenomena that some special boundary data can strengthen the singularities of the stress. Kang and Yu \cite{KY2019} utilized the layer potential techniques and the variational principle to obtain the optimal gradient blow-up rate of solution to the Lam\'{e} system for two closely located inclusions in dimension two. It is worth emphasizing that their method is different from the iterate technique above. For earlier relevant studies on gradient estimates of solutions to the elliptic equation and system arising from composites, we refer to \cite{BASL1999,LN2003,LV2000,LY2009,BC1984,BV2000,BLY2009,AKL2005,K1993,AKLLL2007,Y2007,Y2009} and the reference therein.

\section{Proof of Theorem \ref{JGR}}\label{SEC005}

Denote
\begin{align}\label{DEL010}
\delta:=\delta(x')=\varepsilon+2h(x')=\varepsilon+2\kappa|x'|^{2},
\end{align}
where $\varepsilon\geq0$ and $\kappa$ is defined by \eqref{CONVEX001}. For later use, introduce some constants as follows:
\begin{align}
a_{1}=&\frac{6}{n-1},\;\, a_{2}=-2,\;\,b_{1}=-\frac{12}{2n-1},\;\,b_{2}=\frac{3}{2\kappa(2n-1)},\label{constants001}\\
b_{3}=&-5,\;\,b_{4}=\frac{4(n+1)}{2n-1},\;\,b_{5}=-12\kappa,\;\,b_{6}=\frac{3}{\kappa(2n-1)}.\label{constants002}
\end{align}
Construct a family of vector-valued auxiliary functions $\bar{\mathbf{u}}_{i}^{\alpha}\in C^{2,\gamma}(\Omega;\mathbb{R}^{n})$, $i=1,2,$ $\alpha=1,2,...,\frac{n(n+1)}{2}$, satisfying that $\bar{\mathbf{u}}_{i}^{\alpha}=\boldsymbol{\psi}_{\alpha}$ on $\partial D_{i}$, $\bar{\mathbf{u}}_{i}^{\alpha}=0$ on $\partial D_{j}\cup\partial D$, $i,j=1,2,\,i\neq j$, $\|\bar{\mathbf{u}}_{i}^{\alpha}\|_{C^{2,\gamma}(\Omega\setminus\Omega_{R_{0}})}\leq C$, $\nabla\cdot\bar{\mathbf{u}}_{i}^{\alpha}=0$ in $\Omega_{2R_{0}}$, and
\begin{align}\label{zzwz002}
\bar{\mathbf{u}}_{i}^{\alpha}=&\boldsymbol{\psi}_{\alpha}\left(\frac{1}{2}+(-1)^{i-1}\mathfrak{G}\right)+(-1)^{i-1}\left(\mathfrak{G}^{2}-\frac{1}{4}\right)\boldsymbol{\mathcal{F}}_{\alpha},\quad \mathrm{in}\;\Omega_{2R_{0}},
\end{align}
where $\boldsymbol{\psi}_{\alpha}$ is given by (\ref{OPP}), and
\begin{align}\label{deta}
\mathfrak{G}:=\mathfrak{G}(x)=\frac{x_{n}}{\delta},
\end{align}
and $\boldsymbol{\mathcal{F}}_{\alpha}=(\mathcal{F}_{\alpha}^{1},\mathcal{F}_{\alpha}^{2},...,\mathcal{F}_{\alpha}^{n})$ is defined by
\begin{align}\label{QLA001}
\boldsymbol{\mathcal{F}}_{\alpha}=&
\begin{cases}
\frac{\partial_{\alpha}\delta}{2}\boldsymbol{\psi}_{n},&\alpha=1,...,n-1,\\
\sum\limits^{n-1}_{i=1}\frac{a_{1}x_{i}}{\delta}\boldsymbol{\psi}_{i}+\frac{x_{n}}{\delta}\big(\frac{a_{1}(x'\cdot\nabla_{x'}\delta)}{\delta}+a_{2}\big)\boldsymbol{\psi}_{n},&\alpha=n,\\
\mathbf{0},&\alpha=2n,...,\frac{n(n+1)}{2},\,n\geq3,
\end{cases}
\end{align}
and
\begin{align}\label{DANZW001}
\boldsymbol{\mathcal{F}}_{\alpha}=&\sum\limits^{n-1}_{i=1}\frac{b_{1}x_{i}x_{\alpha-n}}{\delta}\boldsymbol{\psi}_{i}+(b_{2}+b_{3}x_{n}\mathfrak{G})\boldsymbol{\psi}_{\alpha-n}\notag\\
&+\Big[x_{\alpha-n}\mathfrak{G}\Big(\frac{b_{1}(x'\cdot\nabla_{x'}\delta)}{\delta}+b_{4}\Big)+b_{5}x_{\alpha-n}x_{n}\mathfrak{G}^{2}\Big]\boldsymbol{\psi}_{n},
\end{align}
for $\alpha=n+1,...,2n-1.$

In order to eliminate the maximal singular terms contained in $\mu\partial_{nn}\bar{\mathbf{u}}^{\alpha}_{i}$, $i=1,2,$ $\alpha=1,2,...,\frac{n(n+1)}{2}$, we construct the corresponding scalar auxiliary functions $\bar{p}_{i}^{\alpha}$, $i=1,2,$ $\alpha=1,2,...,\frac{n(n+1)}{2}$ as follows:
\begin{align}\label{PMAIN002}
\bar{p}_{i}^{\alpha}=&(-1)^{i-1}
\begin{cases}
\frac{\mu x_{n}\partial_{\alpha}\delta}{\delta^{2}},&\alpha=1,...,n-1,\\
-\frac{\mu a_{1}}{4\kappa\delta^{2}}+\frac{3\mu x_{n}^{2}}{\delta^{3}}\big(\frac{a_{1}(x'\cdot\nabla_{x'}\delta)}{\delta}+a_{2}\big),&\alpha=n,\\
\frac{\mu x_{\alpha-n}}{\delta^{2}}\Big[\frac{3x_{n}^{2}}{\delta}\big(\frac{b_{1}(x'\cdot\nabla_{x'}\delta)}{\delta}+b_{4}\big)+b_{6}\Big],&\alpha=n+1,...,2n-1,\\
0,&\alpha=2n,...,\frac{n(n+1)}{2},\,n\geq3,
\end{cases}
\end{align}
where $a_{i}$, $i=1,2,$ and $b_{j},$ $j=1,4,6$ are given by \eqref{constants001}--\eqref{constants002}. These auxiliary functions verify that for $i=1,2$,

$(1)$ for $j=1,...,n-1$,
\begin{align}\label{DNZ001}
&\mu\partial_{nn}(\bar{\mathbf{u}}_{i}^{\alpha})^{(j)}-\partial_{j}\bar{p}_{i}^{\alpha}\notag\\
&=
\begin{cases}
-\mu x_{n}\partial_{j}(\delta^{-2}\partial_{\alpha}\delta),&\alpha=1,...,n-1,\\
-3\mu x_{n}^{2}\partial_{j}\big[\delta^{-3}\big(\frac{a_{1}(x'\cdot\nabla_{x'}\delta)}{\delta}+a_{2}\big)\big],&\alpha=n,\\
\frac{\mu b_{3}(24\mathfrak{G}^{2}-1)}{2\delta}-3\mu x_{n}^{2}\partial_{j}\big[x_{\alpha-n}\delta^{-3}\big(\frac{b_{1}(x'\cdot\nabla_{x'}\delta)}{\delta}+b_{4}\big)\big],&\alpha=n+1,...,2n-1,\\
0,&\alpha=2n,...,\frac{n(n+1)}{2},\,n\geq3;
\end{cases}
\end{align}

$(2)$ for $j=n$,
\begin{align}\label{DNZ002}
&\mu\partial_{nn}(\bar{\mathbf{u}}_{i}^{\alpha})^{(n)}-\partial_{n}\bar{p}_{i}^{\alpha}=
\begin{cases}
\frac{4\mu b_{5}x_{\alpha-n}\mathfrak{G}(5\mathfrak{G}^{2}-6)}{\delta},&\alpha=n+1,...,2n-1,\\
0,&\text{otherwise}.
\end{cases}
\end{align}

We here would like to emphasize that the auxiliary functions defined by \eqref{zzwz002} and \eqref{PMAIN002} were first given in \cite{LX2022} in two and three dimensions. These auxiliary functions can be found by using method of undetermined coefficients. For readers' convenience, we now give the detailed calculations for the correction terms in \eqref{QLA001}. Take $\bar{\mathbf{u}}_{1}^{\alpha}$, $\alpha=n+1,...,2n-1$ for example. Other cases are the same. In view of \eqref{DEL010} and \eqref{deta}, we have
\begin{align*}
\partial_{i}\delta=4\kappa x_{i},\;\, \partial_{i}\mathfrak{G}=-4\kappa x_{i}\mathfrak{G}\partial_{n}\mathfrak{G},\quad i=1,...,n-1.
\end{align*}
This, in combination with \eqref{zzwz002} and \eqref{DANZW001}, leads to that
\begin{align*}
\nabla\cdot\bar{\mathbf{u}}_{1}^{\alpha}=&-\left(1+\frac{b_{1}n+b_{4}}{4}\right)x_{\alpha-n}\partial_{n}\mathfrak{G}+(b_{1}n+3b_{4}-8\kappa b_{2})x_{\alpha-n}\mathfrak{G}^{2}\partial_{n}\mathfrak{G}\notag\\
&-\left(4\kappa+\frac{3b_{5}-4\kappa b_{3}}{4}\right)x_{\alpha-n}\mathfrak{G}^{2}+(5b_{5}-12\kappa b_{3})x_{\alpha-n}\mathfrak{G}^{4}.
\end{align*}
For the purpose of letting $\bar{\mathbf{u}}_{1}^{\alpha}$ satisfy the incompressible condition, it suffices to require that
\begin{align*}
\begin{cases}
1+\frac{b_{1}n+b_{4}}{4}=0,\\
b_{1}n+3b_{4}-8\kappa b_{2}=0,\\
4\kappa+\frac{3b_{5}-4\kappa b_{3}}{4}=0,\\
5b_{5}-12\kappa b_{3}=0.
\end{cases}
\end{align*}
Then we deduce
\begin{align}\label{AGZW001}
b_{1}=-\frac{6+4\kappa b_{2}}{n},\;\, b_{3}=-5,\;\, b_{4}=2+4\kappa b_{2},\;\,b_{5}=-12\kappa.
\end{align}
In light of \eqref{PMAIN002} and in order to get ride of the largest singular terms of order $O(\delta^{-2})$ in $\mu\partial_{nn}(\bar{\mathbf{u}}^{\alpha}_{1})^{(i)}$, $i=1,...,n-1,$ it also needs to require that
\begin{align*}
\begin{cases}
\frac{2\mu}{\delta^{2}}\left(\frac{b_{1}x_{\alpha-n}^{2}}{\delta}+b_{2}\right)=\mu\partial_{\alpha-n}(b_{6}x_{\alpha-n}\delta^{-2}),\\
\frac{2\mu b_{1}x_{i}x_{\alpha-n}}{\delta^{3}}=\mu x_{\alpha-n}\partial_{i}(b_{6}\delta^{-2}),\quad i=1,...,n-1,\,i\neq\alpha-n,\,n\geq3,
\end{cases}
\end{align*}
which implies that
\begin{align*}
\begin{cases}
(2b_{2}-b_{6})\delta^{-2}+(2b_{1}+8\kappa b_{6})x_{\alpha-n}^{2}\delta^{-3}=0,\\
(2b_{1}+8\kappa b_{6})x_{i}x_{\alpha-n}\delta^{-3}=0,\quad i=1,...,n-1,\,i\neq\alpha-n,\,n\geq3.
\end{cases}
\end{align*}
Then we obtain
\begin{align*}
b_{1}=-4\kappa b_{6},\;\,b_{2}=\frac{b_{6}}{2}.
\end{align*}
This, together with \eqref{AGZW001}, shows that
\begin{align*}
b_{1}=&-\frac{12}{2n-1},\;\,b_{2}=\frac{3}{2\kappa(2n-1)},\;\,b_{3}=-5,\notag\\
b_{4}=&\frac{4(n+1)}{2n-1},\;\,b_{5}=-12\kappa,\;\,b_{6}=\frac{3}{\kappa(2n-1)}.
\end{align*}


For $i=1,2,$ and $\alpha=1,2,...,\frac{n(n+1)}{2}$, denote
\begin{align}\label{QMAZ001}
(q^{\alpha}_{i})_{\delta;x'}:=\frac{1}{|\Omega_{\delta}(x')|}\int_{\Omega_{\delta}(x')}q^{\alpha}_{i}(y)dy,\quad q^{\alpha}_{i}:=p^{\alpha}_{i}-\bar{p}^{\alpha}_{i},\quad\mathrm{for}\;|x'|\leq R_{0},
\end{align}
where $\delta:=\delta(x')$ is given by \eqref{DEL010}. Then we have
\begin{prop}\label{thm86}
Assume as above. For $i=1,2,\,\alpha=1,2,...,\frac{n(n+1)}{2}$, let $(\mathbf{u}_{i}^{\alpha},p_{i}^{\alpha})\in H^{1}(\Omega;\mathbb{R}^{n})$ be the solution of \eqref{qaz003az}. Then, for a arbitrarily small $\varepsilon>0$ and $x\in\Omega_{R_{0}}$,
\begin{align}\label{Le2.025}
&\|\nabla(\mathbf{u}_{i}^{\alpha}-\bar{\mathbf{u}}_{i}^{\alpha})\|_{L^{\infty}(\Omega_{\delta/2}(x'))}+\|q_{i}^{\alpha}-(q_{i}^{\alpha})_{\delta;x'}\|_{L^{\infty}(\Omega_{\delta/2}(x'))}\notag\\
&\leq C\begin{cases}
1,&\alpha=1,...,n-1,\\
\delta^{-1/2},&\alpha=n,\\
1,&\alpha=n+1,...,2n-1,\\
\delta^{1/2},&\alpha=2n,...,\frac{n(n+1)}{2},\,n\geq3,
\end{cases}
\end{align}
where $\delta$ is defined by \eqref{DEL010}, $\bar{\mathbf{u}}_{i}^{\alpha}$ is given by \eqref{zzwz002}, $q_{i}^{\alpha}$ and $(q^{\alpha}_{i})_{\delta;x'}$, $i=1,2,\,\alpha=1,2,...,\frac{n(n+1)}{2}$ are defined by \eqref{QMAZ001}.
\end{prop}

Before utilizing the iteration technique developed in \cite{LX2022} to prove Proposition \ref{thm86}, we first recall the following two lemmas.
\begin{lemma}\label{lem001m}
Let $\Omega\subset\mathbb{R}^{n}(n\geq2)$ be a bounded Lipschitz domain. Then for any fixed $f\in L^{2}(\Omega)$ with $\int_{\Omega}f=0$, there exists a vector-valued function $\boldsymbol{\phi}\in H^{1}_{0}(\Omega;\mathbb{R}^{n})$ such that
\begin{align}\label{PAZ001}
\nabla\cdot\boldsymbol{\phi}=f\;\,\text{in $\Omega$,}\quad\|\boldsymbol{\phi}\|_{H^{1}_{0}(\Omega)}\leq C(n,\mathrm{diam}(\Omega))\|f\|_{L^{2}(\Omega)}.
\end{align}
In particular, if $\Omega=B_{r}$, $r>0$, then
\begin{align}\label{SPHE001}
\|\boldsymbol{\phi}\|_{L^{2}(B_{r})}+r\|\nabla\boldsymbol{\phi}\|_{L^{2}(B_{r})}\leq C(n)r\|f\|_{L^{2}(B_{r})}.
\end{align}

\end{lemma}
The proof of lemma \ref{lem001m} can be seen in \cite{T1984}. Using Lemma \ref{lem001m}, we obtain the following result, which can be seen in Lemma 3.2 of \cite{LX2022}.

\begin{lemma}\label{lem002m}
Let $(\mathbf{w},q)$ be the solution of $\nabla\cdot\sigma[\mathbf{w},q]=\mathbf{g}$ in $\Omega$ with $\mathbf{g}\in L^{2}(\Omega;\mathbb{R}^{n})$. Then
\begin{align}\label{PAZ002}
\|q-q_{_\Omega}\|_{L^{2}(\Omega)}\leq C(\mu,n,\mathrm{diam}(\Omega))(\|\nabla\mathbf{w}\|_{L^{2}(\Omega)}+\|\mathbf{g}\|_{L^{2}(\Omega)}),
\end{align}
where $q_{_\Omega}:=\frac{1}{|\Omega|}\int_{\Omega}q\,dx$. Especially when $\Omega=B_{r}$, $r>0$, we have
\begin{align}\label{SPHE002}
\|q-q_{_{B_{r}}}\|_{L^{2}(B_{r})}\leq C(\mu,n)(\|\nabla\mathbf{w}\|_{L^{2}(B_{r})}+r\|\mathbf{g}\|_{L^{2}(B_{r})}).
\end{align}

\end{lemma}
We here would like to point out that \eqref{SPHE001} and \eqref{SPHE002} can be directly deduced from \eqref{PAZ001} and \eqref{PAZ002} by using a rescaling argument, see Corollary 3.3 in \cite{LX2022} for more details.

\begin{proof}[Proof of Proposition \ref{thm86}]
Without impeding the generality, consider the case of $i=1$. The case when $i=2$ is similar and thus omitted. From \eqref{zzwz002} and \eqref{PMAIN002}--\eqref{DNZ002}, a direct calculation yields that
\begin{align}\label{DM001}
|\nabla\cdot\sigma[\bar{\mathbf{u}}_{1}^{\alpha},\bar{p}_{1}^{\alpha}]|\leq&\sum^{n}_{j=1}|\mu\partial_{nn}(\bar{\mathbf{u}}_{1}^{\alpha})^{(j)}-\partial_{j}\bar{p}_{1}^{\alpha}|+|\mu\Delta_{x'}\bar{\mathbf{u}}_{1}^{\alpha}|\notag\\
\leq&C
\begin{cases}
\delta^{-1},&\alpha=1,...,n-1,\\
\delta^{-3/2},&\alpha=n,\\
\delta^{-1},&\alpha=n+1,...,2n-1,\\
\delta^{-1/2},&\alpha=2n,...,\frac{n(n+1)}{2},\,n\geq3.
\end{cases}
\end{align}

For $x\in\Omega$, let
\begin{equation*}
\mathbf{w}_{1}^{\alpha}:=\mathbf{u}_{1}^{\alpha}-\bar{\mathbf{u}}_{1}^{\alpha},\quad \alpha=1,2,...,\frac{n(n+1)}{2}.
\end{equation*}
Then $\mathbf{w}_{1}^{\alpha}$ is the solution of
\begin{align}\label{RNZ}
\begin{cases}
\nabla\cdot\sigma[\mathbf{w}_{1}^{\alpha},q_{1}^{\alpha}]=-\nabla\cdot\sigma[\bar{\mathbf{u}}_{1}^{\alpha},\bar{p}_{1}^{\alpha}],&
\hbox{in}\  \Omega,  \\
\nabla\cdot\mathbf{w}_{1}^{\alpha}=0,&\mathrm{in}\;\Omega_{2R_{0}},\\
\nabla\cdot\mathbf{w}_{1}^{\alpha}=-\nabla\cdot\bar{\mathbf{u}}_{1}^{\alpha},&\mathrm{in}\;\Omega\setminus\Omega_{2R_{0}},\\
\mathbf{w}_{1}^{\alpha}=0, \quad&\hbox{on} \ \partial\Omega.
\end{cases}
\end{align}

\noindent{\bf Step 1.}
For $\alpha=1,2,...,\frac{n(n+1)}{2}$, let $(\mathbf{w}_{1}^{\alpha},q_{1}^{\alpha})$ be the solution of \eqref{RNZ}. Then
\begin{align}\label{YGZA001}
\int_{\Omega}|\nabla\mathbf{w}_{1}^{\alpha}|^2dx\leq C,\quad \alpha=1,2,...,\frac{n(n+1)}{2}.
\end{align}

From \eqref{RNZ}, we know that $(\mathbf{w}_{1}^{\alpha},q_{1}^{\alpha})$ also verifies
\begin{align}\label{SU001}
\nabla\cdot\sigma[\mathbf{w}_{1}^{\alpha},q_{1}^{\alpha}-(q_{1}^{\alpha})_{_{R_{0},\text{out}}}]=-\nabla\cdot\sigma[\bar{\mathbf{u}}_{1}^{\alpha},\bar{p}_{1}^{\alpha}],
\end{align}
where $(q_{1}^{\alpha})_{_{R_{0},\text{out}}}:=\frac{1}{|\Omega\setminus\Omega_{R_{0}}|}\int_{\Omega\setminus\Omega_{R_{0}}}q_{1}^{\alpha}dx.$ Taking the test function $\mathbf{w}_{1}^{\alpha}$ for equation \eqref{SU001}, it follows from integration by parts that
\begin{align}\label{RIGH001}
\mu\int_{\Omega}|\nabla\mathbf{w}_{1}^{\alpha}|^{2}=-\int_{\Omega\setminus\Omega_{R_{0}}}(q_{1}^{\alpha}-(q_{1}^{\alpha})_{_{R_{0},\text{out}}})\nabla\cdot\bar{\mathbf{u}}_{1}^{\alpha}+\int_{\Omega}(\nabla\cdot\sigma[\bar{\mathbf{u}}_{1}^{\alpha},\bar{p}_{1}^{\alpha}])\cdot\mathbf{w}^{\alpha}_{1}.
\end{align}
Utilizing H\"{o}lder's inequality and applying Lemma \ref{lem002m} to \eqref{PAZ002} in $\Omega\setminus\Omega_{R_{0}}$, it follows from the definitions of $\bar{\mathbf{u}}_{1}^{\alpha}$ and $\bar{p}_{1}^{\alpha}$ that
\begin{align*}
&\left|\int_{\Omega\setminus\Omega_{R_{0}}}(q_{1}^{\alpha}-(q_{1}^{\alpha})_{_{R_{0},\text{out}}})\nabla\cdot\bar{\mathbf{u}}_{1}^{\alpha}\right|\notag\\
&\leq C\|q_{1}^{\alpha}-(q_{1}^{\alpha})_{_{R_{0},\text{out}}}\|_{L^{2}(\Omega\setminus\Omega_{R_{0}})}\leq C\|\nabla\mathbf{w}_{1}^{\alpha}\|_{L^{2}(\Omega\setminus\Omega_{R_{0}})}+C.
\end{align*}

For the second term in the right hand side of \eqref{RIGH001}, we split it as follows:
\begin{align*}
\mathrm{I}^{\alpha}:=\int_{\Omega\setminus\Omega_{R_{0}}}(\nabla\cdot\sigma[\bar{\mathbf{u}}_{1}^{\alpha},\bar{p}_{1}^{\alpha}])\cdot\mathbf{w}^{\alpha}_{1},\;\,\mathrm{II}^{\alpha}:=\int_{\Omega_{R_{0}}}(\nabla\cdot\sigma[\bar{\mathbf{u}}_{1}^{\alpha},\bar{p}_{1}^{\alpha}])\cdot\mathbf{w}^{\alpha}_{1}.
\end{align*}
From Poincar\'{e} inequality, we have
\begin{align}\label{ZAGDQ001}
|\mathrm{I}^{\alpha}|\leq C\|\nabla\mathbf{w}^{\alpha}_{1}\|_{L^{2}(\Omega\setminus\Omega_{R_{0}})}.
\end{align}
With regard to $\mathrm{II}^{\alpha}$, we have
\begin{align}\label{NAGZ001}
\mathrm{II}^{\alpha}=&\int_{\Omega_{R_{0}}}(\mu\Delta\bar{\mathbf{u}}_{1}^{\alpha}-\nabla\bar{p}_{1}^{\alpha})\cdot\mathbf{w}_{1}^{\alpha}\notag\\
=&\int_{\Omega_{R_{0}}}\mu(\Delta_{x'}\bar{\mathbf{u}}_{1}^{\alpha})\cdot\mathbf{w}_{1}^{\alpha}+\int_{\Omega_{R_{0}}}(\mu\partial_{nn}\bar{\mathbf{u}}_{1}^{\alpha}-\nabla\bar{p}_{1}^{\alpha})\cdot\mathbf{w}_{1}^{\alpha}\notag\\
=&:\mathrm{II}_{1}^{\alpha}+\mathrm{II}_{2}^{\alpha}.
\end{align}
For $\mathrm{II}_{1}^{\alpha}$, integrating by parts, we have from the Sobolev trace embedding theorem and H\"{o}lder's inequality that

$(i)$ for $\alpha=1,...,n-1$ and $\alpha=2n,...,\frac{n(n+1)}{2},\,n\geq3$,
\begin{align}\label{MNEZ001}
|\mathrm{II}_{1}^{\alpha}|\leq&\left|\mu\int_{\Omega_{R_{0}}}\nabla_{x'}\mathbf{w}_{1}^{\alpha}\nabla_{x'}\bar{\mathbf{u}}_{1}^{\alpha}\right|+\int\limits_{\scriptstyle |x'|={R_{0}},\atop\scriptstyle
-\varepsilon/2-h(x')<x_{n}<\varepsilon/2+h(x')\hfill}C|\mathbf{w}_{1}^{\alpha}|\notag\\
\leq&C\int_{\Omega_{R_{0}}}|\nabla_{x'}\mathbf{w}_{1}^{\alpha}|\delta^{-1/2}+C\|\nabla\mathbf{w}_{1}^{\alpha}\|_{L^{2}(\Omega\setminus\Omega_{R_{0}})}\leq C\|\nabla\mathbf{w}_{1}^{\alpha}\|_{L^{2}(\Omega)};
\end{align}

$(ii)$ for $\alpha=n,n+1,...,2n-1$, we have
\begin{align}\label{EAZ001}
|\mathrm{II}_{1}^{\alpha}|\leq&\left|\mu\int_{\Omega_{R_{0}}}\nabla_{x'}\mathbf{w}_{1}^{\alpha}\nabla_{x'}\bar{\mathbf{u}}_{1}^{\alpha}\right|+C\|\nabla\mathbf{w}_{1}^{\alpha}\|_{L^{2}(\Omega\setminus\Omega_{R_{0}})}.
\end{align}
Observe that for $i,j=1,...,n-1$, if $\alpha=n,$
\begin{align}\label{DMAZ001}
\Big(\mathfrak{G}^{2}-\frac{1}{4}\Big)\partial_{j}\boldsymbol{\mathcal{F}}_{n}^{i}=&
\begin{cases}
a_{1}x_{i}\delta^{-1}\partial_{j}\delta(\frac{1}{4}\partial_{n}\mathfrak{G}-\frac{1}{3}\partial_{n}\mathfrak{G}^{3}),&i\neq j,\\
a_{1}(x_{i}\delta^{-1}\partial_{i}\delta-1)(\frac{1}{4}\partial_{n}\mathfrak{G}-\frac{1}{3}\partial_{n}\mathfrak{G}^{3}),&i=j,
\end{cases}
\end{align}
and, if $\alpha=n+1,...,2n-1$,
\begin{align}\label{DMAZ002}
&\Big(\mathfrak{G}^{2}-\frac{1}{4}\Big)\partial_{j}\boldsymbol{\mathcal{F}}_{\alpha}^{i}\notag\\
&=
\begin{cases}
b_{1}(x_{i}x_{\alpha-n}\delta^{-1}\partial_{j}\delta-\partial_{j}(x_{i}x_{\alpha-n}))(\frac{1}{4}\partial_{n}\mathfrak{G}-\frac{1}{3}\partial_{n}\mathfrak{G}^{3}),&i\neq j,\\
b_{1}(x^{2}_{\alpha-n}\delta^{-1}\partial_{j}\delta-\partial_{j}x^{2}_{\alpha-n})(\frac{1}{4}\partial_{n}\mathfrak{G}-\frac{1}{3}\partial_{n}\mathfrak{G}^{3})+b_{3}\partial_{j}\delta \mathfrak{G}^{2}(\frac{1}{4}-\mathfrak{G}^{2}),&i=j.
\end{cases}
\end{align}
For $i,j=1,...,n-1$, denote
\begin{align}\label{DMAZ003}
\mathcal{A}_{ij}^{n}=&
\begin{cases}
a_{1}\mathfrak{G}(\frac{1}{4}-\frac{1}{3}\mathfrak{G}^{2})x_{i}\partial_{j}(\delta^{-1}\partial_{j}\delta)+a_{1}(\frac{1}{4}-\mathfrak{G}^{2})\partial_{j}\mathfrak{G}x_{i}\delta^{-1}\partial_{j}\delta,&i\neq j,\\
a_{1}\mathfrak{G}(\frac{1}{4}-\frac{1}{3}\mathfrak{G}^{2})\partial_{i}(x_{i}\delta^{-1}\partial_{i}\delta)+a_{1}(\frac{1}{4}-\mathfrak{G}^{2})\partial_{i}\mathfrak{G}(x_{i}\delta^{-1}\partial_{i}\delta-1),&i=j,
\end{cases}
\end{align}
and
\begin{align}\label{DMAZ005}
\mathcal{A}_{ij}^{\alpha}=&
b_{1}\mathfrak{G}(\frac{1}{4}-\frac{1}{3}\mathfrak{G}^{2})\partial_{j}(x_{i}x_{\alpha-n}\delta^{-1}\partial_{j}\delta-\partial_{j}(x_{i}x_{\alpha-n}))\notag\\
&+b_{1}(\frac{1}{4}-\mathfrak{G}^{2})(x_{i}x_{\alpha-n}\delta^{-1}\partial_{j}\delta-\partial_{j}(x_{i}x_{\alpha-n})),\quad\alpha=n+1,...,2n-1.
\end{align}
Then combining \eqref{DMAZ001}--\eqref{DMAZ005}, we deduce from integration by parts, the Sobolev trace embedding theorem and H\"{o}lder's inequality again that
\begin{align*}
&\left|\mu\int_{\Omega_{R_{0}}}\nabla_{x'}\mathbf{w}_{1}^{\alpha}\nabla_{x'}\bar{\mathbf{u}}_{1}^{\alpha}\right|\notag\\
&\leq\left|\mu\int_{\Omega_{R_{0}}}\nabla_{x'}\mathbf{w}_{1}^{\alpha}\nabla_{x'}((\mathfrak{G}^{2}-1/4)\boldsymbol{\mathcal{F}}_{\alpha})\right|+\left|\mu\int_{\Omega_{R_{0}}}\nabla_{x'}\mathbf{w}_{1}^{\alpha}\nabla_{x'}(\boldsymbol{\psi}_{\alpha}(\mathfrak{G}+1/2))\right|\notag\\
&\leq\left|\mu\int_{\Omega_{R_{0}}}(\mathfrak{G}^{2}-1/4)\nabla_{x'}\mathbf{w}_{1}^{\alpha}\nabla_{x'}\boldsymbol{\mathcal{F}}_{\alpha}\right|+C\int_{\Omega_{R_{0}}}|\nabla_{x'}\mathbf{w}_{1}^{\alpha}|\delta^{-1/2}\\
&\leq\left|\mu\sum^{n}_{i=1}\sum^{n-1}_{j=1}\int_{\Omega_{R_{0}}}\partial_{j}(\mathbf{w}_{1}^{\alpha})^{(i)}\partial_{j}\boldsymbol{\mathcal{F}}^{i}_{\alpha}\right|+C\|\nabla_{x'}\mathbf{w}_{1}^{\alpha}\|_{L^{2}(\Omega_{R_{0}})}\\
&\leq\left|\mu\sum^{n-1}_{i,j=1}\int_{\Omega_{R_{0}}}\mathcal{A}_{ij}^{\alpha}\partial_{n}(\mathbf{w}_{1}^{\alpha})^{(i)}\right|+C\|\nabla_{x'}\mathbf{w}_{1}^{\alpha}\|_{L^{2}(\Omega_{R_{0}})}\leq C\|\nabla\mathbf{w}_{1}^{\alpha}\|_{L^{2}(\Omega_{R_{0}})}.
\end{align*}
Substituting this into \eqref{EAZ001}, we obtain
\begin{align}\label{KWZ001}
|\mathrm{II}_{1}^{\alpha}|\leq C\|\nabla\mathbf{w}_{1}^{\alpha}\|_{L^{2}(\Omega)},\quad\alpha=n,n+1,...,2n-1.
\end{align}

Denote
\begin{align*}
\mathcal{B}_{j}^{\alpha}=&
\begin{cases}
\mu(\frac{1}{2}\partial_{\alpha j}\delta-\delta^{-1}\partial_{j}\delta\partial_{\alpha}\delta)\mathfrak{G},\quad\alpha,j=1,...,n-1,\\
\mu a_{1}\delta^{-1}[\partial_{j}(x'\cdot\nabla_{x'}\delta)-\delta^{-1}\partial_{j}\delta(x'\cdot\nabla_{x'}\delta)]\mathfrak{G}^{3}\\
-3\mu\delta^{-1}\partial_{j}\delta(a_{1}\delta^{-1}x'\cdot\nabla_{x'}\delta+a_{2})\mathfrak{G},\quad\alpha=n,\,j=1,...,n-1,\\
\mu b_{3}\mathfrak{G}(4\mathfrak{G}^{2}-\frac{1}{2})+3\mu x_{\alpha-n}\delta^{-1}\partial_{j}\delta(b_{1}\delta^{-1}x'\cdot\nabla_{x'}\delta+b_{4})\mathfrak{G}^{3}\\
+\mu\partial_{j}[x_{\alpha-n}(b_{1}\delta^{-1}x'\cdot\nabla_{x'}\delta+b_{4})]\mathfrak{G}^{3},\quad\alpha=n+1,...,2n-1,\,j=1,...,n-1,\\
\mu b_{5}x_{\alpha-n}\mathfrak{G}^{2}(5\mathfrak{G}^{2}-12),\quad\alpha=n+1,...,2n-1,\,j=n,\\
0,\quad\alpha=1,2,...,n,\,j=n,\;\mathrm{or}\;\alpha=2n,...,\frac{n(n+1)}{2},\,n\geq3,\,j=1,2,...,n.
\end{cases}
\end{align*}
It then follows from integration by parts and H\"{o}lder's inequality that
\begin{align*}
|\mathrm{II}^{\alpha}_{2}|=&\left|\sum^{n}_{j=1}\int_{\Omega_{R_{0}}}(\mu\partial_{nn}(\bar{\mathbf{u}}_{1}^{\alpha})^{(j)}-\partial_{j}\bar{p}_{1}^{\alpha})(\mathbf{w}_{1}^{\alpha})^{(j)}\right|\notag\\
=&\left|\sum^{n}_{j=1}\int_{\Omega_{R_{0}}}\mathcal{B}_{j}^{\alpha}\partial_{n}(\mathbf{w}_{1}^{\alpha})^{(j)}\right|\leq C\|\nabla\mathbf{w}_{1}^{\alpha}\|_{L^{2}(\Omega)}.
\end{align*}
This, together with \eqref{RIGH001}--\eqref{MNEZ001} and \eqref{KWZ001}, gives that \eqref{YGZA001} holds.

\noindent{\bf Step 2.}
Claim that for $\alpha=1,2,...,\frac{n(n+1)}{2}$,
\begin{align}\label{HN001}
 \int_{\Omega_\delta(z')}|\nabla\mathbf{w}_{1}^{\alpha}|^2dx\leq& C\delta^{n}\begin{cases}
1,&\alpha=1,...,n-1,\\
\delta^{-1},&\alpha=n,\\
1,&\alpha=n+1,...,2n-1,\\
\delta,&\alpha=2n,...,\frac{n(n+1)}{2},\,n\geq3.
\end{cases}
\end{align}

For $|z'|\leq R_{0}$, $\delta\leq t<s\leq\vartheta(\kappa)\delta^{1/2}$, $\vartheta(\kappa)=\frac{1}{8\sqrt{2\kappa}}$, we deduce
\begin{align}\label{AZQY001}
|\delta(x')-\delta(z')|=&2\kappa||x'|^{2}-|z'|^{2}|\leq 4\kappa|x'_{\theta}||x'-z'|\leq4\kappa s(s+|z'|)\leq\frac{\delta(z')}{2},
\end{align}
where $x_{\theta}'$ is some point between $x_{0}'$ and $x'$. This reads that
\begin{align}\label{QWN001}
\frac{1}{2}\delta(z')\leq\delta(x')\leq\frac{3}{2}\delta(z'),\quad\mathrm{in}\;\Omega_{s}(z').
\end{align}
Pick a smooth cutoff function $\eta\in C^{2}(\Omega_{2R_{0}})$ such that $\eta(x')=1$ if $|x'-z'|<t$, $\eta(x')=0$ if $|x'-z'|>s$, $0\leq\eta(x')\leq1$ if $t\leq|x'-z'|\leq s$, and $|\nabla_{x'}\eta(x')|\leq\frac{2}{s-t}$. In light of \eqref{RNZ}, we see that $(\mathbf{w}_{1}^{\alpha},q_{1}^{\alpha})$ also satisfies
\begin{align}\label{SOUL001}
\nabla\cdot\sigma[\mathbf{w}_{1}^{\alpha},q_{1}^{\alpha}-(q_{1}^{\alpha})_{_{s;z'}}]=-\nabla\cdot\sigma[\bar{\mathbf{u}}_{1}^{\alpha},\bar{p}_{1}^{\alpha}],\quad\mathrm{in}\;\Omega_{s}(z'),
\end{align}
where $(q_{1}^{\alpha})_{_{s;z'}}:=\frac{1}{|\Omega_{s}(z')|}\int_{\Omega_{s}(z')}q_{1}^{\alpha}dx.$ Multiplying \eqref{SOUL001} by $\mathbf{w}_{1}^{\alpha}\eta^{2}$, we have from integration by parts that
\begin{align}\label{HN002}
\mu\int_{\Omega_{s}(z')}|\nabla\mathbf{w}_{1}^{\alpha}|^{2}\eta^{2}=&-\mu\int_{\Omega_{s}(z')}(\mathbf{w}_{1}^{\alpha}\nabla\mathbf{w}_{1}^{\alpha})\cdot\eta^{2}+\int_{\Omega_{s}(z')}(\nabla\cdot\sigma[\bar{\mathbf{u}}_{1}^{\alpha},\bar{p}_{1}^{\alpha}])\cdot\mathbf{w}_{1}^{\alpha}\eta^{2}\notag\\
&+\int_{\Omega_{s}(z')}(q_{1}^{\alpha}-(q_{1}^{\alpha})_{_{s;z'}})\nabla\cdot(\mathbf{w}_{1}^{\alpha}\eta^{2}).
\end{align}
Since $\mathbf{w}_{1}^{\alpha}=0$ on $\Gamma_{R_{0}}^{-}$, it follows from \eqref{QWN001}, Young's inequality and Poincar\'{e} inequality that
\begin{align}\label{DAMZ990}
\int_{\Omega_{s}(z')}|\mathbf{w}_{1}^{\alpha}\nabla\mathbf{w}_{1}^{\alpha}||\nabla\eta^{2}|\leq&\frac{1}{4}\int_{\Omega_{s}(z')}|\nabla\mathbf{w}_{1}^{\alpha}|^{2}+\frac{C}{(s-t)^{2}}\int_{\Omega_{s}(z')}|\mathbf{w}_{1}^{\alpha}|^{2}\notag\\
\leq&\frac{1}{4}\int_{\Omega_{s}(z')}|\nabla\mathbf{w}_{1}^{\alpha}|^{2}+\left(\frac{C\delta}{s-t}\right)^{2}\int_{\Omega_{s}(z')}|\nabla\mathbf{w}_{1}^{\alpha}|^{2},
\end{align}
and
\begin{align}\label{DAMZ991}
&\frac{1}{\mu}\left|\int_{\Omega_{s}(z')}(\nabla\cdot\sigma[\bar{\mathbf{u}}_{1}^{\alpha},\bar{p}_{1}^{\alpha}])\cdot\mathbf{w}_{1}^{\alpha}\eta^{2}\right|\notag\\
&\leq\left(\frac{C\delta}{s-t}\right)^{2}\int_{\Omega_{s}(z')}|\nabla\mathbf{w}_{1}^{\alpha}|^{2}+C(s-t)^{2}\int_{\Omega_{s}(z')}\left|\nabla\cdot\sigma[\bar{\mathbf{u}}_{1}^{\alpha},\bar{p}_{1}^{\alpha}]\right|^{2}.
\end{align}
For $\delta(z')\leq s\leq c_{0}\delta(z')$ with any fixed $c_{0}\geq1$, using the following change of variables
\begin{align}\label{change001}
\begin{cases}
x'-z'=sy',\\
x_{n}=sy_{n},
\end{cases}
\end{align}
we rescale $\Omega_{s}(z')$ into $Q_{1}$, where, for $0<r\leq 1$,
\begin{align*}
Q_{r}=\left\{y\in\mathbb{R}^{n}\,\Big|\,|y_{n}|<s^{-1}|\varepsilon/2+h(sy'+z')|,\;|y'|<r\right\}.
\end{align*}
Its top and bottom boundaries can be, respectively, written as
\begin{align*}
\widehat{\Gamma}^{+}_{r}=&\left\{y\in\mathbb{R}^{n}\,\Big|\,y_{n}=s^{-1}(\varepsilon/2+h(sy'+z')),\;|y'|<r\right\},
\end{align*}
and
\begin{align*}
\widehat{\Gamma}^{-}_{r}=&\left\{y\in\mathbb{R}^{n}\,\Big|\,y_{n}=-s^{-1}(\varepsilon/2+h(sy'+z')),\;|y'|<r\right\}.
\end{align*}
Let
\begin{align*}
\mathbf{W}_{1}^{\alpha}(y',y_n):=&\mathbf{w}_{1}^{\alpha}(s y'+z',s y_n),\quad Q_{1}^{\alpha}(y',y_n):=sq_{1}^{\alpha}(s y'+z',s y_n),\\
\bar{\mathbf{U}}_{1}^{\alpha}(y',y_n):=&\bar{\mathbf{u}}_{1}^{\alpha}(s y'+z',s y_n),\quad \bar{P}_{1}^{\alpha}(y',y_n):=s\bar{p}_{1}^{\alpha}(s y'+z',s y_n).
\end{align*}
Then using \eqref{change001}, $(\mathbf{W}_{1}^{\alpha}(y),Q_{1}^{\alpha}(y))$ solves
\begin{align}\label{MAZE001}
\begin{cases}
\nabla\cdot\sigma[\mathbf{W}_{1}^{\alpha},Q_{1}^{\alpha}-(Q_{1}^{\alpha})_{_{Q_{1}}}]=-\nabla\cdot\sigma[\bar{\mathbf{U}}_{1}^{\alpha},\bar{P}_{1}^{\alpha}],&
\hbox{in}\  Q_{1},  \\
\nabla\cdot\mathbf{W}_{1}^{\alpha}=0,&\hbox{in}\  Q_{1},\\
\mathbf{W}_{1}^{\alpha}=0, &\hbox{on} \ \widehat{\Gamma}^{\pm}_{1},
\end{cases}
\end{align}
where $(Q_{1}^{\alpha})_{_{Q_{1}}}:=\frac{1}{|Q_{1}|}\int_{Q_{1}}Q_{1}^{\alpha}$. Similar to \eqref{SPHE002}, applying Lemma \ref{lem002m} to equation \eqref{MAZE001} and rescaling back to original region $\Omega_{s}(z')$, we derive
\begin{align}\label{AWNZ001}
\int_{\Omega_{s}(z')}|q_{1}^{\alpha}-(q_{1}^{\alpha})_{_{s;z'}}|^{2}\leq C_{0}\int_{\Omega_{s}(z')}|\nabla\mathbf{w}_{1}^{\alpha}|^{2}+C\delta^{2}\int_{\Omega_{s}(z')}\left|\nabla\cdot\sigma[\bar{\mathbf{u}}_{1}^{\alpha},\bar{p}_{1}^{\alpha}]\right|^{2}.
\end{align}
Since $\nabla\cdot\mathbf{w}_{1}^{\alpha}=0$ in $\Omega_{2R_{0}}$, then it follows from \eqref{AWNZ001}, Young's inequality and Poincar\'{e} inequality that
\begin{align*}
&\frac{1}{\mu}\left|\int_{\Omega_{s}(z')}(q_{1}^{\alpha}-(q_{1}^{\alpha})_{_{s;z'}})\nabla\cdot(\mathbf{w}_{1}^{\alpha}\eta^{2})\right|\notag\\
&\leq\frac{1}{4C_{0}}\int_{\Omega_{s}(z')}|q_{1}^{\alpha}-(q_{1}^{\alpha})_{_{s;z'}}|^{2}\eta^{2}+\frac{C}{(s-t)^{2}}\int_{\Omega_{s}(z')}|\nabla\mathbf{w}_{1}^{\alpha}|^{2}\notag\\
&\leq\frac{1}{4}\int_{\Omega_{s}(z')}|\nabla\mathbf{w}_{1}^{\alpha}|^{2}+\left(\frac{C\delta}{s-t}\right)^{2}\int_{\Omega_{s}(z')}|\nabla\mathbf{w}_{1}^{\alpha}|^{2}+C\delta^{2}\int_{\Omega_{s}(z')}\left|\nabla\cdot\sigma[\bar{\mathbf{u}}_{1}^{\alpha},\bar{p}_{1}^{\alpha}]\right|^{2},
\end{align*}
which, in combination with \eqref{HN002}--\eqref{DAMZ991}, reads that
\begin{align}\label{HANQ001}
\int_{\Omega_{t}(z')}|\nabla\mathbf{w}_{1}^{\alpha}|^{2}\leq&\frac{1}{2}\int_{\Omega_{s}(z')}|\nabla\mathbf{w}_{1}^{\alpha}|^{2}+\left(\frac{C\delta}{s-t}\right)^{2}\int_{\Omega_{s}(z')}|\nabla\mathbf{w}_{1}^{\alpha}|^{2}\notag\\
&+C((s-t)^{2}+\delta^{2})\int_{\Omega_{s}(z')}\left|\nabla\cdot\sigma[\bar{\mathbf{u}}_{1}^{\alpha},\bar{p}_{1}^{\alpha}]\right|^{2}.
\end{align}
Then applying Lemma 4.3 in \cite{HL2011} to \eqref{HANQ001}, we obtain
\begin{align*}
\int_{\Omega_{t}(z')}|\nabla\mathbf{w}_{1}^{\alpha}|^{2}\leq&\left(\frac{C\delta}{s-t}\right)^{2}\int_{\Omega_{s}(z')}|\nabla\mathbf{w}_{1}^{\alpha}|^{2}+C((s-t)^{2}+\delta^{2})\int_{\Omega_{s}(z')}\left|\nabla\cdot\sigma[\bar{\mathbf{u}}_{1}^{\alpha},\bar{p}_{1}^{\alpha}]\right|^{2}.
\end{align*}
Utilizing \eqref{DM001}, we have
\begin{align*}
\int_{\Omega_{t}(z')}|\nabla\mathbf{w}_{1}^{\alpha}|^{2}\leq&\left(\frac{C_{1}\delta}{s-t}\right)^{2}\int_{\Omega_{s}(z')}|\nabla\mathbf{w}_{1}^{\alpha}|^{2}\notag\\
&+C((s-t)^{2}+\delta^{2})s^{n-1}\begin{cases}
\delta^{-1},&\alpha=1,...,n-1,\\
\delta^{-2},&\alpha=n,\\
\delta^{-1},&\alpha=n+1,...,2n-1,\\
1,&\alpha=2n,...,\frac{n(n+1)}{2},\,n\geq3.
\end{cases}
\end{align*}
Denote
$$F(t):=\int_{\Omega_{t}(z')}|\nabla\mathbf{w}_{1}^{\alpha}|^{2}.$$
Take $s=t_{i+1}$, $t=t_{i}$, $t_{i}=\delta+2C_{1}i\delta,\;i=0,1,2,...,\left[\frac{\vartheta(\kappa)}{4C_{1}\delta^{1/2}}\right]+1$. Then we have
\begin{align*}
F(t_{i})\leq&\frac{1}{4}F(t_{i+1})+C(i+1)^{n+1}\delta^{n}\begin{cases}
1,&\alpha=1,...,n-1,\\
\delta^{-1},&\alpha=n,\\
1,&\alpha=n+1,...,2n-1,\\
\delta,&\alpha=2n,...,\frac{n(n+1)}{2},\,n\geq3.
\end{cases}
\end{align*}
Therefore, by $\left[\frac{\vartheta(\kappa)}{4C_{1}\delta^{1/2}}\right]+1$ iterations, we have from \eqref{YGZA001} that for a arbitrarily small $\varepsilon>0$,
\begin{align*}
F(t_{0})\leq C\delta^{n}\begin{cases}
1,&\alpha=1,...,n-1,\\
\delta^{-1},&\alpha=n,\\
1,&\alpha=n+1,...,2n-1,\\
\delta,&\alpha=2n,...,\frac{n(n+1)}{2},\,n\geq3.
\end{cases}
\end{align*}
Then \eqref{HN001} is proved.

\noindent{\bf Step 3.}
Proof of
\begin{align}\label{AQ3.052}
&\|\nabla \mathbf{w}_{1}^{\alpha}\|_{L^{\infty}(\Omega_{\delta/2}(z'))}+\|q_{1}^{\alpha}-(q_{1}^{\alpha})_{\delta;z'}\|_{L^{\infty}(\Omega_{\delta/2}(z'))}\notag\\
&\leq C\begin{cases}
1,&\alpha=1,...,n-1,\\
\delta^{-1/2},&\alpha=n,\\
1,&\alpha=n+1,...,2n-1,\\
\delta^{1/2},&\alpha=2n,...,\frac{n(n+1)}{2},\,n\geq3,
\end{cases}\quad\mathrm{for}\;|z'|\leq R_{0}.
\end{align}

Similar to \eqref{AZQY001}, we obtain that for $x\in\Omega_{\delta}(z')$,
\begin{align*}
|\delta(x')-\delta(z')|
\leq&4\kappa\delta(\delta+|z'|)\leq4\sqrt{2\kappa}\delta^{3/2},
\end{align*}
which yields that
\begin{align*}
\left|\frac{\delta(x')}{\delta(z')}-1\right|\leq8\sqrt{2}\kappa R_{0}.
\end{align*}
Due to the fact that $R$ is a small positive constant, $Q_{1}$ is of approximate unit size. Therefore, applying Lemma \ref{lem002m}, the Sobolev embedding theorem, the Poincar\'{e} inequality, a bootstrap argument and classical $W^{2, p}$ estimates (see Theorem IV.5.1 in \cite{G2011}) for Stokes flow \eqref{MAZE001}, we derive that for some $p>n$,
\begin{align*}
&\|\nabla \mathbf{W}_{1}^{\alpha}\|_{L^{\infty}(Q_{1/2})}+\|Q_{1}^{\alpha}-(Q_{1}^{\alpha})_{_{Q_{1}}}\|_{L^{\infty}(Q_{1/2})}\notag\\
&\leq C(\|\mathbf{W}_{\alpha}\|_{W^{2,p}(Q_{1/2})}+\|Q_{1}^{\alpha}-(Q_{1}^{\alpha})_{_{Q_{1}}}\|_{W^{1,p}(Q_{1/2})})\notag\\
&\leq C(\|\nabla\mathbf{W}_{1}^{\alpha}\|_{L^{2}(Q_{1})}+\|Q_{1}^{\alpha}-(Q_{1}^{\alpha})_{_{Q_{1}}}\|_{L^{2}(Q_{1})}+\|\nabla\cdot\sigma[\bar{\mathbf{U}}_{1}^{\alpha},\bar{P}_{1}^{\alpha}]\|_{L^{\infty}(Q_{1})})\notag\\
&\leq C(\|\nabla\mathbf{W}_{1}^{\alpha}\|_{L^{2}(Q_{1})}+\|\nabla\cdot\sigma[\bar{\mathbf{U}}_{1}^{\alpha},\bar{P}_{1}^{\alpha}]\|_{L^{\infty}(Q_{1})}).
\end{align*}
This yields that
\begin{align*}
&\|\nabla \mathbf{w}_{1}^{\alpha}\|_{L^{\infty}(\Omega_{\delta/2}(z'))}+\|q_{1}^{\alpha}-(q_{1}^{\alpha})_{\delta;z'}\|_{L^{\infty}(\Omega_{\delta/2}(z'))}\notag\\
&\leq C\left(\delta^{-\frac{n}{2}}\|\nabla \mathbf{w}_{1}^{\alpha}\|_{L^{2}(\Omega_{\delta}(z'))}+\delta\|\nabla\cdot\sigma[\bar{\mathbf{u}}_{1}^{\alpha},\bar{p}_{1}^{\alpha}]\|_{L^{\infty}(\Omega_{\delta}(z'))}\right).
\end{align*}
which, together with \eqref{DM001} and \eqref{HN001}, leads to that \eqref{AQ3.052} holds. The proof is complete.

\end{proof}

Applying the proof of Proposition \ref{thm86} with minor modification, we have
\begin{corollary}\label{coro00z}
Let $(\mathbf{u}_{0},p_{0}),\,(\mathbf{u}_{i}^{\alpha},p_{i}^{\alpha}),\,(\mathbf{u}_{0}^{\ast},p_{0}^{\ast})$ and $(\mathbf{u}_{i}^{\ast\alpha},p_{i}^{\ast\alpha})$, $i=1,2$, $\alpha=1,2,...,\frac{n(n+1)}{2}$ be, respectively, the solutions of \eqref{qaz001}--\eqref{qaz003az} and \eqref{ZG001}--\eqref{qaz001111}. Then for a arbitrarily small $\varepsilon>0$,
\begin{align*}
|\nabla\mathbf{u}_{0}|+\left|\nabla(\mathbf{u}_{1}^{\alpha}+\mathbf{u}_{2}^{\alpha})\right|+|p_{0}|+|p_{1}^{\alpha}+p_{2}^{\alpha}|\leq C\delta^{-\frac{n}{2}}e^{-\frac{1}{2C\delta^{1/2}}},\;\;\mathrm{in}\;\Omega_{R_{0}},
\end{align*}
and
\begin{align}\label{LFN001}
|\nabla\mathbf{u}_{0}^{\ast}|+\left|\nabla(\mathbf{u}_{1}^{\ast\alpha}+\mathbf{u}_{2}^{\ast\alpha})\right|+|p_{0}^{\ast}|+|p_{1}^{\ast\alpha}+p_{2}^{\ast\alpha}|\leq C|x'|^{-n}e^{-\frac{1}{2C|x'|}},\;\;\mathrm{in}\;\Omega_{R_{0}}^{\ast}.
\end{align}
\end{corollary}

For $i,j=1,2$ and $\alpha, \beta=1,2,...,\frac{n(n+1)}{2}$, write
\begin{align}\label{KAT001}
a_{ij}^{\alpha\beta}:=-\int_{\partial{D}_{j}}\boldsymbol{\psi}_{\beta}\cdot\sigma[\mathbf{u}_{i}^{\alpha},p_{i}^{\alpha}]\nu, \quad b_j^{\beta}:=\int_{\partial D_{j}}\psi_{\beta}\cdot\sigma[\mathbf{u}_{0},p_{0}]\nu.
\end{align}
Integrating by parts, we have
\begin{align*}
a_{ij}^{\alpha\beta}=\int_{\Omega}(2\mu e(\mathbf{u}_{i}^{\alpha}),e(\mathbf{u}_{j}^{\beta})).
\end{align*}
Substituting \eqref{Decom} into the fourth line of \eqref{La.002}, we derive
\begin{align}\label{AHNTW009}
\begin{cases}
\sum\limits_{\alpha=1}^{\frac{n(n+1)}{2}}(C_{1}^\alpha-C_{2}^{\alpha}) a_{11}^{\alpha\beta}+\sum\limits_{\alpha=1}^{\frac{n(n+1)}{2}}C_{2}^\alpha \sum\limits^{2}_{i=1}a_{i1}^{\alpha\beta}=b_1^\beta,\\
\sum\limits_{\alpha=1}^{\frac{n(n+1)}{2}}(C_{1}^\alpha-C_{2}^{\alpha}) a_{12}^{\alpha\beta}+\sum\limits_{\alpha=1}^{\frac{n(n+1)}{2}}C_{2}^\alpha \sum\limits^{2}_{i=1}a_{i2}^{\alpha\beta}=b_2^\beta.
\end{cases}
\end{align}
Adding the first line of \eqref{AHNTW009} to its second line, we have
\begin{align}\label{zzw002}
\begin{cases}
\sum\limits_{\alpha=1}^{\frac{n(n+1)}{2}}(C_{1}^\alpha-C_{2}^{\alpha}) a_{11}^{\alpha\beta}+\sum\limits_{\alpha=1}^{\frac{n(n+1)}{2}}C_{2}^\alpha \sum\limits^{2}_{i=1}a_{i1}^{\alpha\beta}=b_1^\beta,\\
\sum\limits_{\alpha=1}^{\frac{n(n+1)}{2}}(C_{1}^{\alpha}-C_{2}^{\alpha})\sum\limits^{2}_{j=1}a_{1j}^{\alpha\beta}+\sum\limits_{\alpha=1}^{\frac{n(n+1)}{2}}C_{2}^\alpha \sum\limits^{2}_{i,j=1}a_{ij}^{\alpha\beta}=\sum\limits^{2}_{i=1}b_{i}^{\beta}.
\end{cases}
\end{align}
For notational simplicity, let
\begin{align*}
&X^{1}=\big(C_{1}^1-C_{2}^{1},...,C_{1}^\frac{n(n+1)}{2}-C_{2}^{\frac{n(n+1)}{2}}\big)^{T},\quad X^{2}=\big(C_{2}^{1},...,C_{2}^{\frac{n(n+1)}{2}}\big)^{T},\\
&Y^{1}=(b_1^1,...,b_{1}^{\frac{n(n+1)}{2}})^{T},\quad Y^{2}=\bigg(\sum\limits^{2}_{i=1}b_{i}^{1},...,\sum\limits^{2}_{i=1}b_{i}^{\frac{n(n+1)}{2}}\bigg)^{T},\notag
\end{align*}
and
\begin{align}
&\mathbb{A}=(a_{11}^{\alpha\beta})_{\frac{n(n+1)}{2}\times\frac{n(n+1)}{2}},\quad \mathbb{B}=\bigg(\sum\limits^{2}_{i=1}a_{i1}^{\alpha\beta}\bigg)_{\frac{n(n+1)}{2}\times\frac{n(n+1)}{2}},\label{GGDA01}\\
&\mathbb{C}=\bigg(\sum\limits^{2}_{j=1}a_{1j}^{\alpha\beta}\bigg)_{\frac{n(n+1)}{2}\times\frac{n(n+1)}{2}},\quad \mathbb{D}=\bigg(\sum\limits^{2}_{i,j=1}a_{ij}^{\alpha\beta}\bigg)_{\frac{n(n+1)}{2}\times\frac{n(n+1)}{2}}.\label{FAMZ91}
\end{align}
Therefore, \eqref{zzw002} becomes
\begin{gather}\label{PLA001}
\begin{pmatrix} \mathbb{A}&\mathbb{B} \\  \mathbb{C}&\mathbb{D}
\end{pmatrix}
\begin{pmatrix}
X^{1}\\
X^{2}
\end{pmatrix}=
\begin{pmatrix}
Y^{1}\\
Y^{2}
\end{pmatrix}.
\end{gather}
The advantage of \eqref{zzw002} or \eqref{PLA001} lies in that it makes the singularities of the coefficient matrix in \eqref{PLA001} focus only on the first block matrix $\mathbb{A}$, but the remaining matrices $\mathbb{B},\mathbb{C},\mathbb{D}$ remain bounded. This greatly simplifies the calculations for solving $X=(X^{1},X^{2})^{T}$. Observe that $a_{ij}^{\alpha\beta}=a_{ji}^{\beta\alpha}$, which implies that $\mathbb{C}=\mathbb{B}^{T}$.

For $\alpha=1,2,...,\frac{n(n+1)}{2}$, define
\begin{align}\label{CONT001}
\mathcal{L}_{\alpha}:=&
\begin{cases}
\mu,&\alpha=1,...,n-1,\\
2\mu,&\alpha=n,n+1,...,\frac{n(n+1)}{2}.
\end{cases}
\end{align}
Denote
\begin{align}\label{rate00}
\rho_{n}(\varepsilon):=&
\begin{cases}
\varepsilon^{-1/2},&n=2,\\
|\ln\varepsilon|,&n=3,\\
1,&n>3,
\end{cases}
\end{align}
and
\begin{align}\label{rate005}
\varrho_{\alpha,n}(\varepsilon):=
\begin{cases}
\varepsilon^{\frac{n-1}{24}},&\alpha=1,...,n-1,\,n\geq2,\\
|\ln\varepsilon|,&\alpha=n,\,n=2,\\
\varepsilon^{\frac{n-2}{24}},&\alpha=n,\,n\geq3.
\end{cases}
\end{align}
Then we have
\begin{lemma}\label{lemmabc}
For a arbitrarily small $\varepsilon>0$, we have

$(\rm{i})$ for $\alpha=1,2,...,n$,
\begin{align}\label{LMC}
a_{11}^{\alpha\alpha}=&
\begin{cases}
\mathcal{L}_{\alpha}\left(\frac{\pi}{2\kappa}\right)^{\frac{n-1}{2}}\rho_{n}(\varepsilon)+\mathcal{G}_{n}^{\ast\alpha}+O(1)\max\{\varepsilon^{\frac{1}{12}},\varrho_{\alpha,n}(\varepsilon)\},&n=2,3,\\
a_{11}^{\ast\alpha\alpha}+O(1)\varepsilon^{\min\{\frac{1}{12},\frac{n-3}{24}\}},&n>3;
\end{cases}
\end{align}

$(\rm{ii})$ for $\alpha=n+1,...,\frac{n(n+1)}{2}$,
\begin{align}\label{LMC1}
a_{11}^{\alpha\alpha}=a_{11}^{\ast\alpha\alpha}+O(1)\varepsilon^{\min\{\frac{1}{12},\frac{n-1}{24}\}};
\end{align}

$(\rm{iii})$ for $\alpha,\beta=1,2,...,\frac{n(n+1)}{2}$,
\begin{align}\label{M001}
a_{11}^{12}=a_{11}^{21}=O(1)|\ln\varepsilon|,\quad n=2,
\end{align}
and
\begin{align}\label{M005}
&a_{11}^{\alpha\beta}-a_{11}^{\ast\alpha\beta}=a_{11}^{\beta\alpha}-a_{11}^{\ast\alpha\beta}\notag\\
&=O(1)
\begin{cases}
\varepsilon^{\min\{\frac{1}{12},\frac{n-2}{24}\}},&\alpha,\beta=1,2,...,n,\,n\geq3,\,\alpha<\beta,\\
\varepsilon^{\min\{\frac{1}{12},\frac{n-1}{24}\}},&\alpha=1,2,...,n,\,\beta=n+1,...,\frac{n(n+1)}{2},\,n\geq2,\\
\varepsilon^{\min\{\frac{1}{12},\frac{n}{24}\}},&\alpha,\beta=n+1,...,\frac{n(n+1)}{2},\,n\geq3,\,\alpha<\beta;
\end{cases}
\end{align}

$(\rm{iv})$ for $\alpha,\beta=1,2,...,\frac{n(n+1)}{2}$,
\begin{align}\label{AZQ001}
\sum\limits^{2}_{i=1}a_{i1}^{\alpha\beta}=&\sum\limits^{2}_{i=1}a_{i1}^{\ast\alpha\beta}+O(\varepsilon^{\frac{1}{4}}),\quad\sum\limits^{2}_{j=1}a_{1j}^{\alpha\beta}=\sum\limits^{2}_{j=1}a_{1j}^{*\alpha\beta}+O(\varepsilon^{\frac{1}{4}}),
\end{align}
and thus
\begin{align*}
\sum\limits^{2}_{i,j=1}a_{ij}^{\alpha\beta}=\sum\limits^{2}_{i,j=1}a_{ij}^{\ast\alpha\beta}+O(\varepsilon^{\frac{1}{4}}).
\end{align*}
\end{lemma}
\begin{remark}
As seen in Lemma \ref{lemmabc}, we give a precise computation for every element of the coefficient matrix in \eqref{PLA001}, which ensures that the values of $C_{1}^{\alpha}-C_{2}^{\alpha}$ and $C_{2}^{\alpha}$, $\alpha=1,2,...,\frac{n(n+1)}{2}$ can be explicitly solved for the general domain and any boundary data in all dimensions. This improves the corresponding results in \cite{LX2022}.
\end{remark}

\begin{proof}
\noindent{\bf Step 1.} Proofs of \eqref{LMC}--\eqref{LMC1}. For $\alpha=1,2,...,\frac{n(n+1)}{2}$, denote $\bar{\mathbf{u}}_{1}^{\ast\alpha}:=\bar{\mathbf{u}}_{1}^{\alpha}|_{\varepsilon=0}$, where $\bar{\mathbf{u}}_{1}^{\alpha}$ is defined by \eqref{zzwz002}. Applying Proposition \ref{thm86} to $\bar{\mathbf{u}}_{1}^{\ast\alpha}$, we obtain that for $x\in\Omega_{R_{0}}^{\ast}$,
\begin{align}\label{GZIJ}
\nabla\mathbf{u}_{1}^{\ast\alpha}=&\nabla\bar{\mathbf{u}}_{1}^{\ast\alpha}+O(1)
\begin{cases}
1,&\alpha=1,...,n-1,\\
|x'|^{-1},&\alpha=n,\\
1,&\alpha=n+1,...,2n-1,\\
|x'|,&\alpha=2n,...,\frac{n(n+1)}{2},\,n\geq3.
\end{cases}
\end{align}
A direct computation shows that for $x\in\Omega_{R_{0}}^{\ast}$,
\begin{align}\label{LKT6.003}
|\partial_{x_{n}}(\bar{\mathbf{u}}_{1}^{\alpha}-\bar{\mathbf{u}}_{1}^{\ast\alpha})|\leq&
\begin{cases}
\frac{C\varepsilon}{|x'|^{2}(\varepsilon+|x'|^{2})},&\alpha=1,2,...,n,\\
\frac{C\varepsilon}{|x'|(\varepsilon+|x'|^{2})},&\alpha=n+1,...,\frac{n(n+1)}{2}.
\end{cases}
\end{align}
For $0<t<R_{0}$, define
\begin{align*}
\mathcal{C}_{t}:=\left\{x\in\mathbb{R}^{n}\Big|\;-\varepsilon-2\max_{|x'|\leq t}h(x')\leq x_{n}\leq\varepsilon+2\max_{|x'|\leq t}h(x'),\;|x'|<t\right\}.
\end{align*}

For $\alpha=1,2,...,\frac{n(n+1)}{2}$, by linearity, $(\mathbf{u}_{1}^{\alpha}-\mathbf{u}_{1}^{\ast\alpha},p_{1}^{\alpha}-p_{1}^{\ast\alpha})$ solves
\begin{align*}
\begin{cases}
\nabla\cdot\sigma[\mathbf{u}_{1}^{\alpha}-\mathbf{u}_{1}^{\ast\alpha},p_{1}^{\alpha}-p_{1}^{\ast\alpha}]=0,&\mathrm{in}\;\,D\setminus(\overline{D_{1}\cup D_{2}\cup D_{1}^{\ast}\cup D_{2}^{\ast}}),\\
\nabla\cdot(\mathbf{u}_{1}^{\alpha}-\mathbf{u}_{1}^{\ast\alpha})=0,&\mathrm{in}\;\,D\setminus(\overline{D_{1}\cup D_{2}\cup D_{1}^{\ast}\cup D_{2}^{\ast}}),\\
\mathbf{u}_{1}^{\alpha}-\mathbf{u}_{1}^{\ast\alpha}=\boldsymbol{\psi}_{\alpha}-\mathbf{u}_{1}^{\ast\alpha},&\mathrm{on}\;\,\partial D_{1}\setminus D_{1}^{\ast},\\
\mathbf{u}_{1}^{\alpha}-\mathbf{u}_{1}^{\ast\alpha}=-\mathbf{u}_{1}^{\ast\alpha},&\mathrm{on}\;\,\partial D_{2}\setminus D_{2}^{\ast},\\
\mathbf{u}_{1}^{\alpha}-\mathbf{u}_{1}^{\ast\alpha}=\mathbf{u}_{1}^{\alpha}-\boldsymbol{\psi}_{\alpha},&\mathrm{on}\;\,\partial D_{1}^{\ast}\setminus(D_{1}\cup\{0\}),\\
\mathbf{u}_{1}^{\alpha}-\mathbf{u}_{1}^{\ast\alpha}=\mathbf{u}_{1}^{\alpha},&\mathrm{on}\;\,\partial D_{2}^{\ast}\setminus(D_{2}\cup\{0\}),\\
\mathbf{u}_{1}^{\alpha}-\mathbf{u}_{1}^{\ast\alpha}=0,&\mathrm{on}\;\,\partial D.
\end{cases}
\end{align*}
By the standard boundary and interior estimates of Stokes flow (see e.g. \cite{K1995,MR2009}), we have
\begin{align*}
|\partial_{x_{n}}\mathbf{u}_{1}^{\ast\alpha}|\leq C, \quad\mathrm{in}\;\Omega^{\ast}\setminus\Omega^{\ast}_{R_{0}},
\end{align*}
which reads that for $x\in\partial D_{i}\setminus D_{i}^{\ast}$, $i=1,2$,
\begin{align}\label{LKT6.007}
|(\mathbf{u}_{1}^{\alpha}-\mathbf{u}_{1}^{\ast\alpha})(x',x_{n})|=|\mathbf{u}_{1}^{\ast\alpha}(x',x_{n}+(-1)^{i}\varepsilon/2)-\mathbf{u}_{1}^{\ast\alpha}(x',x_{n})|\leq C\varepsilon.
\end{align}
From \eqref{Le2.025}, we derive that for $x\in\partial D_{i}^{\ast}\setminus(D_{i}\cup\mathcal{C}_{\varepsilon^{\theta}})$, $i=1,2,$
\begin{align}\label{LKT6.008}
|(\mathbf{u}_{1}^{\alpha}-\mathbf{u}_{1}^{\ast\alpha})(x',x_{n})|=&|\mathbf{u}_{1}^{\alpha}(x',x_{n})-\mathbf{u}_{1}^{\alpha}(x',x_{n}+(-1)^{i-1}\varepsilon/2)|\notag\\
\leq& C
\begin{cases}
\varepsilon^{1-2\theta},&\alpha=1,2,...,n,\\
\varepsilon^{1-\theta},&\alpha=n+1,...,\frac{n(n+1)}{2},
\end{cases}
\end{align}
where $0<\theta<\frac{1}{2}$ to be chosen below. Combining \eqref{Le2.025} and \eqref{GZIJ}--\eqref{LKT6.003}, we deduce that for $x\in\Omega_{R_{0}}^{\ast}\cap\{|x'|=\varepsilon^{\theta}\}$,
\begin{align*}
|\partial_{x_{n}}(\mathbf{u}_{1}^{\alpha}-\mathbf{u}_{1}^{\ast\alpha})|\leq&|\partial_{x_{n}}(\mathbf{u}_{1}^{\alpha}-\bar{\mathbf{u}}_{1}^{\alpha})|+|\partial_{x_{n}}(\bar{\mathbf{u}}_{1}^{\alpha}-\bar{\mathbf{u}}_{1}^{\ast\alpha}|+|\partial_{x_{n}}(\mathbf{u}_{1}^{\ast\alpha}-\bar{\mathbf{u}}_{1}^{\ast\alpha})|\notag\\
\leq&C
\begin{cases}
1+\varepsilon^{1-4\theta},&\alpha=1,...,n-1,\\
\varepsilon^{-\theta}+\varepsilon^{1-4\theta},&\alpha=n,\\
1+\varepsilon^{1-3\theta},&\alpha=n+1,...,2n-1,\\
\varepsilon^{\theta}+\varepsilon^{1-3\theta},&\alpha=2n,...,\frac{n(n+1)}{2},\,n\geq3.
\end{cases}
\end{align*}
This, together with \eqref{LKT6.008}, gives that for $x\in\Omega_{R_{0}}^{\ast}\cap\{|x'|=\varepsilon^{\theta}\}$,
\begin{align}\label{LKT6.009}
&|(\mathbf{u}_{1}^{\alpha}-\mathbf{u}_{1}^{\ast\alpha})(x',x_{n})|\notag\\
&\leq|(\mathbf{u}_{1}^{\alpha}-\mathbf{u}_{1}^{\ast\alpha})(x',x_{n})-(\mathbf{u}_{1}^{\alpha}-\mathbf{u}_{1}^{\ast\alpha})(x',-h(x'))|+|(\mathbf{u}_{1}^{\alpha}-\mathbf{u}_{1}^{\ast\alpha})(x',-h(x'))|\notag\\
&\leq \sup\limits_{|x_{n}|\leq h(x'),|x'|=\varepsilon^{\theta}}2h(x')|\partial_{n}(\mathbf{u}_{1}^{\alpha}-\mathbf{u}_{1}^{\ast\alpha})|+ C
\begin{cases}
\varepsilon^{1-2\theta},&\alpha=1,2,...,n,\\
\varepsilon^{1-\theta},&\alpha=n+1,...,\frac{n(n+1)}{2}
\end{cases}\notag\\
&\leq C
\begin{cases}
\varepsilon^{2\theta}+\varepsilon^{1-2\theta},&\alpha=1,...,n-1,\\
\varepsilon^{\theta}+\varepsilon^{1-2\theta},&\alpha=n,\\
\varepsilon^{2\theta}+\varepsilon^{1-\theta},&\alpha=n+1,...,2n-1,\\
\varepsilon^{3\theta}+\varepsilon^{1-\theta},&\alpha=2n,...,\frac{n(n+1)}{2},\,n\geq3.
\end{cases}
\end{align}
Take $\theta=\frac{1}{3}$. Then combining \eqref{LKT6.007}--\eqref{LKT6.009}, we have
$$|\mathbf{u}_{1}^{\alpha}-\mathbf{u}_{1}^{\ast\alpha}|\leq C\varepsilon^{\frac{1}{3}},\quad\;\,\mathrm{on}\;\,\partial\big(D\setminus\big(\overline{D_{1}\cup D_{2}\cup D_{1}^{\ast}\cup D_{2}^{\ast}\cup\mathcal{C}_{\varepsilon^{\frac{1}{3}}}}\big)\big).$$
Making use of the maximum modulus theorem for Stokes systems in \cite{L1969}, we obtain that for $\alpha=1,2,...,\frac{n(n+1)}{2}$,
\begin{align}\label{AST123}
|\mathbf{u}_{1}^{\alpha}-\mathbf{u}_{1}^{\ast\alpha}|\leq C\varepsilon^{\frac{1}{3}},\quad\;\,\mathrm{in}\;\,D\setminus\big(\overline{D_{1}\cup D_{2}\cup D_{1}^{\ast}\cup D_{2}^{\ast}\cup\mathcal{C}_{\varepsilon^{\frac{1}{3}}}}\big).
\end{align}
For any fixed $\varepsilon^{\frac{1}{24}}\leq|z'|\leq R_{0}$, under the following change of variable
\begin{align*}
\begin{cases}
x'-z'=|z'|^{2}y',\\
x_{n}=|z'|^{2}y_{n},
\end{cases}
\end{align*}
$\Omega_{|z'|+|z'|^{2}}\setminus\Omega_{|z'|}$ and $\Omega_{|z'|+|z'|^{2}}^{\ast}\setminus\Omega_{|z'|}^{\ast}$ become two squares (or cylinders) $Q_{1}$ and $Q_{1}^{\ast}$ of nearly unit-size, respectively. Consider
\begin{align*}
\mathbf{U}_{1}^{\alpha}(y):=\mathbf{u}_{1}^{\alpha}(z'+|z'|^{2}y',|z'|^{2}y_{n}),\quad\mathrm{in}\;Q_{1},
\end{align*}
and
\begin{align*}
\mathbf{U}_{1}^{\ast\alpha}(y):=\mathbf{u}_{1}^{\ast\alpha}(z'+|z'|^{2}y',|z'|^{2}y_{n}),\quad\mathrm{in}\;Q_{1}^{\ast}.
\end{align*}
Applying the maximum modulus theorem and the standard boundary and interior estimates for Stokes flow that
\begin{align}\label{AZQNZ001}
|\nabla^{2}\mathbf{U}_{1}^{\alpha}|\leq C,\quad\mathrm{in}\;Q_{1},\quad\mathrm{and}\;|\nabla^{2}\mathbf{U}_{1}^{\ast\alpha}|\leq C,\quad\mathrm{in}\;Q_{1}^{\ast}.
\end{align}
By an interpolation with \eqref{AST123}--\eqref{AZQNZ001}, we have
\begin{align*}
|\nabla(\mathbf{U}_{1}^{\alpha}-\mathbf{U}_{1}^{\ast\alpha})|\leq C\varepsilon^{\frac{1}{3}(1-\frac{1}{2})}\leq C\varepsilon^{\frac{1}{6}}.
\end{align*}
Rescaling back to $\mathbf{u}_{1}^{\alpha}-\mathbf{u}_{1}^{\ast\alpha}$, we deduce that for $\varepsilon^{\frac{1}{24}}\leq|z'|\leq R_{0}$,
\begin{align*}
|\nabla(\mathbf{u}_{1}^{\alpha}-\mathbf{u}_{1}^{\ast\alpha})(x)|\leq C\varepsilon^{\frac{1}{6}}|z'|^{-2}\leq C\varepsilon^{\frac{1}{12}},\quad x\in\Omega^{\ast}_{|z'|+|z'|^{2}}\setminus\Omega_{|z'|}^{\ast}.
\end{align*}
Combining the above-mentioned facts, we obtain that for $\alpha=1,2,...,\frac{n(n+1)}{2}$,
\begin{align}\label{con035}
|\nabla(\mathbf{u}_{1}^{\alpha}-\mathbf{u}_{1}^{\ast\alpha})|\leq C\varepsilon^{\frac{1}{12}},\quad\;\,\mathrm{in}\;\,D\setminus\big(\overline{D_{1}\cup D_{2}\cup D_{1}^{\ast}\cup D_{2}^{\ast}\cup\mathcal{C}_{\varepsilon^{\frac{1}{24}}}}\big).
\end{align}

Pick $\bar{\theta}=\frac{1}{24}$. For $\alpha=1,2,...,\frac{n(n+1)}{2}$, we decompose $a_{11}^{\alpha\alpha}$ into the following three subparts:
\begin{align*}
\mathrm{I}=&\int_{\Omega_{\varepsilon^{\bar{\theta}}}}(2\mu e(\mathbf{u}_{1}^{\alpha}),e(\mathbf{u}_{1}^{\alpha})),\quad\mathrm{II}=\int_{\Omega\setminus\Omega_{R_{0}}}(2\mu e(\mathbf{u}_{1}^{\alpha}),e(\mathbf{u}_{1}^{\alpha})),\\
\mathrm{III}=&\int_{\Omega_{R_{0}}\setminus\Omega_{\varepsilon^{\bar{\theta}}}}(2\mu e(\mathbf{u}_{1}^{\alpha}),e(\mathbf{u}_{1}^{\alpha})).
\end{align*}
For the convenience of presentation, define
\begin{align}\label{MAINLDING001}
\bar{v}=\frac{x_{n}}{\varepsilon+2h(x')}+\frac{1}{2},\;\,\mathrm{in}\;\Omega_{R_{0}},\quad\bar{v}^{\ast}=\frac{x_{n}}{2h(x')}+\frac{1}{2},\;\,\mathrm{in}\;\Omega^{\ast}_{R_{0}}.
\end{align}
Then we have $\bar{v}=\mathfrak{G}+\frac{1}{2}$ in $\Omega_{R_{0}}$ with $\varepsilon>0$ and $\bar{v}^{\ast}=(\mathfrak{G}+\frac{1}{2})|_{\varepsilon=0}$ in $\Omega_{R_{0}}^{\ast}$. Recall that when $\alpha=1,2,...,n$, $\boldsymbol{\psi}_{\alpha}=\mathbf{e}_{\alpha}$; when $\alpha=n+1,...,\frac{n(n+1)}{2}$, there exist two indices $1\leq i_{\alpha}<j_{\alpha}\leq n$ such that
$\boldsymbol{\psi}_{\alpha}=x_{j_{\alpha}}\mathbf{e}_{i_{\alpha}}-x_{i_{\alpha}}\mathbf{e}_{j_{\alpha}}$. Especially when $n+1\leq\alpha\leq2n-1$, we have $i_{\alpha}=\alpha-n,\,j_{\alpha}=n$. A straightforward computation shows that
\begin{align*}
(2\mu e(\bar{\mathbf{u}}_{1}^{\alpha}),e(\bar{\mathbf{u}}_{1}^{\alpha}))=&
\begin{cases}
\mu\big[(\partial_{x_{\alpha}}\bar{v})^{2}+\sum\limits^{n}_{i=1}(\partial_{x_{i}}\bar{v})^{2}\big]+(2\mu e(\boldsymbol{\mathcal{F}}_{\alpha}),e(\boldsymbol{\mathcal{F}}_{\alpha}))\\
+2(2\mu e(\boldsymbol{\psi}_{\alpha}\bar{v}),e(\boldsymbol{\mathcal{F}}_{\alpha})),\;\,\alpha=1,2,...,n,\\
\mu(x_{i_{\alpha}}^{2}+x_{j_{\alpha}}^{2})\sum\limits^{n}_{k=1}(\partial_{x_{k}}\bar{v})^{2}+\mu(x_{j_{\alpha}}\partial_{x_{i_{\alpha}}}\bar{v}-x_{i_{\alpha}}\partial_{x_{j_{\alpha}}}\bar{v})^{2}\notag\\
+2(2\mu e(\boldsymbol{\psi}_{\alpha}\bar{v}),e(\boldsymbol{\mathcal{F}}_{\alpha}))+(2\mu e(\boldsymbol{\mathcal{F}}_{\alpha}),e(\boldsymbol{\mathcal{F}}_{\alpha})),\;\,\alpha=n+1,...,\frac{n(n+1)}{2},
\end{cases}
\end{align*}
where $\boldsymbol{\mathcal{F}}_{\alpha}$, $\alpha=1,2,...,\frac{n(n+1)}{2}$ are given by \eqref{QLA001} and $\bar{v}$ is defined by \eqref{MAINLDING001}. A consequence of Proposition \ref{thm86} yields that
\begin{align}\label{con03365}
\mathrm{I}=&\int_{\Omega_{\varepsilon^{\bar{\theta}}}}(2\mu e(\bar{\mathbf{u}}_{1}^{\alpha}),e(\bar{\mathbf{u}}_{1}^{\alpha}))+2\int_{\Omega_{\varepsilon^{\bar{\theta}}}}(2\mu e(\mathbf{u}_{1}^{\alpha}-\bar{\mathbf{u}}_{1}^{\alpha}),e(\bar{\mathbf{u}}_{1}^{\alpha}))\notag\\
&+\int_{\Omega_{\varepsilon^{\bar{\theta}}}}(2\mu e(\mathbf{u}_{1}^{\alpha}-\bar{\mathbf{u}}_{1}^{\alpha}),e(\mathbf{u}_{1}^{\alpha}-\bar{\mathbf{u}}_{1}^{\alpha}))\notag\\
=&
\begin{cases}
\mathcal{L}_{\alpha}\int_{|x'|<\varepsilon^{\bar{\theta}}}\frac{1}{\varepsilon+2h(x')}+O(1)\varrho_{\alpha,n}(\varepsilon),&\alpha=1,...,n,\\
\frac{\mathcal{L}_{\alpha}}{n-1}\int_{|x'|<\varepsilon^{\bar{\theta}}}\frac{|x'|^{2}}{\varepsilon+2h(x')}+O(1)\varepsilon^{\frac{n}{24}},&\alpha=n+1,...,\frac{n(n+1)}{2},
\end{cases}
\end{align}
where $\mathcal{L}_{\alpha}$, $\alpha=1,2,...,\frac{n(n+1)}{2}$ and $\varrho_{\alpha,n}(\varepsilon)$ are, respectively, defined by \eqref{CONT001} and \eqref{rate005}.

With regard to the second term $\mathrm{II}$, since $|\nabla \mathbf{u}_{1}^{\alpha}|$ remains bounded in $(D_{1}^{\ast}\cup D_{2}^{\ast})\setminus(D_{1}\cup D_{2}\cup\Omega_{R_{0}})$ and $(D_{1}\cup D_{2})\setminus (D_{1}^{\ast}\cup D_{2}^{\ast})$ with their volume of order $O(\varepsilon)$, it then follows from \eqref{con035} that
\begin{align}\label{KKAA123}
\mathrm{II}=&\int_{D\setminus(D_{1}\cup D_{2}\cup D_{1}^{\ast}\cup D_{2}^{\ast}\cup\Omega_{R_{0}})}(2\mu e(\mathbf{u}_{1}^{\alpha}),e(\mathbf{u}_{1}^{\alpha}))+O(1)\varepsilon\notag\\
=&\int_{D\setminus(D_{1}\cup D_{2}\cup D_{1}^{\ast}\cup D_{2}^{\ast}\cup\Omega_{R_{0}})}\left((2\mu e(\mathbf{u}_{1}^{\ast\alpha}),e(\mathbf{u}_{1}^{\ast\alpha}))+2(2\mu e(\mathbf{u}_{1}^{\alpha}-\mathbf{u}_{1}^{\ast\alpha}),e(\mathbf{u}_{1}^{\ast\alpha}))\right)\notag\\
&+\int_{D\setminus(D_{1}\cup D_{2}\cup D_{1}^{\ast}\cup D_{2}^{\ast}\cup\Omega_{R_{0}})}(2\mu e(\mathbf{u}_{1}^{\alpha}-\mathbf{u}_{1}^{\ast\alpha}),e(\mathbf{u}_{1}^{\alpha}-\mathbf{u}_{1}^{\ast\alpha}))+O(1)\varepsilon\notag\\
=&\int_{\Omega^{\ast}\setminus\Omega^{\ast}_{R_{0}}}(2\mu e(\mathbf{u}_{1}^{\ast\alpha}),e(\mathbf{u}_{1}^{\ast\alpha}))+O(1)\varepsilon^{\frac{1}{12}}.
\end{align}

As for $\mathrm{III}$, it can be further split as follows:
\begin{align*}
\mathrm{III}_{1}=&\int_{\Omega^{\ast}_{R_{0}}\setminus\Omega^{\ast}_{\varepsilon^{\bar{\theta}}}}(2\mu e(\mathbf{u}_{1}^{\alpha}-\mathbf{u}_{1}^{\ast\alpha}),e(\mathbf{u}_{1}^{\alpha}-\mathbf{u}_{1}^{\ast\alpha}))+2\int_{\Omega^{\ast}_{R_{0}}\setminus\Omega^{\ast}_{\varepsilon^{\bar{\theta}}}}(2\mu e(\mathbf{u}_{1}^{\alpha}-\mathbf{u}_{1}^{\ast\alpha}),e(\mathbf{u}_{1}^{\ast\alpha})),\\
\mathrm{III}_{2}=&\int_{(\Omega_{R_{0}}\setminus\Omega_{\varepsilon^{\bar{\theta}}})\setminus(\Omega^{\ast}_{R_{0}}\setminus\Omega^{\ast}_{\varepsilon^{\bar{\theta}}})}(2\mu e(\mathbf{u}_{1}^{\alpha}),e(\mathbf{u}_{1}^{\alpha})),\\
\mathrm{III}_{3}=&\int_{\Omega^{\ast}_{R_{0}}\setminus\Omega^{\ast}_{\varepsilon^{\bar{\theta}}}}(2\mu e(\mathbf{u}_{1}^{\ast\alpha}),e(\mathbf{u}_{1}^{\ast\alpha})).
\end{align*}
From \eqref{GZIJ} and \eqref{con035}, we deduce
\begin{align}\label{con036666}
|\mathrm{III}_{1}|\leq\,C\varepsilon^{\frac{1}{12}}.
\end{align}
Due to the fact that the volume of $(\Omega_{R_{0}}\setminus\Omega_{\varepsilon^{\bar{\theta}}})\setminus(\Omega^{\ast}_{R_{0}}\setminus\Omega^{\ast}_{\varepsilon^{\bar{\theta}}})$ is of order $O(\varepsilon)$, we have from \eqref{Le2.025} that for $\alpha=1,2,...,n,$
\begin{align}\label{con0333355}
|\mathrm{III}_{2}|\leq\,C\varepsilon\int_{\varepsilon^{\bar{\theta}}<|x'|<R_{0}}\frac{dx'}{|x'|^{4}}\leq& C
\begin{cases}
\varepsilon,&n>5,\\
\varepsilon|\ln\varepsilon|,&n=5,\\
\varepsilon^{\frac{n+19}{24}},&n<5,
\end{cases}
\end{align}
and, for $\alpha=n+1,...,\frac{n(n+1)}{2}$,
\begin{align}\label{con0333355KL}
|\mathrm{III}_{2}|\leq\,C\varepsilon\int_{\varepsilon^{\bar{\theta}}<|x'|<R_{0}}\frac{dx'}{|x'|^{2}}\leq& C
\begin{cases}
\varepsilon,&n>3,\\
\varepsilon|\ln\varepsilon|,&n=3,\\
\varepsilon^{\frac{n+21}{24}},&n=2.
\end{cases}
\end{align}
With regard to $\mathrm{III}_{3}$, it follows from \eqref{GZIJ} again that
\begin{align}\label{FT001}
\mathrm{III}_{3}=&\int_{\Omega_{R_{0}}^{\ast}\setminus\Omega^{\ast}_{\varepsilon^{\bar{\theta}}}}(2\mu e(\bar{\mathbf{u}}_{1}^{\ast\alpha}),e(\bar{\mathbf{u}}_{1}^{\ast\alpha}))+2\int_{\Omega_{R_{0}}^{\ast}\setminus\Omega^{\ast}_{\varepsilon^{\bar{\theta}}}}(2\mu e(\mathbf{u}_{1}^{\ast\alpha}-\bar{\mathbf{u}}_{1}^{\ast\alpha}),e(\bar{\mathbf{u}}_{1}^{\ast\alpha}))\notag\\
&+\int_{\Omega_{R_{0}}^{\ast}\setminus\Omega^{\ast}_{\varepsilon^{\bar{\theta}}}}(2\mu e(\mathbf{u}_{1}^{\ast\alpha}-\bar{\mathbf{u}}_{1}^{\ast\alpha}),e(\mathbf{u}_{1}^{\ast\alpha}-\bar{\mathbf{u}}_{1}^{\ast\alpha}))\notag\\
=&
\begin{cases}
\mathcal{L}_{\alpha}\int_{\varepsilon^{\bar{\theta}}<|x'|<R_{0}}\frac{dx'}{2h(x')}-\int_{\Omega^{\ast}\setminus\Omega^{\ast}_{R_{0}}}(2\mu e(\mathbf{u}_{1}^{\ast\alpha}),e(\mathbf{u}_{1}^{\ast\alpha}))\\
+\mathcal{M}_{n}^{\ast\alpha}+O(1)\max\{\varepsilon^{\frac{1}{12}},\varrho_{\alpha,n}(\varepsilon)\},\quad\alpha=1,2,...,n,\\
\frac{\mathcal{L}_{\alpha}}{n-1}\int_{\varepsilon^{\bar{\theta}}<|x'|<R_{0}}\frac{|x'|^{2}}{2h(x')}dx'-\int_{\Omega^{\ast}\setminus\Omega^{\ast}_{R_{0}}}(2\mu e(\mathbf{u}_{1}^{\ast\alpha}),e(\mathbf{u}_{1}^{\ast\alpha}))\\
+\mathcal{M}_{n}^{\ast\alpha}+O(1)\varepsilon^{\frac{1}{12}},\quad\alpha=n+1,...,\frac{n(n+1)}{2},
\end{cases}
\end{align}
where
\begin{align*}
\mathcal{M}_{n}^{\ast\alpha}:=&
\begin{cases}
0,&\alpha=n,\,n=2,\\
M_{n}^{\ast\alpha},&\text{otherwise}
\end{cases}
\end{align*}
with $M_{n}^{\ast\alpha}$ given by
\begin{align*}
M_{n}^{\ast\alpha}=&\int_{\Omega^{\ast}\setminus\Omega^{\ast}_{R_{0}}}(2\mu e(\bar{\mathbf{u}}_{1}^{\ast\alpha}),e(\bar{\mathbf{u}}_{1}^{\ast\alpha}))+\int_{\Omega_{R_{0}}^{\ast}}(2\mu e(\mathbf{u}_{1}^{\ast\alpha}-\bar{\mathbf{u}}_{1}^{\ast\alpha}),e(\mathbf{u}_{1}^{\ast\alpha}-\bar{\mathbf{u}}_{1}^{\ast\alpha}))\notag\\
&+\int_{\Omega_{R_{0}}^{\ast}}\big[2(2\mu e(\mathbf{u}_{1}^{\ast\alpha}-\bar{\mathbf{u}}_{1}^{\ast\alpha}),e(\bar{\mathbf{u}}_{1}^{\ast\alpha}))+2(2\mu e(\boldsymbol{\psi}_{\alpha}\bar{v}^{\ast}),e(\boldsymbol{\mathcal{F}}_{\alpha}^{\ast}))+(2\mu e(\boldsymbol{\mathcal{F}}_{\alpha}^{\ast}),e(\boldsymbol{\mathcal{F}}_{\alpha}^{\ast}))\big]\notag\\
&+
\begin{cases}
\int_{\Omega_{R_{0}}^{\ast}}\mu(\partial_{x_{\alpha}}\bar{v}^{\ast})^{2}+\mu\sum\limits^{n-1}_{i=1}(\partial_{x_{i}}\bar{v}^{\ast})^{2},\quad\alpha=1,...,n-1,\\
\int_{\Omega_{R_{0}}^{\ast}}\mu\sum\limits^{n-1}_{i=1}(\partial_{x_{i}}\bar{v}^{\ast})^{2},\quad\alpha=n,\,n\geq3,\\
\int_{\Omega_{R_{0}}^{\ast}}\mu\big[\mu(x_{\alpha-n}^{2}+x_{n}^{2})\sum\limits^{n-1}_{k=1}(\partial_{x_{k}}\bar{v}^{\ast})^{2}+x_{n}^{2}((\partial_{x_{n}}\bar{v}^{\ast})^{2}+(\partial_{x_{\alpha-n}}\bar{v}^{\ast})^{2})\\
\quad\quad\;-2x_{\alpha-n}x_{n}\partial_{x_{\alpha-n}}\bar{v}^{\ast}\partial_{x_{n}}\bar{v}^{\ast}\big],\;\quad\alpha=n+1,...,2n-1,\\
\int_{\Omega_{R_{0}}^{\ast}}\mu\big[(x_{i_{\alpha}}^{2}+x_{j_{\alpha}}^{2})\sum\limits^{n-1}_{k=1}(\partial_{x_{k}}\bar{v}^{\ast})^{2}\\
\quad\quad\;+(x_{j_{\alpha}}\partial_{x_{i_{\alpha}}}\bar{v}^{\ast}-x_{i_{\alpha}}\partial_{x_{j_{\alpha}}}\bar{v}^{\ast})^{2}\big],\;\quad\alpha=2n,...,\frac{n(n+1)}{2},\;n\geq3.
\end{cases}
\end{align*}
Consequently, from \eqref{con03365}--\eqref{FT001}, it follows that

$(\rm{i})$ if $\alpha=1,2,...,n$, then
\begin{align*}
a_{11}^{\alpha\alpha}=&\mathcal{L}_{\alpha}\left(\int_{\varepsilon^{\bar{\theta}}<|x'|<R_{0}}\frac{dx'}{2h(x')}+\int_{|x'|<\varepsilon^{\bar{\theta}}}\frac{dx'}{\varepsilon+2h(x')}\right)+\mathcal{M}_{n}^{\ast\alpha}+O(1)\max\{\varepsilon^{\frac{1}{12}},\varrho_{\alpha,n}(\varepsilon)\}.
\end{align*}
On one hand, for $n=2,3$, then
\begin{align*}
&\int_{\varepsilon^{\bar{\theta}}<|x'|<R_{0}}\frac{1}{2h(x')}+\int_{|x'|<\varepsilon^{\bar{\theta}}}\frac{1}{\varepsilon+2h(x')}\notag\\
=&\int_{|x'|<R_{0}}\frac{1}{\varepsilon+2h(x')}+\int_{\varepsilon^{\bar{\theta}}<|x'|<R_{0}}\frac{\varepsilon}{2h(x')(\varepsilon+2h(x'))}\notag\\
=&\int_{|x'|<R_{0}}\frac{1}{\varepsilon+2\kappa|x'|^{2}}+O(1)\varepsilon^{\frac{n+19}{24}}\\
=&\left(\frac{\pi}{2\kappa}\right)^{\frac{n-1}{2}}\rho_{n}(\varepsilon)+\mathcal{K}_{n}+O(1)\varepsilon^{\frac{n+19}{24}},
\end{align*}
where
\begin{align*}
\mathcal{K}_{n}=&
\begin{cases}
-(\kappa R_{0})^{-1},&n=2,\\
\pi(\ln\sqrt{2\kappa}+\ln R_{0})\kappa^{-1},&n=3.
\end{cases}
\end{align*}
Then we derive
\begin{align}\label{LNQ001}
a_{\alpha\alpha}=&\mathcal{L}_{\alpha}\left(\frac{\pi}{2\kappa}\right)^{\frac{n-1}{2}}\rho_{n}(\varepsilon)+\mathcal{G}_{n}^{\ast\alpha}+O(1)\max\{\varepsilon^{\frac{1}{12}},\varrho_{\alpha,n}(\varepsilon)\},
\end{align}
where
\begin{align}\label{NZWA001}
\mathcal{G}_{n}^{\ast\alpha}:=\mathcal{L}_{\alpha}\mathcal{K}_{n}+\mathcal{M}_{n}^{\ast\alpha},\quad\alpha=1,2,...,n,\,n=2,3.
\end{align}
Claim that for every $\alpha=1,...,n-1$, $n\geq2$ or $\alpha=n$, $n\geq3$, the geometry constant $\mathcal{G}_{n}^{\ast\alpha}$ defined by \eqref{NZWA001} depends not on the length parameter $R_{0}$ of the thin gap. Otherwise, assume that there exist two $\varepsilon$-independent constants $\mathcal{G}_{n}^{\ast\alpha}(r_{1})$ and $\mathcal{G}_{n}^{\ast\alpha}(r_{2})$, $r_{i}>0,\,i=1,2,\,r_{1}\neq r_{2}$ such that \eqref{LNQ001} holds. Hence we obtain
\begin{align*}
\mathcal{G}_{n}^{\ast\alpha}(r_{1})-\mathcal{G}_{n}^{\ast\alpha}(r_{2})=O(1)\max\{\varepsilon^{\frac{1}{12}},\varrho_{\alpha,n}(\varepsilon)\}.
\end{align*}
This implies that $\mathcal{G}_{n}^{\ast\alpha}(r_{1})=\mathcal{G}_{n}^{\ast\alpha}(r_{2})$.

On the other hand, if $n>3$, then
\begin{align}\label{LNQ002}
a_{\alpha\alpha}=&\mathcal{L}_{\alpha}\left(\int_{|x'|<R_{0}}\frac{1}{2h(x')}-\int_{|x'|<\varepsilon^{\bar{\theta}}}\frac{\varepsilon}{2h(x')(\varepsilon+2h(x'))}\right)\notag\\
&+\mathcal{M}_{n}^{\ast\alpha}+O(1)\max\{\varepsilon^{\frac{1}{12}},\varrho_{\alpha,n}(\varepsilon)\}\notag\\
=&\mathcal{L}_{\alpha}\int_{\Omega_{R_{0}}^{\ast}}|\partial_{x_{n}}\bar{\mathbf{u}}_{1}^{\ast\alpha}|^{2}+\mathcal{M}_{n}^{\ast\alpha}+O(1)\varepsilon^{\min\{\frac{1}{12},\frac{n-3}{24}\}}\notag\\
=&a_{11}^{\ast\alpha\alpha}+O(1)\varepsilon^{\min\{\frac{1}{12},\frac{n-3}{24}\}};
\end{align}

$(\rm{ii})$ if $\alpha=n+1,...,2n-1$, then
\begin{align}\label{QZH002}
a_{11}^{\alpha\alpha}=&\mathcal{L}_{\alpha}\left(\int_{\varepsilon^{\bar{\theta}}<|x'|<R_{0}}\frac{x_{\alpha-n}^{2}}{2h(x')}+\int_{|x'|<\varepsilon^{\bar{\theta}}}\frac{x_{\alpha-n}^{2}}{\varepsilon+2h(x')}\right)+\mathcal{M}_{n}^{\ast\alpha}+O(1)\varepsilon^{\min\{\frac{1}{12},\frac{n}{24}\}}\notag\\
=&\mathcal{L}_{\alpha}\left(\int_{|x'|<R_{0}}\frac{x_{\alpha-n}^{2}}{2h(x')}-\int_{|x'|<\varepsilon^{\bar{\theta}}}\frac{\varepsilon x_{\alpha-n}^{2}}{2h(x')(\varepsilon+2h(x'))}\right)\notag\\
&+\mathcal{M}_{n}^{\ast\alpha}+O(1)\varepsilon^{\min\{\frac{1}{12},\frac{n}{24}\}}\notag\\
=&\mathcal{L}_{\alpha}\int_{\Omega_{R_{0}}^{\ast}}|x_{\alpha-n}\partial_{x_{n}}\bar{v}^{\ast}|^{2}+\mathcal{M}_{n}^{\ast\alpha}+O(1)\varepsilon^{\min\{\frac{1}{12},\frac{n-1}{24}\}}\notag\\
=&a_{11}^{\ast\alpha\alpha}+O(1)\varepsilon^{\min\{\frac{1}{12},\frac{n-1}{24}\}},
\end{align}
and, for $\alpha=2n,...,\frac{n(n+1)}{2},\,n\geq3$,
\begin{align*}
a_{11}^{\alpha\alpha}=&\frac{\mathcal{L}_{\alpha}}{2}\left(\int_{\varepsilon^{\bar{\theta}}<|x'|<R_{0}}\frac{x_{i_{\alpha}}^{2}+x_{j_{\alpha}}^{2}}{2h(x')}+\int_{|x'|<\varepsilon^{\bar{\theta}}}\frac{x_{i_{\alpha}}^{2}+x_{j_{\alpha}}^{2}}{\varepsilon+2h(x')}\right)\notag\\
&+\mathcal{M}_{n}^{\ast\alpha}+O(1)\varepsilon^{\min\{\frac{1}{12},\frac{n}{24}\}}\\
=&\frac{\mathcal{L}_{\alpha}}{2}\int_{|x'|<R_{0}}\frac{x_{i_{\alpha}}^{2}+x_{j_{\alpha}}^{2}}{2h(x')}+\mathcal{M}_{n}^{\ast\alpha}+O(1)\varepsilon^{\min\{\frac{1}{12},\frac{n-1}{24}\}}\\
=&\frac{\mathcal{L}_{\alpha}}{2}\int_{\Omega_{R_{0}}^{\ast}}(x_{i_{\alpha}}^{2}+x_{j_{\alpha}}^{2})|\partial_{x_{n}}\bar{v}^{\ast}|^{2}+\mathcal{M}_{n}^{\ast\alpha}+O(1)\varepsilon^{\min\{\frac{1}{12},\frac{n-1}{24}\}}\\
=&a_{11}^{\ast\alpha\alpha}+O(1)\varepsilon^{\min\{\frac{1}{12},\frac{n-1}{24}\}}.
\end{align*}
A combination of \eqref{LNQ001}--\eqref{QZH002} leads to that \eqref{LMC}--\eqref{LMC1} hold.

\noindent{\bf Step 2.} Proofs of \eqref{M001}--\eqref{M005}. By symmetry, it suffices to consider the case when $\alpha<\beta$. Analogously as before, for $\alpha,\beta=1,2,...,\frac{n(n+1)}{2}$, $\alpha\neq\beta,$ we split $a_{11}^{\alpha\beta}$ as follows:
\begin{align*}
a_{11}^{\alpha\beta}=&\int_{\Omega\setminus\Omega_{R_{0}}}(2\mu e(\mathbf{u}_{1}^{\alpha}),e(\mathbf{u}_{1}^{\beta}))+\int_{\Omega_{R_{0}}\setminus\Omega_{\varepsilon^{\bar{\theta}}}}(2\mu e(\mathbf{u}_{1}^{\alpha}),e(\mathbf{u}_{1}^{\beta}))+\int_{\Omega_{\varepsilon^{\bar{\theta}}}}(2\mu e(\mathbf{u}_{1}^{\alpha}),e(\mathbf{u}_{1}^{\beta}))\nonumber\\
=&:\mathrm{I}+\mathrm{II}+\mathrm{III},
\end{align*}
where $\bar{\theta}=\frac{1}{24}$. In exactly the same way as in \eqref{KKAA123}, we have
\begin{align}\label{KKAA1233333}
\mathrm{I}=&\int_{D\setminus(D_{1}\cup D_{1}^{\ast}\cup D_{2}\cup D_{2}^{\ast}\cup\Omega_{R_{0}})}(2\mu e(\mathbf{u}_{1}^{\alpha}),e(\mathbf{u}_{1}^{\beta}))+O(1)\varepsilon\notag\\
=&\int_{D\setminus(D_{1}\cup D_{1}^{\ast}\cup D_{2}\cup D_{2}^{\ast}\Omega_{R_{0}})}\big[(2\mu e(\mathbf{u}_{1}^{\ast\alpha}),e(\mathbf{u}_{1}^{\ast\beta}))+(2\mu e(\mathbf{u}_{1}^{\alpha}-\mathbf{u}_{1}^{\ast\alpha}),e(\mathbf{u}_{1}^{\beta}-\mathbf{u}_{1}^{\ast\beta}))\big]\notag\\
&+\int_{D\setminus(D_{1}\cup D_{1}^{\ast}\cup D_{2}\cup D_{2}^{\ast}\cup\Omega_{R_{0}})}\big[(2\mu e(\mathbf{u}_{1}^{\ast\alpha}),e(\mathbf{u}_{1}^{\beta}-\mathbf{u}_{1}^{\ast\beta}))+(2\mu e(\mathbf{u}_{1}^{\alpha}-\mathbf{u}_{1}^{\ast\alpha}),e(\mathbf{u}_{1}^{\ast\beta}))\big]\notag\\
=&\int_{\Omega^{\ast}\setminus\Omega^{\ast}_{R_{0}}}(2\mu e(\mathbf{u}_{1}^{\ast\alpha}),e(\mathbf{u}_{1}^{\ast\beta}))+O(1)\varepsilon^{\frac{1}{12}}.
\end{align}

With regard to $\mathrm{II}$, it can be further decomposed as follows:
\begin{align*}
\mathrm{II}_{1}=&\int_{(\Omega_{R_{0}}\setminus\Omega_{\varepsilon^{\bar{\theta}}})\setminus(\Omega^{\ast}_{R_{0}}\setminus\Omega^{\ast}_{\varepsilon^{\bar{\theta}}})}(2\mu e(\mathbf{u}_{1}^{\alpha}),e(\mathbf{u}_{1}^{\beta}))+\int_{\Omega^{\ast}_{R_{0}}\setminus\Omega^{\ast}_{\varepsilon^{\bar{\theta}}}}(2\mu e(\mathbf{u}_{1}^{\ast\alpha}),e(\mathbf{u}_{1}^{\beta}-\mathbf{u}_{1}^{\ast\beta}))\notag\\
&+\int_{\Omega^{\ast}_{R_{0}}\setminus\Omega^{\ast}_{\varepsilon^{\bar{\theta}}}}(2\mu e(\mathbf{u}_{1}^{\alpha}-\mathbf{u}_{1}^{\ast\alpha}),e(\mathbf{u}_{1}^{\ast\beta}))+\int_{\Omega^{\ast}_{R_{0}}\setminus\Omega^{\ast}_{\varepsilon^{\bar{\theta}}}}(2\mu e(\mathbf{u}_{1}^{\alpha}-\mathbf{u}_{1}^{\ast\alpha}),e(\mathbf{u}_{1}^{\beta}-\mathbf{u}_{1}^{\ast\beta})),\\
\mathrm{II}_{2}=&\int_{\Omega^{\ast}_{R_{0}}\setminus\Omega^{\ast}_{\varepsilon^{\bar{\theta}}}}(2\mu e(\mathbf{u}_{1}^{\ast\alpha}),e(\mathbf{u}_{1}^{\ast\beta})).
\end{align*}
Observe that the volume of $(\Omega_{R_{0}}\setminus\Omega_{\varepsilon^{\bar{\theta}}})\setminus(\Omega^{\ast}_{R_{0}}\setminus\Omega^{\ast}_{\varepsilon^{\bar{\theta}}})$ is of order $O(\varepsilon)$, we deduce from \eqref{Le2.025} and \eqref{con035} that
\begin{align}\label{con036}
\mathrm{II}_{1}=O(1)\varepsilon^{\frac{1}{12}}.
\end{align}
We proceed to calculate $\mathrm{II}_{2}$ and divide into two cases as follows:

{\bf Case 1.} In the case of $n=2$, $\alpha=1,\beta=2$, we deduce from \eqref{GZIJ} that
\begin{align}\label{PAHN}
\mathrm{II}_{2}=&\int_{\Omega_{R_{0}}^{\ast}\setminus\Omega^{\ast}_{\varepsilon^{\bar{\theta}}}}(2\mu e(\bar{\mathbf{u}}_{1}^{\ast1}),e(\bar{\mathbf{u}}_{1}^{\ast 2}))+\int_{\Omega_{R_{0}}^{\ast}\setminus\Omega^{\ast}_{\varepsilon^{\bar{\theta}}}}(2\mu e(\mathbf{u}_{1}^{\ast1}-\bar{\mathbf{u}}_{1}^{\ast1}),e(\mathbf{u}_{1}^{\ast 2}-\bar{\mathbf{u}}_{1}^{\ast2}))\notag\\
&+\int_{\Omega_{R_{0}}^{\ast}\setminus\Omega^{\ast}_{\varepsilon^{\bar{\theta}}}}(2\mu e(\mathbf{u}_{1}^{\ast1}-\bar{\mathbf{u}}_{1}^{\ast1}),e(\bar{\mathbf{u}}_{1}^{\ast2}))+\int_{\Omega_{R_{0}}^{\ast}\setminus\Omega^{\ast}_{\varepsilon^{\bar{\theta}}}}(2\mu e(\bar{\mathbf{u}}_{1}^{\ast1}),e(\mathbf{u}_{1}^{\ast2}-\bar{\mathbf{u}}_{1}^{\ast2}))\notag\\
=&\int_{\Omega_{R_{0}}^{\ast}\setminus\Omega^{\ast}_{\varepsilon^{\bar{\theta}}}}\mu\partial_{x_{1}}\bar{v}^{\ast}\partial_{x_{2}}\bar{v}^{\ast}+\int_{\Omega_{R_{0}}^{\ast}\setminus\Omega^{\ast}_{\varepsilon^{\bar{\theta}}}}(2\mu e(\bar{\mathbf{u}}_{1}^{\ast1}),e(\mathbf{u}_{1}^{\ast2}-\bar{\mathbf{u}}_{1}^{\ast2}))+O(1)\notag\\
=&O(1)|\ln\varepsilon|;
\end{align}

{\bf Case 2.} Consider the case when $(\alpha,\beta)\in\{(\alpha,\beta)|\,\alpha<\beta,\,\alpha,\beta=1,2,...,\frac{n(n+1)}{2},\,n\geq2\}\setminus\{(\alpha,\beta)|\,\alpha<\beta,\,\alpha,\beta=1,2,...,n,\,n=2\}.$ Observe that
\begin{align}\label{FAZWQ001}
(2\mu e(\bar{\mathbf{u}}_{1}^{\ast\alpha}),e(\bar{\mathbf{u}}_{1}^{\ast\beta}))=&(2\mu e(\boldsymbol{\psi}_{\alpha}\bar{v}^{\ast}),e(\boldsymbol{\psi}_{\beta}\bar{v}^{\ast}))+(2\mu e(\boldsymbol{\mathcal{F}}_{\alpha}^{\ast}),e(\boldsymbol{\psi}_{\beta}\bar{v}^{\ast}))\notag\\
&+(2\mu e(\boldsymbol{\psi}_{\alpha}\bar{v}^{\ast}),e(\boldsymbol{\mathcal{F}}_{\beta}^{\ast}))+(2\mu e(\boldsymbol{\mathcal{F}}_{\alpha}^{\ast}),e(\boldsymbol{\mathcal{F}}_{\beta}^{\ast})).
\end{align}
It then follows from \eqref{GZIJ} again that
\begin{align}\label{QKL}
&\mathrm{II}_{2}-\int_{\Omega^{\ast}_{R_{0}}}(2\mu e(\mathbf{u}_{1}^{\ast\alpha}),e(\mathbf{u}_{1}^{\ast\beta}))=-\int_{\Omega^{\ast}_{\varepsilon^{\bar{\theta}}}}(2\mu e(\mathbf{u}_{1}^{\ast\alpha}),e(\mathbf{u}_{1}^{\ast\beta}))\notag\\
&=\int_{\Omega^{\ast}_{\varepsilon^{\bar{\theta}}}}\big[(2\mu e(\bar{\mathbf{u}}_{1}^{\ast\alpha}),e(\bar{\mathbf{u}}_{1}^{\ast\beta}))+(2\mu e(\mathbf{u}_{1}^{\ast\alpha}-\bar{\mathbf{u}}_{1}^{\ast\alpha}),e(\mathbf{u}_{1}^{\ast\beta}-\bar{\mathbf{u}}_{1}^{\ast\beta}))\notag\\
&\quad\quad\quad+(2\mu e(\bar{\mathbf{u}}_{1}^{\ast\alpha}),e(\mathbf{u}_{1}^{\ast\beta}-\bar{\mathbf{u}}_{1}^{\ast\beta}))+(2\mu e(\mathbf{u}_{1}^{\ast\alpha}-\bar{\mathbf{u}}_{1}^{\ast\alpha}),e(\bar{\mathbf{u}}_{1}^{\ast\beta}))\big]\notag\\
&=\int_{\Omega^{\ast}_{\varepsilon^{\bar{\theta}}}}(2\mu e(\boldsymbol{\psi}_{\alpha}\bar{v}^{\ast}),e(\boldsymbol{\psi}_{\beta}\bar{v}^{\ast}))\notag\\
&+O(1)
\begin{cases}
\varepsilon^{(n-2)\bar{\theta}},&\alpha,\beta=1,2,...,n,\,n\geq3,\,\alpha<\beta,\\
\varepsilon^{(n-1)\bar{\theta}},&\alpha=1,2,...,n,\,\beta=n+1,...,\frac{n(n+1)}{2},\,n\geq2,\,\alpha<\beta,\\
\varepsilon^{n\bar{\theta}},&\alpha,\beta=n+1,...,\frac{n(n+1)}{2},\,n\geq3,\,\alpha<\beta.
\end{cases}
\end{align}
Write $E_{\alpha\beta}(\bar{v}^{\ast})=(2\mu e(\boldsymbol{\psi}_{\alpha}\bar{v}^{\ast}),e(\boldsymbol{\psi}_{\beta}\bar{v}^{\ast}))$. By a direct calculation, we obtain

$(i)$ for $\alpha,\beta=1,2,...,n,$ $\alpha<\beta$,
\begin{align}\label{ZH0000}
E_{\alpha\beta}(\bar{v}^{\ast})=\mu\partial_{x_{\alpha}}\bar{v}^{\ast}\partial_{x_{\beta}}\bar{v}^{\ast};
\end{align}

$(ii)$ for $\alpha=1,2,...,n$, $\beta=n+1,...,\frac{n(n+1)}{2}$, there exist two integers $1\leq i_{\beta}<j_{\beta}\leq n$ such that
$\boldsymbol{\psi}_{\beta}\bar{v}^{\ast}=\bar{v}^{\ast}(x_{j_{\beta}}\mathbf{e}_{x_{i_{\beta}}}-x_{i_{\beta}}\mathbf{e}_{x_{j_{\beta}}})$. If $i_{\beta}\neq\alpha,\,j_{\beta}\neq\alpha$,
\begin{align}\label{ZH000}
E_{\alpha\beta}(\bar{v}^{\ast})=\mu\partial_{x_{\alpha}}\bar{v}^{\ast}(x_{j_{\beta}}\partial_{i_{\beta}}\bar{v}^{\ast}-x_{i_{\beta}}\partial_{x_{j_{\beta}}}\bar{v}^{\ast}),
\end{align}
and if $i_{\beta}=\alpha,\,j_{\beta}\neq\alpha$, then
\begin{align}\label{ZH001}
E_{\alpha\beta}(\bar{v}^{\ast})=\mu x_{j_{\beta}}\sum^{n}_{k=1}(\partial_{x_{k}}\bar{v}^{\ast})^{2}+\mu\partial_{x_{\alpha}}\bar{v}^{\ast}(x_{j_{\beta}}\partial_{i_{\beta}}\bar{v}^{\ast}-x_{i_{\beta}}\partial_{x_{j_{\beta}}}\bar{v}^{\ast}),
\end{align}
and if $i_{\beta}\neq\alpha,\,j_{\beta}=\alpha$, then
\begin{align}\label{ZH002}
E_{\alpha\beta}(\bar{v}^{\ast})=&-\mu x_{i_{\beta}}\sum^{n}_{k=1}(\partial_{x_{k}}\bar{v}^{\ast})^{2}+\mu\partial_{x_{\alpha}}\bar{v}^{\ast}(x_{j_{\beta}}\partial_{i_{\beta}}\bar{v}^{\ast}-x_{i_{\beta}}\partial_{x_{j_{\beta}}}\bar{v}^{\ast});
\end{align}

$(iii)$ for $\alpha,\beta=n+1,...,\frac{n(n+1)}{2}$, $\alpha<\beta$, there exist four integers $1\leq i_{\alpha}<j_{\alpha}\leq n$ and $1\leq i_{\beta}<j_{\beta}\leq n$ such that $\boldsymbol{\psi}_{\alpha}\bar{v}^{\ast}=\bar{v}^{\ast}(x_{j_{\alpha}}\mathbf{e}_{x_{i_{\alpha}}}-x_{i_{\alpha}}\mathbf{e}_{x_{j_{\alpha}}})$
and
$\boldsymbol{\psi}_{\beta}\bar{v}^{\ast}=\bar{v}^{\ast}(x_{j_{\beta}}\mathbf{e}_{x_{i_{\beta}}}-x_{i_{\beta}}\mathbf{e}_{x_{j_{\beta}}})$. Since $\alpha<\beta$, then $j_{\beta}\leq j_{\alpha}$. If $i_{\alpha}\neq i_{\beta},\,j_{\alpha}\neq j_{\beta},\,i_{\alpha}\neq j_{\beta}$, then
\begin{align}\label{ZH003}
E_{\alpha\beta}(\bar{v}^{\ast})=\mu(x_{j_{\alpha}}\partial_{x_{i_{\alpha}}}\bar{v}^{\ast}-x_{i_{\alpha}}\partial_{x_{j_{\alpha}}}\bar{v}^{\ast})(x_{j_{\beta}}\partial_{x_{i_{\beta}}}\bar{v}^{\ast}-x_{i_{\beta}}\partial_{x_{j_{\beta}}}\bar{v}^{\ast}),
\end{align}
and if $i_{\alpha}\neq i_{\beta},\,j_{\alpha}=j_{\beta}$, then
\begin{align}\label{ZH005}
E_{\alpha\beta}(\bar{v}^{\ast})=&\mu x_{i_{\alpha}}x_{i_{\beta}}\sum^{n}_{k=1}(\partial_{x_{k}}\bar{v}^{\ast})^{2}\notag\\
&+\mu(x_{j_{\alpha}}\partial_{x_{i_{\alpha}}}\bar{v}^{\ast}-x_{i_{\alpha}}\partial_{x_{j_{\alpha}}}\bar{v}^{\ast})(x_{j_{\beta}}\partial_{x_{i_{\beta}}}\bar{v}^{\ast}-x_{i_{\beta}}\partial_{x_{j_{\beta}}}\bar{v}^{\ast}),
\end{align}
and if $i_{\alpha}=i_{\beta},\,j_{\alpha}\neq j_{\beta}$, then
\begin{align}\label{ZH004}
E_{\alpha\beta}(\bar{v}^{\ast})=&\mu x_{j_{\alpha}}x_{j_{\beta}}\sum^{n}_{k=1}(\partial_{x_{k}}\bar{v}^{\ast})^{2}\notag\\
&+\mu(x_{j_{\alpha}}\partial_{x_{i_{\alpha}}}\bar{v}^{\ast}-x_{i_{\alpha}}\partial_{x_{j_{\alpha}}}\bar{v}^{\ast})(x_{j_{\beta}}\partial_{x_{i_{\beta}}}\bar{v}^{\ast}-x_{i_{\beta}}\partial_{x_{j_{\beta}}}\bar{v}^{\ast}),
\end{align}
and if $i_{\beta}<j_{\beta}=i_{\alpha}<j_{\alpha}$, then
\begin{align}\label{ZH006}
E_{\alpha\beta}(\bar{v}^{\ast})=&-\mu x_{i_{\beta}}x_{j_{\alpha}}\sum^{n}_{k=1}(\partial_{x_{k}}\bar{v}^{\ast})^{2}\notag\\
&+\mu(x_{j_{\alpha}}\partial_{x_{i_{\alpha}}}\bar{v}^{\ast}-x_{i_{\alpha}}\partial_{x_{j_{\alpha}}}\bar{v}^{\ast})(x_{j_{\beta}}\partial_{x_{i_{\beta}}}\bar{v}^{\ast}-x_{i_{\beta}}\partial_{x_{j_{\beta}}}\bar{v}^{\ast}).
\end{align}

Observe that
\begin{align*}
\int^{h(x')}_{-h(x')}x_{n}\,dx_{n}=0,\quad \mathrm{in}\; B'_{R_{0}},
\end{align*}
which, in combination with \eqref{QKL}--\eqref{ZH006}, the symmetry of integral domain and the parity of integrand, reads that
\begin{align}\label{AMQZ001}
&\mathrm{II}_{2}-\int_{\Omega^{\ast}_{R_{0}}}(2\mu e(\mathbf{u}_{1}^{\ast\alpha}),e(\mathbf{u}_{1}^{\ast\beta}))\notag\\
&=O(1)
\begin{cases}
\varepsilon^{(n-2)\bar{\theta}},&\alpha,\beta=1,2,...,n,\,n\geq3,\,\alpha<\beta,\\
\varepsilon^{(n-1)\bar{\theta}},&\alpha=1,2,...,n,\,\beta=n+1,...,\frac{n(n+1)}{2},\,n\geq2,\,\alpha<\beta,\\
\varepsilon^{n\bar{\theta}},&\alpha,\beta=n+1,...,\frac{n(n+1)}{2},\,n\geq3,\,\alpha<\beta.
\end{cases}
\end{align}
A combination of \eqref{con036}--\eqref{PAHN} and \eqref{AMQZ001} shows that
\begin{align}\label{FATL001}
\mathrm{II}=&O(1)|\ln\varepsilon|,\quad n=2,\alpha=1,\beta=2,
\end{align}
and
\begin{align}\label{FATL002}
&\mathrm{II}-\int_{\Omega^{\ast}_{R_{0}}}(2\mu e(\mathbf{u}_{1}^{\ast\alpha}),e(\mathbf{u}_{1}^{\ast\beta}))\notag\\
&=O(1)
\begin{cases}
\varepsilon^{\min\{\frac{1}{12},\frac{n-2}{24}\}},&\alpha,\beta=1,2,...,n,\,n\geq3,\,\alpha<\beta,\\
\varepsilon^{\min\{\frac{1}{12},\frac{n-1}{24}\}},&\alpha=1,2,...,n,\,\beta=n+1,...,\frac{n(n+1)}{2},\,n\geq2,\\
\varepsilon^{\min\{\frac{1}{12},\frac{n}{24}\}},&\alpha,\beta=n+1,...,\frac{n(n+1)}{2},\,n\geq3,\,\alpha<\beta.
\end{cases}
\end{align}

Analogously, using \eqref{Le2.025} and applying \eqref{FAZWQ001}--\eqref{ZH006} with $\bar{v}^{\ast}$ replaced by $\bar{v}$, we have
\begin{align}\label{KMAZE001}
\mathrm{III}=&\int_{\Omega_{\varepsilon^{\bar{\theta}}}}[(2\mu e(\boldsymbol{\psi}_{\alpha}\bar{v}),e(\boldsymbol{\psi}_{\beta}\bar{v}))+(2\mu e(\boldsymbol{\psi}_{\alpha}\bar{v}),e(\boldsymbol{\mathcal{F}}_{\beta}))]\notag\\
&+\int_{\Omega_{\varepsilon^{\bar{\theta}}}}[(2\mu e(\boldsymbol{\mathcal{F}}_{\alpha}),e(\boldsymbol{\psi}_{\beta}\bar{v}))+(2\mu e(\boldsymbol{\mathcal{F}}_{\alpha}),e(\boldsymbol{\mathcal{F}}_{\beta}))]\notag\\
=&O(1)
\begin{cases}
|\ln\varepsilon|,&\alpha=1,\beta=2,\,n=2,\\
\varepsilon^{\frac{n-2}{24}},&\alpha,\beta=1,2,...,n,\,n\geq3,\,\alpha<\beta,\\
\varepsilon^{\frac{n-1}{24}},&\alpha=1,2,...,n,\,\beta=n+1,...,\frac{n(n+1)}{2},\,n\geq2,\\
\varepsilon^{\frac{n}{24}},&\alpha,\beta=n+1,...,\frac{n(n+1)}{2},\,n\geq3,\,\alpha<\beta.
\end{cases}
\end{align}
Consequently, combining \eqref{KKAA1233333} and \eqref{FATL001}--\eqref{KMAZE001}, we deduce
\begin{align*}
a_{11}^{12}=O(1)|\ln\varepsilon|,\quad n=2,
\end{align*}
and
\begin{align*}
&a_{11}^{\alpha\beta}-a_{11}^{\ast\alpha\beta}\notag\\
&=O(1)
\begin{cases}
\varepsilon^{\min\{\frac{1}{12},\frac{n-2}{24}\}},&\alpha,\beta=1,2,...,n,\,n\geq3,\,\alpha<\beta,\\
\varepsilon^{\min\{\frac{1}{12},\frac{n-1}{24}\}},&\alpha=1,2,...,n,\,\beta=n+1,...,\frac{n(n+1)}{2},\,n\geq2,\\
\varepsilon^{\min\{\frac{1}{12},\frac{n}{24}\}},&\alpha,\beta=n+1,...,\frac{n(n+1)}{2},\,n\geq3,\,\alpha<\beta.
\end{cases}
\end{align*}
That is, \eqref{M001}--\eqref{M005} are proved.

\noindent{\bf Step 3.} Proof of \eqref{AZQ001}. By linearity, we see that for $\alpha=1,2,...,\frac{n(n+1)}{2}$, $\mathbf{u}_{1}^{\alpha}+\mathbf{u}_{2}^{\alpha}-\mathbf{u}_{1}^{\ast\alpha}-\mathbf{u}_{2}^{\ast\alpha}$ and $p_{1}^{\alpha}+p_{2}^{\alpha}-p_{1}^{\ast\alpha}-p_{2}^{\ast\alpha}$ verify
\begin{align*}
\begin{cases}
\nabla\cdot\sigma\big[\sum\limits^{2}_{i=1}\mathbf{u}_{i}^{\alpha}-\sum\limits^{2}_{i=1}\mathbf{u}_{i}^{\ast\alpha},\sum\limits^{2}_{i=1}p_{i}^{\alpha}-\sum\limits^{2}_{i=1}p_{i}^{\ast\alpha}\big]=0,&\mathrm{in}\;\,D\setminus(\overline{D_{1}\cup D_{2}\cup D_{1}^{\ast}\cup D_{2}^{\ast}}),\\
\nabla\cdot(\mathbf{u}_{1}^{\alpha}+\mathbf{u}_{2}^{\alpha}-\mathbf{u}_{1}^{\ast\alpha}-\mathbf{u}_{2}^{\ast\alpha})=0,&\mathrm{in}\;\,D\setminus(\overline{D_{1}\cup D_{2}\cup D_{1}^{\ast}\cup D_{2}^{\ast}}),\\
\mathbf{u}_{1}^{\alpha}+\mathbf{u}_{2}^{\alpha}-\mathbf{u}_{1}^{\ast\alpha}-\mathbf{u}_{2}^{\ast\alpha}=\boldsymbol{\psi}_{\alpha}-\mathbf{u}_{1}^{\ast\alpha}-\mathbf{u}_{2}^{\ast\alpha},&\mathrm{on}\;\,(\partial D_{1}\setminus D_{1}^{\ast})\cup(\partial D_{2}\setminus D_{2}^{\ast}),\\
\mathbf{u}_{1}^{\alpha}+\mathbf{u}_{2}^{\alpha}-\mathbf{u}_{1}^{\ast\alpha}-\mathbf{u}_{2}^{\ast\alpha}=\mathbf{u}_{1}^{\alpha}+\mathbf{u}_{2}^{\alpha}-\boldsymbol{\psi}_{\alpha},&\mathrm{on}\;\,\bigcup\limits^{2}_{i=1}(\partial D_{i}^{\ast}\setminus(D_{i}\cup\{0\})),\\
\mathbf{u}_{1}^{\alpha}+\mathbf{u}_{2}^{\alpha}-\mathbf{u}_{1}^{\ast\alpha}-\mathbf{u}_{2}^{\ast\alpha}=0,&\mathrm{on}\;\,\partial D.
\end{cases}
\end{align*}
Similarly as before, from the standard interior and boundary estimates for Stokes equation, it follows that for $x\in\partial D_{i}\setminus D_{i}^{\ast}$, $i=1,2,$
\begin{align*}
&|(\mathbf{u}_{1}^{\alpha}+\mathbf{u}_{2}^{\alpha}-\mathbf{u}_{1}^{\ast\alpha}-\mathbf{u}_{2}^{\ast\alpha})(x',x_{n})|\notag\\
&=|(\mathbf{u}_{1}^{\ast\alpha}+\mathbf{u}_{2}^{\ast\alpha})(x',x_{n}+(-1)^{i}\varepsilon/2)-(\mathbf{u}_{1}^{\ast\alpha}+\mathbf{u}_{2}^{\ast\alpha})(x',x_{n})|\leq C\varepsilon.
\end{align*}
Utilizing Corollary \ref{coro00z}, we obtain that for $x\in\partial D_{i}^{\ast}\setminus(D_{i}\cup\mathcal{C}_{\sqrt{\varepsilon}})$, $i=1,2,$
\begin{align*}
&|(\mathbf{u}_{1}^{\alpha}+\mathbf{u}_{2}^{\alpha}-\mathbf{u}_{1}^{\ast\alpha}-\mathbf{u}_{2}^{\ast\alpha})(x',x_{n})|\notag\\
&=|(\mathbf{u}_{1}^{\alpha}+\mathbf{u}_{2}^{\alpha})(x',x_{n})-(\mathbf{u}_{1}^{\alpha}+\mathbf{u}_{2}^{\alpha})(x',x_{n}-(-1)^{i}\varepsilon/2)|\leq C\varepsilon,
\end{align*}
which, together with Corollary \ref{coro00z} again, shows that for $x\in\Omega_{R_{0}}^{\ast}\cap\{|x'|=\sqrt{\varepsilon}\}$,
\begin{align*}
&|(\mathbf{u}_{1}^{\alpha}+\mathbf{u}_{2}^{\alpha}-\mathbf{u}_{1}^{\ast\alpha}-\mathbf{u}_{2}^{\ast\alpha})(x',x_{n})|\notag\\
&\leq|(\mathbf{u}_{1}^{\alpha}+\mathbf{u}_{2}^{\alpha}-\mathbf{u}_{1}^{\ast\alpha}-\mathbf{u}_{2}^{\ast\alpha})(x',x_{n})-(\mathbf{u}_{1}^{\alpha}+\mathbf{u}_{2}^{\alpha}-\mathbf{u}_{1}^{\ast\alpha}-\mathbf{u}_{2}^{\ast\alpha})(x',h(x'))|\notag\\ &\quad+|(\mathbf{u}_{1}^{\alpha}+\mathbf{u}_{2}^{\alpha}-\mathbf{u}_{1}^{\ast\alpha}-\mathbf{u}_{2}^{\ast\alpha})(x',h(x'))|\leq C\varepsilon.
\end{align*}
Combining these above facts, we have
\begin{align}\label{MIH01}
|\mathbf{u}_{1}^{\alpha}+\mathbf{u}_{2}^{\alpha}-\mathbf{u}_{1}^{\ast\alpha}-\mathbf{u}_{2}^{\ast\alpha}|\leq C\varepsilon,\quad\;\,\mathrm{on}\;\,\partial \big(D\setminus\big(\overline{D_{1}\cup D_{1}^{\ast}\cup D_{2}\cup D_{2}^{\ast}\cup\mathcal{C}_{\varepsilon^{\frac{1}{2}}}}\big)\big).
\end{align}
Similar to \eqref{con035}, a consequence of \eqref{MIH01}, the maximum modulus theorem, the interpolation inequality, the rescale argument and the standard estimates for Stokes flow yields that
\begin{align}\label{ZQWZW001}
|\nabla(\mathbf{u}_{1}^{\alpha}+\mathbf{u}_{2}^{\alpha}-\mathbf{u}_{1}^{\ast\alpha}-\mathbf{u}_{2}^{\ast\alpha})|\leq C\varepsilon^{\frac{1}{4}},\quad\mathrm{in}\;\,D\setminus\big(\overline{D_{1}\cup D_{1}^{\ast}\cup D_{2}\cup D_{2}^{\ast}\cup\mathcal{C}_{\varepsilon^{\frac{1}{8}}}}\big).
\end{align}

Set $\tilde{\theta}=\frac{1}{8}$ and split $\sum\limits^{2}_{i=1}a_{i1}^{\alpha\beta}$ as follows:
\begin{align*}
\mathrm{I}=&\int_{\Omega_{\varepsilon^{\tilde{\theta}}}}(2\mu e(\mathbf{u}_{1}^{\alpha}+\mathbf{u}_{2}^{\alpha}),e(\mathbf{u}_{1}^{\beta})),\notag\\
\mathrm{II}=&\int_{\Omega\setminus\Omega_{R_{0}}}(2\mu e(\mathbf{u}_{1}^{\alpha}+\mathbf{u}_{2}^{\alpha}),e(\mathbf{u}_{1}^{\beta})),\notag\\
\mathrm{III}=&\int_{\Omega_{R_{0}}\setminus\Omega_{\varepsilon^{\tilde{\theta}}}}(2\mu e(\mathbf{u}_{1}^{\alpha}+\mathbf{u}_{2}^{\alpha}),e(\mathbf{u}_{1}^{\beta})).
\end{align*}
Since exponential function decays faster than power function, then we infer from Proposition \ref{thm86} and Corollary \ref{coro00z} that
\begin{align}\label{GAZ001}
|\mathrm{I}|\leq \int_{|x'|\leq \varepsilon^{\tilde{\theta}}}C(\varepsilon+|x'|^{2})^{1/2}\leq C\varepsilon^{n\tilde{\theta}}.
\end{align}
Analogous to \eqref{KKAA123}, it follows from \eqref{ZQWZW001} that
\begin{align}\label{GAZ0011}
\mathrm{II}=\int_{\Omega^{\ast}\setminus\Omega^{\ast}_{R_{0}}}(2\mu e(\mathbf{u}_{1}^{\ast\alpha}+\mathbf{u}_{2}^{\ast\alpha}),e(\mathbf{u}_{1}^{\ast\beta}))+O(1)\varepsilon^{\frac{1}{4}}.
\end{align}

As for the last term $\mathrm{III}$, we further split it as follows:
\begin{align*}
\mathrm{III}_{1}=&\int_{\Omega^{\ast}_{R_{0}}\setminus\Omega^{\ast}_{\varepsilon^{\tilde{\theta}}}}\sum^{2}_{i=1}\Big[(2\mu e(\mathbf{u}_{i}^{\alpha}-\mathbf{u}_{i}^{\ast\alpha}),e(\mathbf{u}_{1}^{\ast\beta}))+(2\mu e(\mathbf{u}_{i}^{\ast\alpha}),e(\mathbf{u}_{1}^{\beta}-\mathbf{u}_{1}^{\ast\beta}))\notag\\
&\quad\quad\quad\quad\quad\quad+(2\mu e(\mathbf{u}_{i}^{\alpha}-\mathbf{u}_{i}^{\ast\alpha}),e(\mathbf{u}_{1}^{\beta}-\mathbf{u}_{1}^{\ast\beta}))\Big],\\
\mathrm{III}_{2}=&\int_{(\Omega_{R_{0}}\setminus\Omega_{\varepsilon^{\tilde{\theta}}})\setminus(\Omega^{\ast}_{R_{0}}\setminus\Omega^{\ast}_{\varepsilon^{\tilde{\theta}}})}(2\mu e(\mathbf{u}_{1}^{\alpha}+\mathbf{u}_{2}^{\alpha}),e(\mathbf{u}_{1}^{\beta})),\\
\mathrm{III}_{3}=&\int_{\Omega^{\ast}_{R_{0}}\setminus\Omega^{\ast}_{\varepsilon^{\tilde{\theta}}}}(2\mu e(\mathbf{u}_{1}^{\ast\alpha}+\mathbf{u}_{2}^{\ast\alpha}),e(\mathbf{u}_{1}^{\ast\beta})).
\end{align*}
From \eqref{ZQWZW001}, we get
\begin{align}\label{GAP001}
|\mathrm{III}_{1}|\leq C\varepsilon^{\frac{1}{4}}.
\end{align}
Similarly as in \eqref{GAZ001}, we have
\begin{align}\label{GAZ00195}
|\mathrm{III}_{1}|\leq \int_{\varepsilon^{\tilde{\theta}}\leq |x'|\leq R_{0}}\frac{C\varepsilon(\varepsilon+|x'|^{2})^{1/2}}{|x'|^{2}}\leq C
\begin{cases}
\varepsilon|\ln\varepsilon|,&n=2,\\
\varepsilon,&n\geq3.
\end{cases}
\end{align}
With regard to $\mathrm{III}_{3}$, we have from Corollary \ref{coro00z} and \eqref{GZIJ} that
\begin{align}\label{HMGD001}
\mathrm{III}_{3}=&\int_{\Omega^{\ast}_{R_{0}}}(2\mu e(\mathbf{u}_{1}^{\ast\alpha}+\mathbf{u}_{2}^{\ast\alpha}),e(\mathbf{u}_{1}^{\ast\beta}))-\int_{\Omega^{\ast}_{\varepsilon^{\tilde{\theta}}}}(2\mu e(\mathbf{u}_{1}^{\ast\alpha}+\mathbf{u}_{2}^{\ast\alpha}),e(\mathbf{u}_{1}^{\ast\beta}))\notag\\
=&\int_{\Omega^{\ast}_{R_{0}}}(2\mu e(\mathbf{u}_{1}^{\ast\alpha}+\mathbf{u}_{2}^{\ast\alpha}),e(\mathbf{u}_{1}^{\ast\beta}))+O(1)\varepsilon^{n\tilde{\theta}}.
\end{align}
A combination of \eqref{GAP001}--\eqref{HMGD001} reads that
\begin{align}\label{RZMA001}
\mathrm{III}=&\int_{\Omega^{\ast}_{R_{0}}}(2\mu e(\mathbf{u}_{1}^{\ast\alpha}+\mathbf{u}_{2}^{\ast\alpha}),e(\mathbf{u}_{1}^{\ast\beta}))+O(1)\varepsilon^{\frac{1}{4}}.
\end{align}
It then follows from \eqref{GAZ001}--\eqref{GAZ0011} and \eqref{RZMA001} that
\begin{align*}
\sum\limits^{2}_{i=1}a_{i1}^{\alpha\beta}=\sum\limits^{2}_{i=1}a_{i1}^{\ast\alpha\beta}+O(1)\varepsilon^{\frac{1}{4}}.
\end{align*}
Analogously,
\begin{align*}
\sum\limits^{2}_{j=1}a_{1j}^{\alpha\beta}=\sum\limits^{2}_{j=1}a_{1j}^{*\alpha\beta}+O(\varepsilon^{\frac{1}{4}}).
\end{align*}
The proof is complete.

\end{proof}

In view of \eqref{AST123}, we know that for $\beta=1,2,...,\frac{n(n+1)}{2}$,
\begin{align}\label{GCT}
|\mathbf{u}_{1}^{\beta}-\mathbf{u}_{1}^{\ast\beta}|\leq C\varepsilon^{\frac{1}{3}},\quad\mathrm{in}\;D\setminus\big(\overline{D_{1}\cup D_{1}^{\ast}\cup D_{2}\cup D_{2}^{\ast}\cup\mathcal{C}_{\varepsilon^{\frac{1}{3}}}}\big).
\end{align}
Making use of the standard boundary estimates for Stokes equation with \eqref{GCT} and the fact that $\mathbf{u}_{1}^{\beta}-\mathbf{u}_{1}^{\ast\beta}=0$ on $\partial D$, we get
\begin{align*}
|\nabla(\mathbf{u}_{1}^{\beta}-\mathbf{u}_{1}^{\ast\beta})|+|p_{1}^{\beta}-p_{1}^{\ast\beta}|\leq C\varepsilon^{\frac{1}{3}},\quad\mathrm{on}\;\partial D,
\end{align*}
which yields that
\begin{align*}
|b_1^{\beta}-b_{1}^{\ast\beta}|\leq\left|\int_{\partial D}\boldsymbol{\varphi}\cdot\sigma[\mathbf{u}_{1}^{\beta}-\mathbf{u}_{1}^{\ast\beta},p_{1}^{\beta}-p_{1}^{\ast\beta}]\nu\right|\leq C\|\varphi\|_{C^{0}(\partial D)}\varepsilon^{\frac{1}{3}}.
\end{align*}
Similarly,
\begin{align*}
|b_{2}^{\beta}-b_{2}^{\ast\beta}|\leq C\|\varphi\|_{C^{0}(\partial D)}\varepsilon^{\frac{1}{3}}.
\end{align*}
Then we obtain
\begin{lemma}\label{KM323}
Assume as above. Then for a arbitrarily small $\varepsilon>0$,
\begin{align*}
b_{i}^{\beta}=b_{i}^{\ast\beta}+O(\varepsilon^{\frac{1}{3}}),\quad i=1,2,\;\beta=1,2,...,\frac{n(n+1)}{2},
\end{align*}
and thus
\begin{align*}
\sum\limits^{2}_{i=1}b_{i}^{\beta}=\sum\limits^{2}_{i=1}b_{i}^{\ast\beta}+O(\varepsilon^{\frac{1}{3}}),
\end{align*}
where $b_{i}^{\ast\beta}$ and $b_{i}^{\beta}$, $i=1,2,\beta=1,2,...,\frac{n(n+1)}{2}$ are, respectively, given by \eqref{LMZR} and \eqref{KAT001}.

\end{lemma}

A consequence of Lemmas \ref{lemmabc} and \ref{KM323} shows that
\begin{lemma}\label{COOO}
Assume as above. Let $C_{2}^{\alpha}$, $\alpha=1,2,...,\frac{n(n+1)}{2}$ be defined in \eqref{Decom}. Then for a sufficiently small $\varepsilon>0$,
\begin{align}\label{LGR001}
C_{2}^{\alpha}=&C_{\ast}^{\alpha}+O(r_{\varepsilon}),\quad\alpha=1,2,...,\frac{n(n+1)}{2},
\end{align}
where $C_{\ast}^{\alpha}$ is defined by \eqref{ZZWWWW}, $r_{\varepsilon}$ is defined in \eqref{JTD}.

\end{lemma}

For later use, we recall Lemma 6.1 of \cite{BLL2017} as follows.
\begin{lemma}\label{GLW}
Let $\boldsymbol{\Psi}$ the linear space of rigid displacement defined by \eqref{LAK01} with $n\geq2$. Assume that $\xi\in\boldsymbol{\Psi}$ vanishes at $n$ different points $\bar{x}_{1}$, $i=1,2,...,n$, which not lie on a $(n-1)$-dimensional plane. Then $\xi=0$.
\end{lemma}

\begin{proof}[Proof of Lemma \ref{COOO}]
To begin with, for $\alpha=1,2,...,\frac{n(n+1)}{2}$, $n\geq2$, let the elements of $\alpha$-th column of $\mathbb{D}$ given in \eqref{FAMZ91} be replaced by column vector $\Big(\sum\limits_{i=1}^{2}b_{i}^{1},...,\sum\limits_{i=1}^{2}b_{i}^{\frac{n(n+1)}{2}}\Big)^{T}$ and we then obtain new matrix $\mathbb{D}^{\ast\alpha}$ as follows:
\begin{gather*}
\mathbb{D}^{\alpha}=
\begin{pmatrix}
\sum\limits^{2}_{i,j=1}a_{ij}^{11}&\cdots&\sum\limits_{i=1}^{2}b_{i}^{1}&\cdots&\sum\limits^{2}_{i,j=1}a_{ij}^{1\,\frac{n(n+1)}{2}} \\\\ \vdots&\ddots&\vdots&\ddots&\vdots\\\\ \sum\limits^{2}_{i,j=1}a_{ij}^{\frac{n(n+1)}{2}\,1}&\cdots&\sum\limits_{i=1}^{2}b_{i}^{\frac{n(n+1)}{2}}&\cdots&\sum\limits^{2}_{i,j=1}a_{ij}^{\frac{n(n+1)}{2}\,\frac{n(n+1)}{2}}
\end{pmatrix}.
\end{gather*}
The proof is divided into two cases as follows.

{\bf Case 1.} Consider $n=2,3$. Denote
\begin{gather}\mathbb{A}_{0}=\begin{pmatrix} a_{11}^{n+1\,n+1}&\cdots&a_{11}^{n+1\frac{n(n+1)}{2}} \\\\ \vdots&\ddots&\vdots\\\\a_{11}^{\frac{n(n+1)}{2}n+1}&\cdots&a_{11}^{\frac{n(n+1)}{2}\frac{n(n+1)}{2}}\end{pmatrix},\label{HARBN001}\\
\mathbb{B}_{0}=\begin{pmatrix} \sum\limits^{2}_{i=1}a_{i1}^{n+1\,1}&\cdots&\sum\limits^{2}_{i=1}a_{i1}^{n+1\,\frac{n(n+1)}{2}} \\\\ \vdots&\ddots&\vdots\\\\ \sum\limits^{2}_{i=1}a_{i1}^{\frac{n(n+1)}{2}1}&\cdots&\sum\limits^{2}_{i=1}a_{i1}^{\frac{n(n+1)}{2}\frac{n(n+1)}{2}}\end{pmatrix},\notag\\
\mathbb{C}_{0}=\begin{pmatrix} \sum\limits^{2}_{j=1}a_{1j}^{1\,n+1}&\cdots&\sum\limits^{2}_{j=1}a_{1j}^{1\frac{n(n+1)}{2}} \\\\ \vdots&\ddots&\vdots\\\\ \sum\limits^{2}_{j=1}a_{1j}^{\frac{n(n+1)}{2}\,n+1}&\cdots&\sum\limits^{2}_{j=1}a_{1j}^{\frac{n(n+1)}{2}\frac{n(n+1)}{2}}\end{pmatrix}.\notag
\end{gather}
For $\alpha=1,2,...,\frac{n(n+1)}{2}$, by replacing the elements of $\alpha$-th column of $\mathbb{B}_{0}$ by $\Big(b_{1}^{n+1},...,b_{1}^{\frac{n(n+1)}{2}}\Big)^{T}$, we derive new matrix $\mathbb{B}_{0}^{\alpha}$ as follows:
\begin{gather*}
\mathbb{B}_{0}^{\alpha}=
\begin{pmatrix}
\sum\limits^{2}_{i=1}a_{i1}^{n+1\,1}&\cdots&b_{1}^{n+1}&\cdots&\sum\limits^{2}_{i=1}a_{i1}^{n+1\,\frac{n(n+1)}{2}} \\\\ \vdots&\ddots&\vdots&\ddots&\vdots\\\\ \sum\limits^{2}_{i=1}a_{i1}^{\frac{n(n+1)}{2}\,1}&\cdots&b_{1}^{\frac{n(n+1)}{2}}&\cdots&\sum\limits^{2}_{i=1}a_{i1}^{\frac{n(n+1)}{2}\frac{n(n+1)}{2}}
\end{pmatrix}.
\end{gather*}

Denote
\begin{align*}
\mathbb{F}_{0}=\begin{pmatrix} \mathbb{A}_{0}&\mathbb{B}_{0} \\  \mathbb{C}_{0}&\mathbb{D}
\end{pmatrix},\quad \mathbb{F}_{0}^{\alpha}=\begin{pmatrix} \mathbb{A}_{0}&\mathbb{B}_{0}^{\alpha} \\  \mathbb{C}_{0}&\mathbb{D}^{\alpha}
\end{pmatrix},\quad\alpha=1,2,...,\frac{n(n+1)}{2}.
\end{align*}
Making use of Lemmas \ref{lemmabc} and \ref{KM323}, we have
\begin{align*}
\det\mathbb{F}_{0}=\det\mathbb{F}_{0}^{\ast}+O(\varepsilon^{\frac{n-1}{24}}),\quad\det\mathbb{F}_{0}^{\alpha}=\det\mathbb{F}^{\ast\alpha}_{0}+O(\varepsilon^{\frac{n-1}{24}}),
\end{align*}
and then
\begin{align}\label{QGH01}
\frac{\det\mathbb{F}_{0}^{\alpha}
}{\det\mathbb{F}_{0}}=&\frac{\det\mathbb{F}_{0}^{\ast\alpha}}{\det\mathbb{F}_{0}^{\ast}}\frac{1}{1-{\frac{\det\mathbb{F}_{0}^{\ast}-\det\mathbb{F}_{0}}{\det\mathbb{F}_{0}^{\ast}}}}+\frac{\det\mathbb{F}_{0}^{\alpha}-\det\mathbb{F}_{0}^{\ast\alpha}}{\det\mathbb{F}_{0}}\notag\\
=&\frac{\det\mathbb{F}_{0}^{\ast\alpha}}{\det\mathbb{F}_{0}^{\ast}}(1+O(\varepsilon^{\frac{n-1}{24}})).
\end{align}
Claim that $\det\mathbb{F}_{0}^{\ast}\neq0$. In fact, for any $\xi=(\xi_{1},\xi_{2},...,\xi_{n^{2}})^{T}\neq0$, we have
\begin{align*}
\xi^{T}\mathbb{F}_{0}^{\ast}\xi=&\int_{\Omega^{\ast}}\Bigg(2\mu e\bigg(\sum^{\frac{n(n+1)}{2}}_{\alpha=n+1}\xi_{\alpha-n}\mathbf{u}_{1}^{\ast\alpha}+\sum^{\frac{n(n+1)}{2}}_{\alpha=1}\xi_{\alpha+\frac{n(n-1)}{2}}(\mathbf{u}_{1}^{\ast\alpha}+\mathbf{u}_{2}^{\ast\alpha})\bigg),\notag\\
&\quad\quad\quad\quad e\bigg(\sum^{\frac{n(n+1)}{2}}_{\beta=n+1}\xi_{\beta-n}\mathbf{u}_{1}^{\ast\beta}+\sum^{\frac{n(n+1)}{2}}_{\beta=1}\xi_{\beta+\frac{n(n-1)}{2}}(\mathbf{u}_{1}^{\ast\beta}+\mathbf{u}_{2}^{\ast\beta})\bigg)\Bigg)\notag\\
=&2\mu\int_{\Omega^{\ast}}\bigg|e\bigg(\sum^{\frac{n(n+1)}{2}}_{\alpha=n+1}\xi_{\alpha-n}\mathbf{u}_{1}^{\ast\alpha}+\sum^{\frac{n(n+1)}{2}}_{\alpha=1}\xi_{\alpha+\frac{n(n-1)}{2}}(\mathbf{u}_{1}^{\ast\alpha}+\mathbf{u}_{2}^{\ast\alpha})\bigg)\bigg|^{2}>0,
\end{align*}
where in the last inequality we used the fact that $$e\bigg(\sum^{\frac{n(n+1)}{2}}_{\alpha=n+1}\xi_{\alpha-n}\mathbf{u}_{1}^{\ast\alpha}+\sum^{\frac{n(n+1)}{2}}_{\alpha=1}\xi_{\alpha+\frac{n(n-1)}{2}}(\mathbf{u}_{1}^{\ast\alpha}+\mathbf{u}_{2}^{\ast\alpha})\bigg)$$ is not identically zero in $\Omega^{\ast}$. Reasoning by contradiction, assume that $$e\bigg(\sum^{\frac{n(n+1)}{2}}_{\alpha=n+1}\xi_{\alpha-n}\mathbf{u}_{1}^{\ast\alpha}+\sum^{\frac{n(n+1)}{2}}_{\alpha=1}\xi_{\alpha+\frac{n(n-1)}{2}}(\mathbf{u}_{1}^{\ast\alpha}+\mathbf{u}_{2}^{\ast\alpha})\bigg)=0,\quad\mathrm{in}\;\Omega^{\ast},$$ then
\begin{align}\label{DAK}
\sum^{\frac{n(n+1)}{2}}_{\alpha=n+1}\xi_{\alpha-n}\mathbf{u}_{1}^{\ast\alpha}+\sum^{\frac{n(n+1)}{2}}_{\alpha=1}\xi_{\alpha+\frac{n(n-1)}{2}}(\mathbf{u}_{1}^{\ast\alpha}+\mathbf{u}_{2}^{\ast\alpha})=\sum^{\frac{n(n+1)}{2}}_{i=1}a_{i}\boldsymbol{\psi}_{i},\quad\mathrm{in}\;\Omega^{\ast},
\end{align}
for some constants $a_{i}$, $i=1,2,...,\frac{n(n+1)}{2}$. Since $\mathbf{u}_{1}^{\ast\alpha}=\mathbf{u}_{2}^{\ast\alpha}=0$ on $\partial D$, it follows from \eqref{DAK} that $\sum^{\frac{n(n+1)}{2}}_{i=1}a_{i}\boldsymbol{\psi}_{i}=0$, which implies that $a_{i}=0$, $i=1,2,...,\frac{n(n+1)}{2}$. Since
\begin{align*}
0=&\sum^{\frac{n(n+1)}{2}}_{\alpha=n+1}\xi_{\alpha-n}\mathbf{u}_{1}^{\ast\alpha}+\sum^{\frac{n(n+1)}{2}}_{\alpha=1}\xi_{\alpha+\frac{n(n-1)}{2}}(\mathbf{u}_{1}^{\ast\alpha}+\mathbf{u}_{2}^{\ast\alpha})\notag\\
=&
\begin{cases}
\sum\limits^{n}_{\alpha=1}\xi_{\alpha+\frac{n(n-1)}{2}}\boldsymbol{\psi}_{\alpha}+\sum\limits^{\frac{n(n+1)}{2}}_{\alpha=n+1}(\xi_{\alpha-n}+\xi_{\alpha+\frac{n(n-1)}{2}})\boldsymbol{\psi}_{\alpha},&\mathrm{on}\;\partial D_{1}^{\ast},\\
\sum\limits^{\frac{n(n+1)}{2}}_{\alpha=1}\xi_{\alpha+\frac{n(n-1)}{2}}\boldsymbol{\psi}_{\alpha},&\mathrm{on}\;\partial D_{2}^{\ast},
\end{cases}
\end{align*}
then we obtain that $\xi=0$. This is a contradiction.

Then applying Cramer's rule to \eqref{PLA001}, we deduce from \eqref{QGH01} that for $\alpha=1,2,...,\frac{n(n+1)}{2}$,
\begin{align*}
C_{2}^{\alpha}=&\frac{\det\mathbb{F}_{0}^{\alpha}}{\det \mathbb{F}_{0}}
\begin{cases}
1+O(\varepsilon^{\frac{1}{2}}|\ln\varepsilon|),&n=2,\\
1+O(|\ln\varepsilon|^{-1}),&n=3
\end{cases}
\notag\\
=&\frac{\det\mathbb{F}_{0}^{\ast\alpha}}{\det \mathbb{F}_{0}^{\ast}}
\begin{cases}
1+O(\varepsilon^{\frac{1}{24}}),&n=2,\\
1+O(|\ln\varepsilon|^{-1}),&n=3.
\end{cases}
\end{align*}

{\bf Case 2.} Consider $n>3$. For $\alpha=1,2,...,\frac{n(n+1)}{2}$, let $\big(b_{1}^{1},...,b_{1}^{\frac{n(n+1)}{2}}\big)^{T}$ substitute for the elements of $\alpha$-th column $\mathbb{B}$ and we then derive new matrix $\mathbb{B}_{1}^{\alpha}$ as follows:
\begin{gather*}
\mathbb{B}_{1}^{\alpha}=
\begin{pmatrix}
\sum\limits^{2}_{i=1}a_{i1}^{11}&\cdots&b_{1}^{1}&\cdots&\sum\limits^{2}_{i=1}a_{i1}^{1\,\frac{n(n+1)}{2}} \\\\ \vdots&\ddots&\vdots&\ddots&\vdots\\\\ \sum\limits^{2}_{i=1}a_{i1}^{\frac{n(n+1)}{2}\,1}&\cdots&b_{1}^{\frac{n(n+1)}{2}}&\cdots&\sum\limits^{2}_{i=1}a_{i1}^{\frac{n(n+1)}{2}\frac{n(n+1)}{2}}
\end{pmatrix}.
\end{gather*}
Denote
\begin{align*}
\mathbb{F}_{1}=\begin{pmatrix} \mathbb{A}&\mathbb{B} \\  \mathbb{C}&\mathbb{D}
\end{pmatrix},\quad\mathbb{F}^{\alpha}_{1}=\begin{pmatrix} \mathbb{A}&\mathbb{B}_{1}^{\alpha} \\  \mathbb{C}&\mathbb{D}^{\alpha}
\end{pmatrix},\;\,\alpha=1,2,...,\frac{n(n+1)}{2}.
\end{align*}
Making use of Lemmas \ref{lemmabc} and \ref{KM323}, we deduce that for $\alpha=1,2,...,\frac{n(n+1)}{2}$,
\begin{align*}
\det\mathbb{F}_{1}=\det\mathbb{F}_{1}^{\ast}+O(\varepsilon^{\min\{\frac{1}{12},\frac{n-3}{24}\}}),\quad\det\mathbb{F}_{1}^{\alpha}=\det\mathbb{F}^{\ast\alpha}_{1}+O(\varepsilon^{\min\{\frac{1}{12},\frac{n-3}{24}\}}).
\end{align*}
Analogously as above, it follows that $\det\mathbb{F}^{\ast}_{1}\neq0$. Then for $\alpha=1,2,...,\frac{n(n+1)}{2},$
\begin{align*}
\frac{\det\mathbb{F}_{1}^{\alpha}
}{\det\mathbb{F}_{1}}=&\frac{\det\mathbb{F}_{1}^{\ast\alpha}}{\det\mathbb{F}_{1}^{\ast}}\frac{1}{1-{\frac{\det\mathbb{F}_{1}^{\ast}-\det\mathbb{F}_{1}}{\det\mathbb{F}_{1}^{\ast}}}}+\frac{\det\mathbb{F}_{1}^{\alpha}-\det\mathbb{F}_{1}^{\ast\alpha}}{\det\mathbb{F}_{1}}\notag\\
=&\frac{\det\mathbb{F}^{\ast\alpha}_{1}}{\det\mathbb{F}_{1}^{\ast}}(1+O(\varepsilon^{\min\{\frac{1}{12},\frac{n-3}{24}\}})).
\end{align*}
Hence, applying Cramer's rule to \eqref{PLA001} again, we derive
\begin{align*}
C_{2}^{\alpha}=&\frac{\det\mathbb{F}_{1}^{\alpha}}{\det \mathbb{F}_{1}}=\frac{\det\mathbb{F}_{1}^{\ast\alpha}}{\det \mathbb{F}^{\ast}_{1}}(1+O(\varepsilon^{\min\{\frac{1}{12},\frac{n-3}{24}\}})),\quad\alpha=1,2,...,\frac{n(n+1)}{2}.
\end{align*}
The proof is complete.

\end{proof}

Based on the results above, we are led to give the proof of Theorem \ref{JGR}.
\begin{proof}[Proof of Theorem \ref{JGR}.]
Using linearity, we split the solution $(\mathbf{u}_{b}^{\ast},p_{b}^{\ast})$ of problem \eqref{LCRD001} as follows:
\begin{align}\label{KTG001}
\mathbf{u}_{b}^{\ast}=\sum^{\frac{n(n+1)}{2}}_{\alpha=1}C_{\ast}^{\alpha}(\mathbf{u}_{1}^{\ast\alpha}+\mathbf{u}_{2}^{\ast\alpha})+\mathbf{u}_{0}^{\ast},\quad p_{b}^{\ast}=\sum^{\frac{n(n+1)}{2}}_{\alpha=1}C_{\ast}^{\alpha}(p_{1}^{\ast\alpha}+p_{2}^{\ast\alpha})+p_{0}^{\ast},
\end{align}
where $(\mathbf{u}_{0}^{\ast},p_{0}^{\ast})$ and $(\mathbf{u}_{i}^{\ast\alpha},p_{i}^{\ast\alpha})$, $i=1,2,\,\alpha=1,2,...,\frac{n(n+1)}{2}$  are, respectively, the solutions of \eqref{ZG001} and \eqref{qaz001111}. Therefore, combining \eqref{AZQ001}, \eqref{KTG001}, Lemmas \ref{KM323} and \ref{COOO}, we have from integration by parts that
\begin{align*}
\mathcal{B}_{\beta}[\boldsymbol{\varphi}]-\mathcal{B}_{\beta}^{\ast}[\boldsymbol{\varphi}]=&\int_{\partial D_{1}}\boldsymbol{\psi}_{\beta}\cdot\sigma[\mathbf{u}_{b},p_{b}]\nu-\int_{\partial D_{1}^{\ast}}\boldsymbol{\psi}_{\beta}\cdot\sigma[\mathbf{u}^{\ast}_{b},p^{\ast}_{b}]\nu\notag\\
=&\sum^{\frac{n(n+1)}{2}}_{\alpha=1}(C_{\ast}^{\alpha}-C_{2}^{\alpha})\sum^{2}_{i=1}a_{i1}^{\alpha\beta}+\sum^{\frac{n(n+1)}{2}}_{\alpha=1}C_{\ast}^{\alpha}\left(\sum^{2}_{i=1}a_{i1}^{\ast\alpha\beta}-\sum^{2}_{i=1}a_{i1}^{\alpha\beta}\right)\notag\\
&+b_{1}^{\beta}-b_{1}^{\ast\beta}=O(r_{\varepsilon}),
\end{align*}
where $r_{\varepsilon}$ is given by \eqref{JTD}. The proof is finished.

\end{proof}

\section{Optimal gradient estimates and asymptotics }\label{KGRA90}

This section is to establish the optimal gradient estimates and asymptotic formulas for solution to Stokes problem \eqref{La.002} by using the united stress concentration factors captured in Theorem \ref{JGR}.

First, for $\alpha=1,2,...,n$, denote
\begin{gather}\label{PTCA001}
\mathbb{A}^{\ast\alpha}_{1}=\begin{pmatrix}\mathcal{B}_{\alpha}^{\ast}[\varphi]&a_{11}^{\ast\alpha\,n+1}&\cdots&a_{11}^{\ast\alpha\frac{n(n+1)}{2}} \\ \mathcal{B}_{n+1}^{\ast}[\varphi]&a_{11}^{\ast n+1\,n+1}&\cdots&a_{11}^{\ast n+1\frac{n(n+1)}{2}}\\ \vdots&\vdots&\ddots&\vdots\\ \mathcal{B}_{\frac{n(n+1)}{2}}^{\ast}[\varphi]&a_{11}^{\ast\frac{n(n+1)}{2}d+1}&\cdots&a_{11}^{\ast\frac{n(n+1)}{2}\frac{n(n+1)}{2}}
\end{pmatrix}.
\end{gather}
Second, for $\alpha=n+1,...,\frac{n(n+1)}{2}$, by substituting $\big(\mathcal{B}_{n+1}^{\ast}[\varphi],...,\mathcal{B}_{\frac{n(n+1)}{2}}^{\ast}[\varphi]\big)^{T}$ for the elements of $\alpha$-th column of $\mathbb{A}_{0}^{\ast}$ given in \eqref{LAGT001}, we get the following new matrix
\begin{gather}\label{DKY001}
\mathbb{A}_{2}^{\ast\alpha}=
\begin{pmatrix}
a_{11}^{\ast n+1\,n+1}&\cdots&\mathcal{B}_{n+1}^{\ast}[\varphi]&\cdots&a_{11}^{\ast n+1\,\frac{n(n+1)}{2}} \\\\ \vdots&\ddots&\vdots&\ddots&\vdots\\\\a_{11}^{\ast\frac{n(n+1)}{2}\,n+1}&\cdots&\mathcal{B}_{\frac{n(n+1)}{2}}^{\ast}[\varphi]&\cdots&a_{11}^{\ast\frac{n(n+1)}{2}\,\frac{n(n+1)}{2}}
\end{pmatrix}.
\end{gather}
Finally, for $\alpha=1,2,...,\frac{n(n+1)}{2}$, we replace the elements of $\alpha$-th column of $\mathbb{A}^{\ast}$ defined in \eqref{WZW} by $\big(\mathcal{B}_{1}^{\ast}[\varphi],...,\mathcal{B}^{\ast}_{\frac{n(n+1)}{2}}[\varphi]\big)^{T}$ and obtain new matrix $\mathbb{A}_{3}^{\ast\alpha}$ as follows:
\begin{gather}\label{JGAT001}
\mathbb{A}_{3}^{\ast\alpha}=
\begin{pmatrix}
a_{11}^{\ast11}&\cdots&\mathcal{B}_{1}^{\ast}[\varphi]&\cdots&a_{11}^{\ast1\,\frac{n(n+1)}{2}} \\\\ \vdots&\ddots&\vdots&\ddots&\vdots\\\\a_{11}^{\ast\frac{n(n+1)}{2}\,1}&\cdots&\mathcal{B}_{\frac{n(n+1)}{2}}^{\ast}[\varphi]&\cdots&a_{11}^{\ast\frac{n(n+1)}{2}\,\frac{n(n+1)}{2}}
\end{pmatrix}.
\end{gather}

Then we have
\begin{theorem}\label{MAINZW002}
Suppose that $D_{1},D_{2}\subset D\subset\mathbb{R}^{n}\,(n\geq2)$ are defined as above. For given $\boldsymbol{\varphi}\in C^{2}(\partial D;\mathbb{R}^{n})$, let $\mathbf{u}\in H^{1}(D;\mathbb{R}^{n})\cap C^{1}(\overline{\Omega};\mathbb{R}^{n})$ and $p\in L^{2}(D)\cap (\overline{\Omega})$ be the solution of \eqref{La.002} and \eqref{COM001}. Then for a arbitrarily small $\varepsilon>0$ and $x=(x',x_{n})\in\Omega_{R_{0}}$,

$(i)$ in the case of $n=2,3,$
\begin{align*}
\sigma[\mathbf{u},p-(q)_{\delta;x'}]=&\sum^{n}_{\alpha=1}\frac{\det\mathbb{A}_{1}^{\ast\alpha}}{\det \mathbb{A}_{0}^{\ast}}\frac{\rho_{n}^{-1}(\varepsilon)}{\mathcal{L}_{\alpha}(\frac{\pi}{2\kappa})^{\frac{n-1}{2}}}\frac{1+O(r_{\varepsilon})}{1+\widetilde{\mathcal{G}}_{n}^{\ast\alpha}\rho^{-1}_{n}(\varepsilon)}\sigma[\bar{\mathbf{u}}_{1}^{\alpha},\bar{p}_{1}^{\alpha}]\notag\\
&+\sum^{\frac{n(n+1)}{2}}_{\alpha=n+1}\frac{\det\mathbb{A}_{2}^{\ast\alpha}}{\det \mathbb{A}^{\ast}_{0}}\big(1+O(r_{\varepsilon})\big)\sigma[\bar{\mathbf{u}}_{1}^{\alpha},\bar{p}_{1}^{\alpha}]+O(1)\rho^{-1}_{n}(\varepsilon)\delta^{-1/2};
\end{align*}

$(ii)$ in the case of $n>3$,
\begin{align*}
\sigma[\mathbf{u},p-(q)_{\delta;x'}]=&\sum^{\frac{n(n+1)}{2}}_{\alpha=1}\frac{\det\mathbb{A}_{3}^{\ast\alpha}}{\det \mathbb{A}^{\ast}}(1+O(\varepsilon^{\min\{\frac{1}{12},\frac{n-3}{24}\}}))\sigma[\bar{\mathbf{u}}_{1}^{\alpha},\bar{p}_{1}^{\alpha}]+O(1)\delta^{-1/2},
\end{align*}
where $(q)_{\delta;x'}:=\frac{1}{|\Omega_{\delta}(x')|}\int_{\Omega_{\delta}(x')}q(y)dy$ with $q$ defined by \eqref{QQQ001} below, $r_{\varepsilon}$ and $\delta$ are, respectively, defined by \eqref{JTD} and \eqref{DEL010}, $\bar{\mathbf{u}}_{1}^{\alpha}$ and $\bar{p}_{1}^{\alpha}$, $\alpha=1,2,...,\frac{n(n+1)}{2}$ are, respectively, given in \eqref{zzwz002} and \eqref{PMAIN002}, $\mathcal{L}_{\alpha}$, $\alpha=1,2,...,\frac{n(n+1)}{2}$ and $\rho_{n}(\varepsilon)$ are, respectively, given by \eqref{CONT001}--\eqref{rate00}, $\mathbb{A}^{\ast}$ and $\mathbb{A}_{0}^{\ast}$ are given by \eqref{WZW} and \eqref{LAGT001}, respectively, $\mathbb{A}_{1}^{\ast\alpha}$, $\alpha=1,2,...,n$, $\mathbb{A}_{2}^{\ast\alpha}$, $\alpha=n+1,...,\frac{n(n+1)}{2}$, $\mathbb{A}_{3}^{\ast\alpha}$, $\alpha=1,2,...,\frac{n(n+1)}{2}$ are, respectively, given by \eqref{PTCA001}--\eqref{JGAT001}, the geometry constants $\widetilde{\mathcal{G}}_{n}^{\ast\alpha}$, $\alpha=1,2,...,n,\,n=2,3$ are defined by \eqref{GEOT001}.

\end{theorem}
\begin{remark}\label{ANALY001}
According to the definitions of $\bar{\mathbf{u}}_{1}^{\alpha}$ and $\bar{p}_{1}^{\alpha}$, $\alpha=1,2,...,\frac{n(n+1)}{2}$, we know that the main singularity of $\nabla\bar{\mathbf{u}}_{1}^{\alpha}$ is of order $O(\boldsymbol{\psi}_{\alpha}\delta^{-1})$ with the blow-up rates of $\varepsilon^{-1}$ and $\varepsilon^{-1/2}$ achieving at the $(n-1)$-dimensional ball $\{|x'|\leq\sqrt{\varepsilon}\}$, which correspond to the case of $\alpha=1,2,...,n$ and $\alpha=n+1,...,\frac{n(n+1)}{2}$, respectively. By contrast, the leading singularity of $\bar{p}_{1}^{\alpha}$, $\alpha=1,2,...,\frac{n(n+1)}{2}$ is determined by the $n$-th term $\bar{p}_{1}^{n}$ with its blow-up rates of order $O(\varepsilon^{-2})$ attaining at $\{|x'|\leq\sqrt{\varepsilon}\}$, which is greater than that of $\nabla\bar{\mathbf{u}}_{1}^{\alpha}$. Then combining the asymptotic results in Theorem \ref{MAINZW002}, we obtain that the largest blow-up rate of $\sigma[\mathbf{u},p-(q)_{\delta;x'}]$ arises from the pressure and is of order $O(\rho_{n}^{-1}(\varepsilon)\varepsilon^{-2})$, see Corollary \ref{MGA001} below for further details. In addition, the phenomena that the greatest singularity is determined by the pressure has also been revealed in previous work \cite{LWZ2020} on the blow-up of hydrodynamic forces, see the explicit singular functions corresponding to the pressure in Sections 4.3 and 4.6 of \cite{LWZ2020} for further details.

\end{remark}

\begin{proof}

{\bf Step 1.} The first step is devoted to giving a precise computation for $C_{1}^{\alpha}-C_{2}^{\alpha}$, $\alpha=1,2,...,\frac{n(n+1)}{2}$.

{\bf Case 1.} Consider the case when $n=2,3$. For $\alpha=1,2,...,n$, define
\begin{gather*}
\mathbb{A}^{\alpha}_{1}=\begin{pmatrix}\mathcal{B}_{\alpha}[\varphi]&a_{11}^{\alpha\,n+1}&\cdots&a_{11}^{\alpha\frac{n(n+1)}{2}} \\ \mathcal{B}_{n+1}[\varphi]&a_{11}^{n+1\,n+1}&\cdots&a_{11}^{n+1\frac{n(n+1)}{2}}\\ \vdots&\vdots&\ddots&\vdots\\ \mathcal{B}_{\frac{n(n+1)}{2}}[\varphi]&a_{11}^{\frac{n(n+1)}{2}d+1}&\cdots&a_{11}^{\frac{n(n+1)}{2}\frac{n(n+1)}{2}}
\end{pmatrix}.
\end{gather*}
For $\alpha=n+1,...,\frac{n(n+1)}{2}$, let $\big(\mathcal{B}_{d+1}[\varphi],...,\mathcal{B}_{\frac{n(n+1)}{2}}[\varphi]\big)^{T}$ substitute for the elements of $\alpha$-th column of $\mathbb{A}_{0}$ given in \eqref{HARBN001} and then generate new matrix $\mathbb{A}_{2}^{\alpha}$ as follows:
\begin{gather*}
\mathbb{A}_{2}^{\alpha}=
\begin{pmatrix}
a_{11}^{d+1\,d+1}&\cdots&\mathcal{B}_{d+1}[\varphi]&\cdots&a_{11}^{d+1\,\frac{n(n+1)}{2}} \\\\ \vdots&\ddots&\vdots&\ddots&\vdots\\\\a_{11}^{\frac{n(n+1)}{2}\,d+1}&\cdots&\mathcal{B}_{\frac{n(n+1)}{2}}[\varphi]&\cdots&a_{11}^{\frac{n(n+1)}{2}\,\frac{n(n+1)}{2}}
\end{pmatrix}.
\end{gather*}

Denote
\begin{align}\label{GEOT001}
\widetilde{\mathcal{G}}_{n}^{\ast\alpha}:=\frac{\mathcal{G}_{n}^{\ast\alpha}}{\mathcal{L}_{\alpha}(\frac{\pi}{2\kappa})^{\frac{n-1}{2}}},\quad\alpha=1,2,...,n,\,n=2,3,
\end{align}
where $\mathcal{G}_{n}^{\ast\alpha}$, $\alpha=1,2,...,n,\,n=2,3$ are defined by \eqref{NZWA001}. A combination of Theorem \ref{JGR} and Lemma \ref{lemmabc} shows that
\begin{align}\label{MZGA001}
\frac{1}{a_{11}^{\alpha\alpha}}=&\frac{\rho_{n}^{-1}(\varepsilon)}{\mathcal{L}_{\alpha}(\frac{\pi}{2\kappa})^{\frac{n-1}{2}}}\frac{1}{1-\frac{\mathcal{L}_{\alpha}(\frac{\pi}{2\kappa})^{\frac{n-1}{2}}-a_{11}^{\alpha\alpha}\rho_{n}^{-1}(\varepsilon)}{\mathcal{L}_{\alpha}(\frac{\pi}{2\kappa})^{\frac{n-1}{2}}}}\notag\\
=&\frac{\rho_{n}^{-1}(\varepsilon)}{\mathcal{L}_{\alpha}(\frac{\pi}{2\kappa})^{\frac{n-1}{2}}}\frac{1}{1+\widetilde{\mathcal{G}}_{n}^{\ast\alpha}\rho^{-1}_{n}(\varepsilon)+O(1)\rho_{n}^{-1}(\varepsilon)\max\{\varepsilon^{\frac{1}{12}},\varrho_{\alpha,n}(\varepsilon)\}}\notag\\
=&\frac{\rho_{n}^{-1}(\varepsilon)}{\mathcal{L}_{\alpha}(\frac{\pi}{2\kappa})^{\frac{n-1}{2}}}\frac{1+O(\rho_{n}^{-1}(\varepsilon)\varrho_{\alpha,n}(\varepsilon))}{1+\widetilde{\mathcal{G}}_{n}^{\ast\alpha}\rho^{-1}_{n}(\varepsilon)},\quad\alpha=1,2,...,n,\,n=2,3,
\end{align}
and
\begin{align}\label{MZGA002}
\begin{cases}
\det\mathbb{A}_{0}=\det\mathbb{A}_{0}^{\ast}+O(\varepsilon^{\frac{n-1}{24}}),\\
\det\mathbb{A}_{1}^{\alpha}=\det\mathbb{A}^{\ast\alpha}_{1}+O(r_{\varepsilon}),&\alpha=1,2,...,n,\\
\det\mathbb{A}_{2}^{\alpha}=\det\mathbb{A}^{\ast\alpha}_{2}+O(r_{\varepsilon}),&\alpha=n+1,...,\frac{n(n+1)}{2},
\end{cases}
\end{align}
where $r_{\varepsilon}$ is given by \eqref{JTD}. Denote
\begin{align}\label{OME001}
\omega_{n}(\varepsilon):=
\begin{cases}
|\ln\varepsilon|,&n=2,\\
1,&n=3.
\end{cases}
\end{align}
Utilizing Cramer's rule for \eqref{JGRO001}, it then follows from \eqref{MZGA001} and \eqref{MZGA002} that for $n=2,3,$
\begin{align}\label{LAMNZ001}
C_{1}^{\alpha}-C_{2}^{\alpha}=&
\begin{cases}
\frac{\prod\limits_{i\neq\alpha}^{n}a_{11}^{ii}\det\mathbb{A}_{1}^{\alpha}}{\prod\limits_{i=1}^{n}a_{11}^{ii}\det \mathbb{A}_{0}}\big(1+O(\rho^{-1}_{n}(\varepsilon)\omega_{n}(\varepsilon))\big),\quad\alpha=1,2,...,n,\\
\frac{\det\mathbb{A}_{2}^{\alpha}}{\det \mathbb{A}_{0}}\big(1+O(\rho_{n}^{-1}(\varepsilon))\big),\quad\quad\quad\alpha=n+1,...,\frac{n(n+1)}{2}
\end{cases}\notag\\
=&\begin{cases}
\frac{\rho_{n}^{-1}(\varepsilon)}{\mathcal{L}_{\alpha}(\frac{\pi}{2\kappa})^{\frac{n-1}{2}}}\frac{\det\mathbb{A}_{1}^{\ast\alpha}}{\det \mathbb{A}_{0}^{\ast}}\frac{1+O(r_{\varepsilon})}{1+\widetilde{\mathcal{G}}_{n}^{\ast\alpha}\rho^{-1}_{n}(\varepsilon)},\quad\alpha=1,2,...,n, \vspace{0.5ex} \\
\frac{\det\mathbb{A}_{2}^{\ast\alpha}}{\det \mathbb{A}^{\ast}_{0}}\big(1+O(r_{\varepsilon})\big),\quad\quad\quad\alpha=n+1,...,\frac{n(n+1)}{2},
\end{cases}
\end{align}
where $\rho_{n}(\varepsilon)$ and $\omega_{n}(\varepsilon)$ are, respectively, given by \eqref{rate00} and \eqref{OME001}.

{\bf Case 2.} Consider the case when $n>3$. For $\alpha=1,2,...,\frac{n(n+1)}{2}$, let the elements of $\alpha$-th column of $\mathbb{A}$ given in \eqref{GGDA01} be replaced by $\big(\mathcal{B}_{1}[\varphi],...,\mathcal{B}_{\frac{n(n+1)}{2}}[\varphi]\big)^{T}$ and we thus obtain new matrix $\mathbb{A}_{3}^{\alpha}$ as follows:
\begin{gather*}
\mathbb{A}_{3}^{\alpha}=
\begin{pmatrix}
a_{11}^{11}&\cdots&\mathcal{B}_{1}[\varphi]&\cdots&a_{11}^{1\,\frac{n(n+1)}{2}} \\\\ \vdots&\ddots&\vdots&\ddots&\vdots\\\\a_{11}^{\frac{n(n+1)}{2}\,1}&\cdots&\mathcal{B}_{\frac{n(n+1)}{2}}[\varphi]&\cdots&a_{11}^{\frac{n(n+1)}{2}\,\frac{n(n+1)}{2}}
\end{pmatrix}.
\end{gather*}
Hence, using Theorem \ref{JGR} and Lemma \ref{lemmabc} again, we deduce that for $\alpha=1,2,...,\frac{n(n+1)}{2}$,
\begin{align*}
\det\mathbb{A}=\det\mathbb{A}^{\ast}+O(\varepsilon^{\min\{\frac{1}{12},\frac{n-3}{24}\}}),\quad \det\mathbb{A}_{3}^{\alpha}=\det\mathbb{A}_{3}^{\ast\alpha}+O(\varepsilon^{\min\{\frac{1}{12},\frac{n-3}{24}\}}).
\end{align*}
Applying Cramer's rule to \eqref{JGRO001} again, we have
\begin{align}\label{GMARZT001}
C_{1}^{\alpha}-C_{2}^{\alpha}=&\frac{\det\mathbb{A}_{3}^{\alpha}}{\det \mathbb{A}}=\frac{\det\mathbb{A}_{3}^{\ast\alpha}}{\det \mathbb{A}^{\ast}}(1+O(\varepsilon^{\min\{\frac{1}{12},\frac{n-3}{24}\}})).
\end{align}

\noindent{\bf Step 2.}
A consequence of Corollary \ref{coro00z} and \eqref{LGR001} is that
\begin{align}\label{KGTA}
\begin{cases}
|\nabla \mathbf{u}_{b}|=\sum\limits_{\alpha=1}^{\frac{n(n+1)}{2}}|C_{2}^\alpha\nabla(\mathbf{u}_{1}^\alpha+\mathbf{u}_{2}^\alpha)|+|\nabla\mathbf{u}_{0}|\leq C\delta^{-\frac{n}{2}}e^{-\frac{1}{2C\delta^{1/2}}},\vspace{0.5ex} \\
|p_{b}|=\sum\limits_{\alpha=1}^{\frac{n(n+1)}{2}}|C_{2}^\alpha(p_{1}^\alpha+p_{2}^\alpha)|+|p_{0}|\leq C\delta^{-\frac{n}{2}}e^{-\frac{1}{2C\delta^{1/2}}},
\end{cases}\quad\mathrm{in}\;\Omega_{R_{0}}.
\end{align}
For simplicity, denote
\begin{align*}
\tilde{r}_{\alpha}(\delta):=\begin{cases}
1,&\alpha=1,...,n-1,\\
\delta^{-1/2},&\alpha=n,\\
1,&\alpha=n+1,...,2n-1,\\
\delta^{1/2},&\alpha=2n,...,\frac{n(n+1)}{2},\,n\geq3,
\end{cases}
\end{align*}
and
\begin{align}\label{QQQ001}
q:=\sum^{\frac{n(n+1)}{2}}_{\alpha=1}(C_{1}^{\alpha}-C_{2}^{\alpha})(p_{1}^{\alpha}-\bar{p}_{1}^{\alpha}),
\end{align}
where the exact values of $C_{1}^{\alpha}-C_{2}^{\alpha}$, $\alpha=1,2,...,\frac{n(n+1)}{2}$ are given by \eqref{LAMNZ001} and \eqref{GMARZT001}. According to \eqref{Decom002}, it follows from \eqref{Le2.025} and \eqref{LAMNZ001}--\eqref{KGTA} that for $x\in\Omega_{R_{0}}$,

$(i)$ in the case of $n=2,3,$
\begin{align}\label{EQUA001}
\nabla\mathbf{u}=&\sum^{n}_{\alpha=1}\frac{\det\mathbb{A}_{1}^{\ast\alpha}}{\det \mathbb{A}_{0}^{\ast}}\frac{\rho_{n}^{-1}(\varepsilon)}{\mathcal{L}_{\alpha}(\frac{\pi}{2\kappa})^{\frac{n-1}{2}}}\frac{1+O(r_{\varepsilon})}{1+\widetilde{\mathcal{G}}_{n}^{\ast\alpha}\rho^{-1}_{n}(\varepsilon)}(\nabla\bar{\mathbf{u}}_{1}^{\alpha}+O(\tilde{r}_{\alpha}(\delta)))\notag\\
&+\sum^{\frac{n(n+1)}{2}}_{\alpha=n+1}\frac{\det\mathbb{A}_{2}^{\ast\alpha}}{\det \mathbb{A}^{\ast}_{0}}\big(1+O(r_{\varepsilon})\big)(\nabla\bar{\mathbf{u}}_{1}^{\alpha}+O(\tilde{r}_{\alpha}(\delta)))+O(1)\delta^{-\frac{n}{2}}e^{-\frac{1}{2C\delta^{1/2}}}\notag\\
=&\sum^{n}_{\alpha=1}\frac{\det\mathbb{A}_{1}^{\ast\alpha}}{\det \mathbb{A}_{0}^{\ast}}\frac{\rho_{n}^{-1}(\varepsilon)}{\mathcal{L}_{\alpha}(\frac{\pi}{2\kappa})^{\frac{n-1}{2}}}\frac{1+O(r_{\varepsilon})}{1+\widetilde{\mathcal{G}}_{n}^{\ast\alpha}\rho^{-1}_{n}(\varepsilon)}\nabla\bar{\mathbf{u}}_{1}^{\alpha}\notag\\
&+\sum^{\frac{n(n+1)}{2}}_{\alpha=n+1}\frac{\det\mathbb{A}_{2}^{\ast\alpha}}{\det \mathbb{A}^{\ast}_{0}}\big(1+O(r_{\varepsilon})\big)\nabla\bar{\mathbf{u}}_{1}^{\alpha}+O(1)\rho^{-1}_{n}(\varepsilon)\delta^{-1/2},
\end{align}
and
\begin{align}\label{EQUA002}
p-(q)_{\delta;x'}=&\sum^{\frac{n(n+1)}{2}}_{\alpha=1}(C_{1}^{\alpha}-C_{2}^{\alpha})(\bar{p}_{1}^{\alpha}+q_{1}^{\alpha}-(q_{1}^{\alpha})_{\delta;x'})+O(1)\delta^{-\frac{n}{2}}e^{-\frac{1}{2C\delta^{1/2}}}\notag\\
=&\sum^{n}_{\alpha=1}\frac{\det\mathbb{A}_{1}^{\ast\alpha}}{\det \mathbb{A}_{0}^{\ast}}\frac{\rho_{n}^{-1}(\varepsilon)}{\mathcal{L}_{\alpha}(\frac{\pi}{2\kappa})^{\frac{n-1}{2}}}\frac{1+O(r_{\varepsilon})}{1+\widetilde{\mathcal{G}}_{n}^{\ast\alpha}\rho^{-1}_{n}(\varepsilon)}(\bar{p}_{1}^{\alpha}+O(\tilde{r}_{\alpha}(\delta)))\notag\\
&+\sum^{\frac{n(n+1)}{2}}_{\alpha=n+1}\frac{\det\mathbb{A}_{2}^{\ast\alpha}}{\det \mathbb{A}^{\ast}_{0}}\big(1+O(r_{\varepsilon})\big)(\bar{p}_{1}^{\alpha}+O(\tilde{r}_{\alpha}(\delta)))+O(1)\delta^{-\frac{n}{2}}e^{-\frac{1}{2C\delta^{1/2}}}\notag\\
=&\sum^{n}_{\alpha=1}\frac{\det\mathbb{A}_{1}^{\ast\alpha}}{\det \mathbb{A}_{0}^{\ast}}\frac{\rho_{n}^{-1}(\varepsilon)}{\mathcal{L}_{\alpha}(\frac{\pi}{2\kappa})^{\frac{n-1}{2}}}\frac{1+O(r_{\varepsilon})}{1+\widetilde{\mathcal{G}}_{n}^{\ast\alpha}\rho^{-1}_{n}(\varepsilon)}\bar{p}_{1}^{\alpha}\notag\\
&+\sum^{\frac{n(n+1)}{2}}_{\alpha=n+1}\frac{\det\mathbb{A}_{2}^{\ast\alpha}}{\det \mathbb{A}^{\ast}_{0}}\big(1+O(r_{\varepsilon})\big)\bar{p}_{1}^{\alpha}+O(1)\rho^{-1}_{n}(\varepsilon)\delta^{-1/2};
\end{align}

$(ii)$ in the case of $n>3$,
\begin{align}\label{EQUA003}
\nabla\mathbf{u}=&\sum^{\frac{n(n+1)}{2}}_{\alpha=1}\frac{\det\mathbb{A}_{3}^{\ast\alpha}}{\det \mathbb{A}^{\ast}}(1+O(\varepsilon^{\min\{\frac{1}{12},\frac{n-3}{24}\}}))(\nabla\bar{\mathbf{u}}_{1}^{\alpha}+O(\tilde{r}_{\alpha}(\delta)))+O(1)\delta^{-\frac{n}{2}}e^{-\frac{1}{2C\delta^{1/2}}}\notag\\
=&\sum^{\frac{n(n+1)}{2}}_{\alpha=1}\frac{\det\mathbb{A}_{3}^{\ast\alpha}}{\det \mathbb{A}^{\ast}}(1+O(\varepsilon^{\min\{\frac{1}{12},\frac{n-3}{24}\}}))\nabla\bar{\mathbf{u}}_{1}^{\alpha}+O(1)\delta^{-1/2},
\end{align}
and
\begin{align}\label{EQUA005}
p-(q)_{\delta;x'}=&\sum^{\frac{n(n+1)}{2}}_{\alpha=1}\frac{\det\mathbb{A}_{3}^{\ast\alpha}}{\det \mathbb{A}^{\ast}}(1+O(\varepsilon^{\min\{\frac{1}{12},\frac{n-3}{24}\}}))(\bar{p}_{1}^{\alpha}+O(\tilde{r}_{\alpha}(\delta)))\notag\\
=&\sum^{\frac{n(n+1)}{2}}_{\alpha=1}\frac{\det\mathbb{A}_{3}^{\ast\alpha}}{\det \mathbb{A}^{\ast}}(1+O(\varepsilon^{\min\{\frac{1}{12},\frac{n-3}{24}\}}))\bar{p}_{1}^{\alpha}+O(1)\delta^{-1/2}.
\end{align}
Consequently, combining these asymptotics above, we obtain that Theorem \ref{MAINZW002} holds.

\end{proof}

For the convenience of presentation, in the following $a\lesssim b$ (or $a\gtrsim b$) implies $a\leq Cb$ (or $a\geq \frac{1}{C}b$) for some positive $\varepsilon$-independent constant $C$, which depends only on $\mu,n,R_{0},\kappa$, and the upper bounds of the $C^{2,\gamma}$ norm of $\partial D$ and $\partial D_{i}$, $i=1,2$.

\begin{corollary}\label{MGA001}
Assume as in Theorem \ref{MAINZW002}. If $\det\mathbb{A}_{i}^{\ast n}\neq0$, $i=1,3$, then for a arbitrarily small $\varepsilon>0$ and $x\in\{|x'|=0\}\cap\Omega$,
\begin{align*}
|\sigma[\mathbf{u},p-(q)_{\delta;x'}]|\lesssim&
\begin{cases}
\frac{\max\limits_{1\lesssim\alpha\lesssim n}\kappa^{\frac{n-1}{2}}|\det\mathbb{A}_{1}^{\ast\alpha}|}{\mu|\det\mathbb{A}_{0}^{\ast}|}\frac{1}{\varepsilon^{2}\rho_{n}(\varepsilon)},&n=2,3,\vspace{0.5ex}\\
\frac{\max\limits_{1\leq\alpha\leq n}|\det\mathbb{A}_{3}^{\ast\alpha}|}{|\det \mathbb{A}^{\ast}|}\frac{1}{\varepsilon^{2}},&n>3,
\end{cases}
\end{align*}
and
\begin{align*}
|\sigma[\mathbf{u},p-(q)_{\delta;x'}]|\gtrsim&
\begin{cases}
\frac{\kappa^{\frac{n-1}{2}}|\det\mathbb{A}_{1}^{\ast n}|}{\mu|\det\mathbb{A}_{0}^{\ast}|}\frac{1}{\varepsilon^{2}\rho_{n}(\varepsilon)},&n=2,3,\vspace{0.5ex}\\
\frac{|\det\mathbb{A}_{3}^{\ast n}|}{|\det \mathbb{A}^{\ast}|}\frac{1}{\varepsilon^{2}},&n>3.
\end{cases}
\end{align*}

\end{corollary}

\begin{remark}
The optimal gradient estimates in Corollary \ref{MGA001} are shown more precisely than that of \cite{LX2022} due to the captured blow-up matrices in any dimension.
\end{remark}

\begin{proof}
According to Remark \ref{ANALY001} and \eqref{EQUA001}--\eqref{EQUA005}, it follows that for $x\in\{x'=0'\}\cap\Omega$,
\begin{align*}
|\nabla \mathbf{u}|\lesssim&\sum^{n}_{\alpha=1}\frac{|\det\mathbb{A}_{1}^{\ast\alpha}|}{|\det \mathbb{A}_{0}^{\ast}|}\frac{\kappa^{\frac{n-1}{2}}}{\mu\rho_{n}(\varepsilon)}|\nabla\bar{\mathbf{u}}_{1}^{\alpha}|+
\rho^{-1}_{n}(\varepsilon)\varepsilon^{-1/2}\notag\\
\lesssim&
\begin{cases}
\frac{\max\limits_{1\leq\alpha\leq n}\kappa^{\frac{n-1}{2}}|\det\mathbb{A}_{1}^{\ast\alpha}|}{\mu|\det\mathbb{A}_{0}^{\ast}|}\frac{1}{\varepsilon\rho_{n}(\varepsilon)},&n=2,3,\vspace{0.5ex}\\
\frac{\max\limits_{1\leq\alpha\leq n}|\det\mathbb{A}_{3}^{\ast\alpha}|}{|\det \mathbb{A}^{\ast}|}\frac{1}{\varepsilon},&n>3,
\end{cases}\notag\\
|p-(q)_{\delta;x'}|\lesssim&\sum^{n}_{\alpha=1}\frac{|\det\mathbb{A}_{1}^{\ast\alpha}|}{|\det \mathbb{A}_{0}^{\ast}|}\frac{\kappa^{\frac{n-1}{2}}}{\mu\rho_{n}(\varepsilon)}|\bar{p}_{1}^{\alpha}|+
\rho^{-1}_{n}(\varepsilon)\varepsilon^{-1/2}\notag\\
\lesssim&
\begin{cases}
\frac{\max\limits_{1\leq\alpha\leq n}\kappa^{\frac{n-1}{2}}|\det\mathbb{A}_{1}^{\ast\alpha}|}{\mu|\det\mathbb{A}_{0}^{\ast}|}\frac{1}{\varepsilon^{2}\rho_{n}(\varepsilon)},&n=2,3,\vspace{0.5ex}\\
\frac{\max\limits_{1\leq\alpha\leq n}|\det\mathbb{A}_{3}^{\ast\alpha}|}{|\det \mathbb{A}^{\ast}|}\frac{1}{\varepsilon^{2}},&n>3,
\end{cases}
\end{align*}
and
\begin{align*}
|\nabla \mathbf{u}|\geq&\left|\sum^{n}_{\alpha=1}\frac{\det\mathbb{A}_{1}^{\ast\alpha}}{\det \mathbb{A}_{0}^{\ast}}\frac{\rho_{n}^{-1}(\varepsilon)}{\mathcal{L}_{\alpha}(\frac{\pi}{2\kappa})^{\frac{n-1}{2}}}\frac{1+O(r_{\varepsilon})}{1+\widetilde{\mathcal{G}}_{n}^{\ast\alpha}\rho^{-1}_{n}(\varepsilon)}\partial_{x_{n}}(\bar{\mathbf{u}}_{1}^{\alpha})^{(n)}\right|-C\rho^{-1}_{n}(\varepsilon)\varepsilon^{-1/2}\notag\\
\gtrsim&
\begin{cases}
\frac{\kappa^{\frac{n-1}{2}}|\det\mathbb{A}_{1}^{\ast n}|}{\mu|\det\mathbb{A}_{0}^{\ast}|}\frac{1}{\varepsilon\rho_{n}(\varepsilon)},&n=2,3,\vspace{0.5ex}\\
\frac{|\det\mathbb{A}_{3}^{\ast n}|}{|\det \mathbb{F}^{\ast}|}\frac{1}{\varepsilon},&n>3,
\end{cases}\notag\\
|p-(q)_{\delta;x'}|\geq&\left|\frac{\det\mathbb{A}_{1}^{\ast n}}{\det \mathbb{A}_{0}^{\ast}}\frac{\rho_{n}^{-1}(\varepsilon)}{\mathcal{L}_{\alpha}(\frac{\pi}{2\kappa})^{\frac{n-1}{2}}}\frac{1+O(r_{\varepsilon})}{1+\widetilde{\mathcal{G}}_{n}^{\ast n}\rho^{-1}_{n}(\varepsilon)}\bar{p}_{1}^{n}\right|\notag\\
&-\left|\sum^{n-1}_{\alpha=1}\frac{\det\mathbb{A}_{1}^{\ast\alpha}}{\det \mathbb{A}_{0}^{\ast}}\frac{\rho_{n}^{-1}(\varepsilon)}{\mathcal{L}_{\alpha}(\frac{\pi}{2\kappa})^{\frac{n-1}{2}}}\frac{1+O(r_{\varepsilon})}{1+\widetilde{\mathcal{G}}_{n}^{\ast\alpha}\rho^{-1}_{n}(\varepsilon)}\bar{p}_{1}^{\alpha}\right|-C\rho^{-1}_{n}(\varepsilon)\varepsilon^{-1/2}\notag\\
\gtrsim&
\begin{cases}
\frac{\kappa^{\frac{n-1}{2}}|\det\mathbb{A}_{1}^{\ast n}|}{\mu|\det\mathbb{A}_{0}^{\ast}|}\frac{1}{\varepsilon^{2}\rho_{n}(\varepsilon)},&n=2,3,\vspace{0.5ex}\\
\frac{|\det\mathbb{A}_{3}^{\ast n}|}{|\det \mathbb{F}^{\ast}|}\frac{1}{\varepsilon^{2}},&n>3.
\end{cases}
\end{align*}
Since
\begin{align*}
\sqrt{n}|p-(q)_{\delta;x'}|-2\mu|\nabla\mathbf{u}|\leq|\sigma[\mathbf{u},p-(q)_{\delta;x'}]|\leq\sqrt{n}|p-(q)_{\delta;x'}|-2\mu|\nabla\mathbf{u}|,
\end{align*}
then we have
\begin{align*}
\begin{cases}
\frac{\kappa^{\frac{n-1}{2}}|\det\mathbb{A}_{1}^{\ast n}|}{\mu|\det\mathbb{A}_{0}^{\ast}|}\frac{1}{\varepsilon^{2}\rho_{n}(\varepsilon)}\lesssim|\sigma[\mathbf{u},p-(q)_{\delta;x'}]|\lesssim\frac{\max\limits_{1\lesssim\alpha\lesssim n}\kappa^{\frac{n-1}{2}}|\det\mathbb{A}_{1}^{\ast\alpha}|}{\mu|\det\mathbb{A}_{0}^{\ast}|}\frac{1}{\varepsilon^{2}\rho_{n}(\varepsilon)},&n=2,3,\vspace{0.5ex}\\
\frac{|\det\mathbb{A}_{3}^{\ast n}|}{|\det \mathbb{A}^{\ast}|}\frac{1}{\varepsilon^{2}}\lesssim|\sigma[\mathbf{u},p-(q)_{\delta;x'}]|\lesssim\frac{\max\limits_{1\leq\alpha\leq n}|\det\mathbb{A}_{3}^{\ast\alpha}|}{|\det \mathbb{A}^{\ast}|}\frac{1}{\varepsilon^{2}},&n>3.
\end{cases}
\end{align*}

\end{proof}


\noindent{\bf{\large Acknowledgements.}} Z. Zhao was partially supported by CPSF (2021M700358). Z. Zhao would like to thank his friend L.J. Xu for her useful discussions on constructions of leading terms.

\end{document}